\documentclass{article}
\usepackage[utf8]{inputenc}
\usepackage{enumerate}
\usepackage{amssymb, bm}
\usepackage{amsthm}
\usepackage{amsmath}
\usepackage{bbm}
 \usepackage[mathscr]{euscript}
 \usepackage[all]{xy}
\title{Strongly pseudo-effective and numerically flat reflexive sheaves}
\author{Xiaojun WU}
\date{\today}
\newtheorem{mythm}{Theorem}
\newtheorem{mylem}{Lemma}
\newtheorem{myprop}{Proposition}
\newtheorem{myex}{Example}
\newtheorem{mycor}{Corollary}
\newtheorem{mydef}{Definition}
\newtheorem{myrem}{Remark}
\newcommand{\Addresses}{{
  \bigskip
  \footnotesize

Xiaojun Wu, \textsc{Institut Fourier, Univerisité Grenoble Alpes,
38400 Saint-Martin-d'Hères}\par\nopagebreak
  \textit{E-mail address},  \texttt{xiaojun.wu@univ-grenoble-alpes.fr}

}}
\newcommand{\Subjclass}{{
  \bigskip
  \footnotesize
Primary 32Q15; Secondary 32Q57

}}
\begin{document}
\def\cI{\mathcal{I}}
\def\Z{\mathbb{Z}}
\def\Q{\mathbb{Q}}  \def\C{\mathbb{C}}
 \def\R{\mathbb{R}}
 \def\N{\mathbb{N}}
 \def\H{\mathbb{H}}
  \def\P{\mathbb{P}}
 \def\rC{\mathcal{C}}
  \def\tors{\mathrm{Tors}}
  \def\d{\partial}
 \def\dbar{{\overline{\partial}}}
\def\dzbar{{\overline{dz}}}
 \def\ii{\mathrm{i}}
  \def\d{\partial}
 \def\dbar{{\overline{\partial}}}
\def\dzbar{{\overline{dz}}}
\def \ddbar {\partial \overline{\partial}}
\def\cN{\mathcal{N}}
\def\cE{\mathcal{E}}  \def\cO{\mathcal{O}}
\def\cF{\mathcal{F}}
\def\cS{\mathcal{S}}
\def\cQ{\mathcal{Q}}
\def\P{\mathbb{P}}
\def\cI{\mathcal{I}}
\def \loc{\mathrm{loc}}
\def \cC{\mathcal{C}}
\bibliographystyle{plain}
\def \dim{\mathrm{dim}}
\def \Sing{\mathrm{Sing}}
\def \Id{\mathrm{Id}}
\def \rank{\mathrm{rank}}
\def \tr{\mathrm{tr}}
\def \ch{\mathrm{ch}}
\def \td{\mathrm{Td}}
\def \Ric{\mathrm{Ric}}
\def \Vol{\mathrm{Vol}}
\def \Gr{\mathrm{Gr}}
\def \RHS{\mathrm{RHS}}
\def \liminf{\mathrm{liminf}}
\def \ker{\mathrm{Ker}}
\def \nn{\mathrm{nn}}

\def\stateparagraph{\vskip7pt plus 2pt minus 1pt\noindent}
\maketitle

\Addresses
\Subjclass
\begin{abstract}
In this paper, we discuss the concept of strongly pseudoeffective vector bundle and also introduce strongly pseudoeffective torsion-free sheaves over compact K\"ahler manifolds. We show that a strongly pseudoeffective reflexive sheaf over a compact K\"ahler manifold with vanishing first Chern class is in fact a numerically flat vector bundle. A proof is obtained through a natural construction of positive currents representing the Segre classes of strongly pseudoeffective vector bundles.
\end{abstract}
\section{Introduction}
The concept of numerical flatness introduced in \cite{DPS94} proved itself to be instrumental in the study and classification theory of compact K\"ahler manifolds with nef anticanonical bundles. It has been studied by many authors and in many works, cf. \cite{Cao18}, \cite{Cao19}, \cite{CH17}, \cite{CH19}, \cite{CCM}, \cite{CP17}, \cite{HIM},  \cite{HPS}, \cite{Wang19} among others.

Recall that a holomorphic vector bundle $E$ is called numerically flat if both $E$ and $E^*$ are nef (equivalently if $E$ and $(\det E)^{-1}$ are nef).
In fact, the condition of being numerically flat yields strong restrictions for the curvature of the corresponding vector bundle. Actually,
in \cite{DPS94}, Demailly, Peternell and Schneider proved that a numerically flat bundle $E$ on a compact Kähler manifold $X$ admits a filtration by vector bundles whose graded pieces are Hermitian flat. In some sense, numerical flatness is the algebraic analogue of metric flatness.

In \cite{CCM} and \cite{HIM}, the authors consider the following question.
If a strongly pseudo-effective vector bundle over a projective manifold has a vanishing first Chern class, is this vector bundle numerically flat?
Since a vector bundle $E$ is numerically flat if and only if $E$ and $\det(E)^{-1}$ are nef, the question amounts to ask whether the vector bundle is in fact nef. 

Intuitively, a positive singular metric on the vector bundle $E$ would induce a positive singular metric on the determinant $\det(E)$. 
But since the first Chern class of $E$ (i.e.\ the Chern class of $\det(E)$) is trivial, any metric with (semi)positive curvature must be flat and thus cannot possess any singularity. This implies that the given positive singular metric on $E$ has to be smooth as well.

From this point of view, the same property should hold on an arbitrary compact K\"ahler manifold, and not just on projective manifolds, since all properties under consideration are independent of the projectivity condition. One of the goals of this work is to confirm this philosophy. Namely, we prove the following
\stateparagraph
{\bf Main Theorem.} {\it Let $E$ be a strongly psef vector bundle over a compact K\"ahler manifold $(X, \omega)$ with $c_1(E)=0$. 
 Then $E$ is a nef vector bundle.}
\paragraph{}
The main technical tool is the construction of Segre currents. More precisely, we define a Segre $(k,k)$-closed positive current as the direct image of the wedge product of the curvature current of $\cO_{\P(E)}(1)$, as soon as we have an appropriate codimension condition on the singular locus of the metric.
\stateparagraph
{\bf Main technical lemma.} {\it
Let $E$ be a strongly psef vector bundle of rank $r$ over a compact K\"ahler manifold $(X, \omega)$. 
Let $(\cO_{\P(E)}(1),h_\varepsilon)$ be singular metric with analytic singularities such that 
$$i \Theta(\cO_{\P(E)}(1),h_\varepsilon) \geq - \varepsilon \pi^* \omega$$
and the codimension of $\pi(\Sing (h_\varepsilon))$ is at least $k$ in $X$.
Then there exists a $(k,k)$-positive current in the class $\pi_* (c_1(\cO_{\P(E)}(1))+\varepsilon \pi^* \{\omega\})^{r+k-1}$.}
\paragraph{}
In fact, the construction in our Main technical lemma would work for a broader situation as stated in the following theorem.
\stateparagraph
{\bf Theorem A.} {\it
Let $\pi: X \to Y$ be a submersion between compact K\"ahler manifolds  of relative dimension $r-1$. 
Let $T$ be a closed positive $(1,1)-$current in the cohomology class $\{ \alpha \} \in H^{1,1}(X, \R)$
such that $T$ has analytic singularities and is smooth on $X \setminus \pi^{-1}(Z)$ with $Z$ a closed analytic set of codimension at least $k$.
Assume that for any $y \in Y$, there exist an open neighborhood $U$ of $y$ and a quasi-psh function $\psi$ on $X$ such that
$\alpha + i \d \dbar \psi \geq 0$
in the sense of currents on
$\pi^{-1}(U)$ and $\psi$ is smooth outside a closed analytic set of codimension at least $k+r$.
Then there exists a closed positive current in the cohomology class $\pi_* \alpha^{r+k-1}$.
}
\paragraph{}
The strategy of the proof of the Main theorem is as follows.
We show that the Lelong numbers of the corresponding Segre current control the Lelong numbers of the weight functions of the singular metrics prescribed in the definition of a strongly pseudoeffective vector bundle. Then, we observe that the Lelong numbers of Segre currents must tend to 0 in the limit, as the unique (semi)positive current in $c_1(E)$ is the zero current.
Thus the Lelong numbers of the weight functions uniformally tend to 0 as the Lelong numbers of the Segre currents.
By Demailly's regularisation theorem, the weight functions of the metrics can be regularised, thus the vector bundle is actually nef. 

In fact, we can expect an even stronger property.
Since $E$ is strongly psef, the class $c_1(\cO_{\P(E)}(1))$ is psef.
Intuitively, $c_1(\cO_{\P(E)}(1))$ contains a not too singular current (in the sense that the projection of the singular part onto $X$ is contained in some analytic subset of codimension at least 1).
Thus the wedge powers of appropriate exponents of this current in the first Chern class are defined and positive, as well as their direct images under $\pi: \P(E) \to X$.
In particular, if $r$ is the rank of $E$, we can hope that the second Segre class $\pi_* (c_1(\cO_{\P(E)}(1)))^{r+1}$ is positive (by this, we mean that its cohomology class contains a positive current) 

Remind that the second Segre class is equal to $c_1(E)^2-c_2(E)$. By the Bogomolov inequality if $E$ is semistable, when $c_1(E)=0$, the integration of $c_2(E) \wedge \omega^{n-2}$ on $X$ is positive where $\omega$ is a Kähler form on $X$ and $n$ is the dimension of $X$.
Comparing these two facts, one knows that $c_2(E)=0$ and the Bogomolov inequality is in fact an equality which implies that $E$ is in fact flat.

For a reflexive sheaf $\cF$, the Chern classes can be defined as follows.
Let $\sigma$ be any modification such that $\sigma^* \cF/\tors$ is a vector bundle.
The existence of such modification is provided by the fundamental work of \cite{Ros}, \cite{GR} and \cite{Rie}.
Then for $i=1,2$, $c_i(\cF)=\sigma_* c_i(\sigma^* \cF/\tors)$ which is independent of the choice of modification $\sigma$. The rough idea is that 
the above consideration should hold on some birational model of $X$
and we conclude by the work of \cite{DPS94} instead of using the Bogomolov inequality.


In order to study the positivity of torsion free coherent sheaves, it is useful to define in full generality the nef (or strongly psef) property for such sheaves.
\stateparagraph
{\bf Definition.} {\it
A torsion free coherent sheaf $\cF$ over a compact complex manifold is called nef (resp. strongly psef) if there exists some modification $\sigma:\tilde{X} \to X$ such that $\sigma^* \cF/\tors$ is a nef (resp. strongly psef) vector bundle.}
\paragraph{}
The above considerations let us hope the stronger fact that over every compact K\"ahler manifold $(X, \omega)$, a strongly psef reflexive sheaf with trivial first Chern class is in fact a nef vector bundle. 
In Section 5, we prove that this is actually the case.
A difficulty of the above approach is that in general a wedge product of positive currents is not necessarily well defined. Instead of proceeding directly, we first prove the following result.
\stateparagraph
{\bf Lemma.} {\it
Let $\cF$ be a nef reflexive sheaf over a compact K\"ahler manifold $(X, \omega)$ with $c_1(\cF)=0$. 
 Then $\cF$ is a nef vector bundle.}
\paragraph{}
Now combining the main theorem, we can conclude that
\stateparagraph
{\bf Corollary.} {\it
Let $\cF$ be a strongly psef reflexive sheaf over a compact K\"ahler manifold $(X, \omega)$ with $c_1(\cF)=0$. 
 Then $\cF$ is a nef vector bundle.}
\paragraph{}
Note that in the approach with the Bogomolov inequality, we have to take wedge products that are well defined without imposing any restriction on the codimension of singular part of the metric. In this situation, for a strongly psef vector bundle $E$, we can find a positive current in $c_1(E)$ but not necessarily in $c_2(E)$.

At the end of the paper, as a geometric application, we show that an irreducible symplectic, or Calabi-Yau manifold does not have a strongly psef tangent bundle or cotangent bundle. This generalises the work of \cite{DPS94} and \cite{Nak} that only applied to the projective setting.
In the singular and projective setting, the strongly result is proven in Theorem 1.6 of \cite{HP} and Corollary 6.5 \cite{Dru} for threefolds. (They prove that in this case $\cO_{\P(E)}(1)$ is not a psef line bundle where $E$ is the tangent bundle or the cotangent bundle.)

We also generalise the main results to the $\Q-$twisted case analogous to the result of \cite{LOY20} in the compact K\"ahler setting.

The organisation of this paper is as follows. In Section 2, the concept of strongly psef vector bundles is discussed.
We give a definition of strongly psef vector bundle of the Kähler version essentially equivalent to the one proposed in \cite{BDPP}. By this equivalent condition, we can show that some usual algebraic operations can still be taken for strongly psef vector bundles. 
For example, the direct sum or tensor product of strongly psef vector bundles is still strongly psef. In Section 3, 
we investigate the concept of nef/strongly psef torsion free coherent sheaves and algebraic operations of these sheaves.
Then we show that a numerically flat reflexive sheaf on an arbitrary compact K\"ahler manifold is in fact a vector bundle. This result can also be generalised to strongly pseudoeffective (strongly psef) reflexive sheaves $\cF$ such that $c_1(\det \cF)=0$ in Section 5. 
In Section 4, we make a digression to introduce the definition of Segre forms (or Segre currents), as a tool to treat the strongly psef case. It~should be observed that a similar construction has been done in \cite{LRRR}.

In this note, all manifolds are supposed to be compact without any explicit mention.

\textbf{Acknowledgement} I thank Jean-Pierre Demailly, my PhD supervisor, for his guidance, patience and generosity. 
I would like to thank Junyan Cao, S\'ebastien Boucksom, Simone Diverio, Andreas H\"oring and Richard Lärkäng for some very useful suggestions on the previous draft of this work.
I would also like to express my gratitude to colleagues of Institut Fourier for all the interesting discussions we had. This work is supported by the PhD program AMX of \'Ecole Polytechnique and Ministère de l'Enseignement Supérieur et de la Recherche et de l’Innovation, and the European Research Council grant ALKAGE number 670846 managed by J.-P. Demailly.
We thank the anonymous reviewer for a very careful reading of this paper, and for insightful comments and suggestions.
\section{Strongly pseudoeffective vector bundles}
The following definition of a strongly psef vector bundle is a reformulation of the definition of \cite{BDPP} (Definition 7.1).
\begin{mydef}
Let $(X, \omega)$ be a compact K\"ahler manifold and $E$ a holomorphic vector bundle on $X$. 
Then $E$ is said to be strongly pseudo-effective (strongly psef for short) if the line bundle $\cO_{\P(E)}(1)$ is pseudo-effective on the projectivized bundle $\P(E)$ of hyperplanes of $E$, i.e.\ if for every $\varepsilon>0$ there exists a singular metric $h_\varepsilon$ with analytic singularities on $\cO_{\P(E)}(1)$ and a curvature current $i \Theta(h_\varepsilon) \geq - \varepsilon \pi^* \omega$, and if the projection $\pi(\Sing(h_\varepsilon))$ of the singular set of~$h_\varepsilon$ is not equal to~$X$.
\end{mydef}
One can observe that in \cite{BDPP} the definition is expressed rather in terms of the non-nef locus. 
\begin{mydef}{(\rm\cite{DPS01})}
Let $\varphi_1,\varphi_2$ be two quasi-psh functions on $X$ $($i.e. $i \d \dbar \varphi_i \geq -C \omega$ in the sense of currents for some $C \geq 0)$. Then, $\varphi_1$ is said to be less singular than $\varphi_2$ $($we write $\varphi_1 \preceq \varphi_2)$ if we have $\varphi_2 \leq \varphi_1+C_1$ for some constant $C_1$.
Let $\alpha$ be a psef class in $H^{1,1}_{BC}(X,\R)$ and $\gamma$ be a smooth real $(1,1)$-form.
Let $T_1,T_2,\theta \in \alpha$
with $\theta$ smooth and $T_i=\theta+i\d \dbar \varphi_i$ $(i=1,2)$, the
potential $\varphi_i$ being defined up to a constant since $X$ is compact.
We say that $T_1 \preceq T_2$, resp.\ singularity equivalent $T_1\sim T_2$,
if $\varphi_1 \preceq \varphi_2$, resp.\ if 
$\varphi_1 \preceq \varphi_2$ and $\varphi_2 \preceq \varphi_1$.

A minimal element $T_{\min,\gamma}$ with respect to the pre-order relation $\preceq$ always exists. Such an element can be obtained by taking the upper semi-continuous upper envelope of all $\varphi_{i}$ such that $\theta +i \d \dbar \varphi_i \geq \gamma$ and $\sup_X \varphi_i=0$. It is unique up to equivalence of
singularities.
\end{mydef}
\begin{mydef}(Non-nef   locus)

The  non-nef  locus  of  a  pseudo-effective  class $\alpha \in H^{1,1}_{BC}(X,\R)$ is defined to be
$$E_{\nn}(\alpha) :=\bigcup_{\varepsilon >0}~\bigcup_{c >0} E_c(T_{\min, -\varepsilon \omega})$$
where $\omega$ is any Hermitian metric.
\end{mydef}
Let us observe that we can replace $\pi^* \omega$ by any smooth K\"ahler form $\tilde{\omega}$ on $\P(E)$ in the definition of a strongly psef vector bundle.
The reason is as follows.
On the one hand, $\pi^* \omega \leq C \tilde{\omega}$ for some $C >0$ since $X$ is compact.
Thus, $i \Theta(h_\varepsilon) \geq - \varepsilon \pi^* \omega$ implies that
$i \Theta(h_\varepsilon) \geq -C \varepsilon \tilde{ \omega}$.
On the other hand,
since $\cO_{\P(E)}(1)$ is relatively $\pi$-ample, we have
$\varepsilon_0 i \Theta_{h_0}(\cO_{\P(E)}(1)) +\pi^* \omega\geq \varepsilon_1\tilde\omega$ for any given smooth Hermitian metric $h_0$ on $E$, if $0<\varepsilon_1\ll\varepsilon_0\ll 1$ are small enough. Assuming that there exists a singular metric $h_\varepsilon$ on $\cO_{\P(E)}(1)$ such that $i \Theta_{h_\varepsilon}(\cO_{\P(E)}(1)) \geq - \varepsilon\tilde\omega$, we infer that the metric $h'_\varepsilon=h_0^{\varepsilon/\varepsilon_1}h_\varepsilon^{1-\varepsilon/\varepsilon_1}$
has a curvature lower bound 
$$i\Theta_{h'_\varepsilon}(\cO_{\P(E)}(1))\geq
\frac{\varepsilon}{\varepsilon_1}\big(\varepsilon_1\tilde\omega-\pi^*\omega\big)
-\Big(1-\frac{\varepsilon}{\varepsilon_1}\Big)\varepsilon\tilde\omega\geq
-\frac{\varepsilon}{\varepsilon_1}\pi^*\omega.
$$
In \cite{BDPP}, a holomorphic vector bundle $E$ was defined to be strongly pseudo-effective if the line bundle $\cO_{\P(E)}(1)$ is pseudo-effective on the projectivized bundle $\P(E)$
of hyperplanes of $E$, and if the projection $\pi(E_{\nn}(\cO_{\P(E)}(1)))$ of the non-nef locus of $\cO_{\P(E)}(1)$ onto $X$ does not cover all of $X$.
By definition, 
$$E_{\nn}(c_1(\cO_{\P(E)}(1))) \subset \bigcup_{\varepsilon >0} \Sing(T_{\min,-\varepsilon\tilde\omega})\subset \bigcup_{\varepsilon >0} \Sing(h_\varepsilon).$$ 
Hence a strongly psef vector bundle defined in Definition 1 is strongly psef under the definition of  \cite{BDPP}.
On the other hand, 
by the regularization theorem,  we can construct from $T_{\min, -\varepsilon \tilde{\omega}}$ a metric $h_{2\varepsilon}$ on $\cO_{\P(E)}(1)$ with
$i \Theta(h_{2\varepsilon}) \geq -2 \varepsilon \tilde{\omega}$.
By definition, $\Sing(h_{2\varepsilon}) \subset \bigcup_{c >0} E_c(T_{\min, -2\varepsilon \tilde{\omega}})$ thus it does not project onto $X$.
Hence our definition is equivalent to the definition of \cite{BDPP}.

We remark that the definition of strongly psef vector bundle we used is stronger than the widely used weak definition. 
A vector bundle $E$ is called psef in the weak sense if $\cO_{\P(E)}(1)$ is a psef line bundle over $\P(E)$.
Of course, our definition of strongly psef vector bundle coincide with the widely used weak definition in the case of line bundle.
However, this weak definition is too weak to give a classification even if we pose some strong topological obstruction like with vanishing first Chern class.
For example, if $X$ is a projective manifold and $A$ is an ample line bundle over $X$, for any $p \neq 0$, $A^p \oplus (A^p)^{*}$ is a psef vector bundle in the weak sense with vanishing first Chern class.
Intuitively, a psef vector bundle can have negative curvature in some direction which is not enough for our propose to construct some positive current in the first Chern class of the determinant bundle.

It should be noticed that pseudo-effectiveness in the weak sense is a Zariski closed condition while strong pseudo-effectiveness is not Zariski closed. 
More precisely, let $p:\mathfrak{X} \to \Delta$ be a proper holomorphic submersion which defines a family of compact K\"ahler manifolds over the unit disc $\Delta$ and $E$ be a holomorphic vector bundle over $\mathfrak{X}$.
Then the set $t \in \Delta$ such that the restriction $E|_{X_t}$ is a psef vector bundle over $X_t$ in the weak sense is a Zariski closed set where $X_t:=p^{-1}(t)$.
A complete proof can be found e.g. in the appendix of \cite{AH19} by Simone Diverio.
However, the same does not hold for strong pseudo-effectiveness.
For example, we can take the following example indicated to the author by Jean-Pierre Demailly.
\begin{myex}{\rm (Theorem 2.2.5 \cite{OSS80})

Let $x_1, \cdots, x_m$ be the points of the projective plane $\P^2$.
There is a holomorphic rank 2 bundle $E$ over $\P^2$ whose restriction to any line $L$, on which exactly a points of the set $\{x_1,\cdots, x_m \}$ lie, splits in the form
$$E|_L =\cO_L(a) \oplus \cO_L(-a).$$
The generic splitting type of this bundle is $(0,0)$.

The construction of the vector bundle is as follows.
Let $\sigma:Y \to \P^2$ be the blow up of $\P^2$ over $\{x_1,\cdots, x_m \}$ with exceptional divisor $C=\sum_{i=1}^m C_i$.
Let $E'$ be a rank two vector bundle over $Y$ such that it satisfies the extension
$$0 \to \cO_Y(C) \to E' \to \cO_Y(-C) \to 0$$
and its restriction to each $C_i$ satisfies the Euler sequence
$$0 \to \cO_{C_i}(-1) \to E'|_{C_i} \cong \cO_{C_i}^{\oplus 2} \to \cO_{C_i}(1) \to 0.$$
It can be proved that $E'$ is the pull back of some vector bundle $E$ over $\P^2$.
We have the short exact sequence
$$0 \to \cO_{\tilde{L}}(a) \to E'|_{\tilde{L}} \to \cO_{\tilde{L}}(-a) \to 0.$$
where $a$ is the number of $\{x_1,\cdots, x_m \}$ which lie in $L$.
The short exact sequence splits since $H^1(\tilde{L},\cO_{\tilde{L}}(2a))=0.$
The blow up induces a biholomorphism between the strict transform of a line $\tilde{L}$ to $L$ which gives the conclusion.

Thus we can construct a family of vector bundles whose restriction to some special fibers is not strongly psef   although the restriction to the general fiber is strongly psef (in fact trivial).
The lines in the projective plane form a family of $\P^1$ over the Grassmannian $Gr(2, 3)$.
The total space $\mathfrak{X}$ is a closed submanifold of $\P^2 \times Gr(2,3)$.
Consider the vector bundle which is the restriction over $\mathfrak{X}$ of the pull back of the previous constructed bundle under $p_1: \P^2 \times Gr(2,3) \to \P^2 $.
}
\end{myex} 

A related definitions in the projective case is also widely used in the literature,
which is weak positivity in the sense of Nakayama (cf. eg. \cite{Nak} Definition 3.20).
A torsion free coherent sheaf $\cF$ is weakly positive  at $x \in X$ a projective manifold if, for any $a \in \N^*$ and for any ample line bundle $A$ on $X$, there exists $b \in \N^*$ such that
$(\mathrm{Sym}^{ab} \cF)^{\lor \lor} \otimes A^b$
is globally generated at $x$, where $(\mathrm{Sym}^{ab} \cF)^{\lor \lor}$ is the double
dual of $ab$-th symmetric power of $\cF$. 
A torsion free coherent sheaf is called weak positive in the sense of Nakayama if it is weak positive at some point.
It is proven in Proposition 7.2 \cite{BDPP} that for a vector bundle $E$ over a projective manifold $X$, $E$ is psef in our strong sense if and only if $E$ is weak positive in the sense of Nakayama.

Now we give still another equivalent definition of a strongly psef vector bundle.
The argument is analogous to the one of  \cite[Theorem 4.1]{Dem92} in the singular setting.
Intuitively, being strongly psef is equivalent to the existence of ``algebraic" approximation currents.
Here ``algebraic" means that the approximation can be obtained from the sections of higher degree tensor product of the vector bundle.
(Of course the sections are local since the global sections on $X$ does not necessarily exist.)
We construct approximating
metrics
by use of a Bergman kernel technique and use a Hörmander 
type  $L^2$ estimate to get
the required curvature estimates.
For the convenience of the reader, we recall the basic $L^2$ estimate
that we need.

\begin{mylem}{\rm(Corollary 5.3 in \cite{Dem12})}

Let $(X,\omega)$ be a K\"ahler manifold, $dim\; X=n$. Assume that $X$ is weakly pseudo-convex (in particular it is the case for any compact K\"ahler manifold). Let $F$ be a holomorphic line bundle equipped with a degenerate metric whose local weights are denoted $\varphi \in L^1_{loc}$, i.e. $H=e^{-\varphi}$. 
Suppose that
$$i \Theta_{F,h}=\frac{i}{\pi}\d \dbar \varphi \ge \varepsilon \omega$$
in the sense of currents 
for some $\varepsilon >0$. Then for any form 
$g \in L^2(X,\wedge^{n,q}T_{X}^{*} \otimes F)$ satisfying $\dbar g= 0$, there exists $f \in L^2(X,\wedge^{n,q-1}T_{X}^{*} \otimes F)$ such that $\dbar f=g$ and
$$\int_X |f|^2 e^{-\varphi}dV_\omega \leq \frac{1}{q\varepsilon} \int_X |g|^2 e^{-\varphi}dV_\omega.$$
\end{mylem}
We will also need the following lemma stated by Demailly to glue the local weights into a global one, via a partition of unity.
\begin{mylem}{\rm(Lemma 13.11 in \cite{Dem12})}

Let $U'_j \subset \subset U''_j$ be locally finite open coverings of a (not necessarily compact) complex manifold $X$ by relatively compact open sets, 
and let $\theta_j$ be smooth non-negative functions with support in $U''_j$,
such that $\theta_j \leq 1$ on $U''_j$ and $\theta_j= 1$ on $U'_j$. Let $A_j\geq 0$ be such that
$$i(\theta_j \d \dbar \theta_j-\d \theta_j \wedge \dbar \theta_j) \geq -A_j \omega$$
on $U''_j \setminus U'_j$ for some positive (1,1)-form $\omega$. 
Finally, let $w_j$ be almost psh functions on $U_j$ with the property that
$i\d \dbar w_j\geq \gamma$ for some real(1,1)-form $\gamma$ on $M$, and let $C_j$ be constants such that
$$w_j(x) \le C_j+   \sup_{k\neq j,x  \in U'_k} w_k(x)  $$ on $U''_j \setminus U'_j$.

Then the function $w:= \log(\sum \theta^2_j e^{w_j})$ is almost psh and satisfies
$$i \d \dbar w\geq \gamma-2(\sum_j \mathbbm{1}_{U''_j\setminus U'_j}A_j e^{C_j})\omega.$$
\end{mylem}
Now we give a generalisation of Theorem 1.12 of \cite{DPS94}.
\begin{myprop}{\it
The following properties are equivalent:

(1) $E$ is strongly psef 

(2) There exists a sequence of quasi-psh functions $w_m(x, \xi)=\log( |\xi|_{h_m})$ with analytic singularities induced from Hermitian metrics $h_m$ on $S^m E^*$ 
such that
the singularity locus projects into a proper Zariski closed set $ Z_m$,
and 
$$i \d \dbar w_m \ge -m\varepsilon_m p^*\omega $$ in the sense of currents
with $\lim \varepsilon_m = 0$.
Here $p: S^m E^* \to X$ is the projection.

(3) There exists a sequence of quasi-psh functions $w_m(x, \xi)=\log( |\xi|_{h_m})$ with analytic singularities induced from Hermitian metrics $h_m$ on $S^m E^*$,
such that the singularity locus projects into a proper Zariski closed set $ Z_m$,
and 
$$i\Theta_{S^mE^*,h_m} \leq m\varepsilon_m \omega \otimes \Id$$ 
on $X \setminus Z_m$ in the sense of Griffiths
with $\lim \varepsilon_m = 0$.}

\end{myprop}
\begin{proof}
Note that when a metric over $F$ a vector bundle over $X$ is smooth near a point $x$, we have the following equivalence (cf.\ Lemma 4.4 in \cite{Dem92}):
for any real $(1,1)$ form $\gamma$ near $x$, over a neighbourhood $U$ near $x$ 
\begin{enumerate}
\item $i \Theta(F) \geq \gamma \otimes \Id_F$ in the sense of Griffiths;
\item $-i \Theta(F^*) \geq \gamma \otimes \Id_F$ in the sense of Griffiths;
\item $\frac{i}{2 \pi} \d \dbar \log|\xi|^2 \geq p^* \gamma,$~~$\xi \in F^*$,
where $\log|\xi|^2$ is seen as a function on $p^{-1}(U)$ and $p: F^* \to X$
is the projection.
\end{enumerate}
In particular, (2) implies (3) by this observation.

The more substantial part of the proof consists of showing that (1) implies (2).
The proof follows closely the proof of Theorem 4.1 in \cite{Dem92}.

It is enough to show that for any $\varepsilon >0$, 
there exists a sequence of quasi-psh functions $w_m(x, \xi)=\log( |\xi|_{h_m})$ with analytic singularities induced from Hermitian metrics $h_m$ on $S^m E^*$,
such that
the singularity locus projects into a proper Zariski closed set $ Z_m$,
and 
$$i \d \dbar w_m \ge -m\varepsilon p^*\omega $$ in the sense of currents.
Here $p: S^m E^* \to X$ is the projection.

We construct the metrics on the symmetric powers of vector bundles, starting from a singular metric $h_{\varepsilon}$ on ${\cO}_{\P(E)}(1)$ given in the definition of strongly psef vector bundle. Namely, we start with a singular metric
such that
the singularity locus projects into a proper Zariski closed set $ Z$,
and 
$$\frac{i}{2 \pi}\Theta_{{\cO}_{\P(E)}(1)} \ge -\varepsilon \pi^* \omega.$$
Since $X$ is compact, we can select a finite covering $(W_\nu)$ of $X$ with open coordinate charts. For any $\delta >0$, we take in each $W_\nu$ a maximal family of points with (coordinate) distance to the boundary${}>3\delta$ and mutual distance $>\delta/2$. In this way, we get for any $\delta >0$ small enough a finite covering of $X$ by open balls $U'_j$ of radius $\delta$ (actually every point is even at distance $\le \delta/2$ of one of the centres, otherwise the family of points would not be maximal), such that the concentric ball $U_j$ of radius $2\delta$ is relatively compact in the corresponding chart $W_\nu$.

Let $\tau_j:U_j \to B(a_j,2\delta)$ be the isomorphism given by the coordinates of $W_\nu$. 
Let $\varepsilon(\delta)$ be a modulus of continuity for $\gamma:=-\varepsilon\omega$ on the sets $U_j$, such that $\lim_{\delta \to 0} \varepsilon(\delta) = 0$ and $\omega_x-\omega_{x'} \leq \varepsilon(\delta) \omega_x$ for all $x,x'\in U_j$.
We denote by $\gamma_j$ the (1,1)-form with constant coefficients on $B(a_j,2\delta)$ 
such that $\tau_j^* \gamma_j$ coincides with
$\gamma-\varepsilon(\delta) \omega$ at $\tau_j^{-1}(a_j)$.
Then we have
$$0 \le \gamma-\tau_j^* \gamma_j \le 2\varepsilon(\delta) \omega
\leqno(1)
$$
on $U'_j$ for $\delta >0$ small enough. Let $\tilde{v}_j(z_j)$ be the associated quadratic function such that $\gamma_j=\frac{i}{\pi}\d \dbar \tilde{v}_j$. 

Now, we consider the Hilbert space $\mathcal{H}_j(m)$ of holomorphic sections $f \in H^0(\pi^{-1}(U_j),\cO_{\P(E)}(m))$
with the $L^2$ norm
$$\Vert f \Vert^2_j:=\int_{\pi^{-1}(U_j)}|f|^2 e^{2m \tilde{v}_j(z_j)}dV,$$
where $dV$ is a volume element on $\P(E)$ (fixed once for all) and $|f|^2$ is the pointwise norm on $\cO_{\P(E)}(m)$ induced by the given (singular) Hermitian metric $h_\varepsilon$ on $\cO_{\P(E)}(1)$.
It can be viewed as a metric on $\cO_{\P(E)}(1)$, twisted by the local weight $\tilde{v}_j$.
Thus the corresponding curvature form is 
$$\frac{i}{2 \pi} \Theta_{\cO_{\P(E)}(1)}-\frac{i}{\pi} \d \dbar \tilde{v}_j \ge \pi^*(\gamma - \tau_j^*\gamma_j) \ge 0$$
by (1).
Let $U'_j\Subset U''_j\Subset U_j$ be concentric balls such that 
$(U'_j)$ still cover $X$ and let $\theta_j$ be smooth functions with support in $U''_j$, 
such that $0 \leq \theta_j \leq 1$ on $U''_j$ and $\theta_j= 1$ on $U'_j$. 

We define a Bergman kernel type metric on $S^m E^*$ as follows:
 for all $x\in X$ and $\xi \in S^mE^*_x$ we set
 $$\Vert \xi \Vert^2_{(m)}:=\sum_j \theta^2_j(x)\exp(2m \tilde{v}_j(z_j) +\sqrt{m}(r'^2_j-|z_j|^2))\sum_l | \sigma_{j,l}(x) \cdot \xi|^2,\leqno(2)$$
where $r'_j$ is the radius of $U'_j$ and $(\sigma_{j,l})_{l \ge 1}$ is an orthonormal basis of $\mathcal{H}_j(m)$. 
The local sections $\sigma_{j,l}$ can be viewed as sections in $H^0(U_j,S^mE)$, and here $\sigma_{j,l}(x) \cdot \xi$ is computed via the natural pairing between $S^mE$ and $S^mE^*$.
The metric is Hermitian since it is a sum of square of linear forms in~$S^m E^*$.
Since the metric on $\cO_{\P(E)}(1)$ can be singular, the Hermitian metric can also be degenerate.
It is degenerate at a point $x$ if $\sigma_{j,l}(x)=0$ for all $j,l$.

However, the infinite sum $\sum_l | \sigma_{j,l}(x) \cdot \xi|^2$ is smooth.
In fact, the sum converges locally uniformly above every compact subset of $U_j$.
This sum is the square of evaluation linear form 
$$f \mapsto f(x)\cdot \xi$$
which is continuous on $\mathcal{H}_j(m)$.
The reason is as follows.
Given $\sigma$ an element of $H^0(U_j, S^m E)$. It can be identified as an element of $H^0(\pi^{-1}(U_j),\cO_{\P(E)}(m))\cong H^0(U_j, S^m E)$ 
by considering the quotient of $\pi^* \sigma \in H^0(\pi^{-1}(U_j), \pi^* S^mE)$
under the tautological map $\pi^* S^m E \to  \cO_{\P(E)}(m)$.
On the other hand, $\xi \in S^m E^*_x$ can be pulled back to $\P(E)$ as an element of $\cO_{\P(E)}(-m)_{x, [\xi]} \subset \pi^* S^m E^*_{x,[\xi]}$.
The natural pairing between $S^m E^*$ and $S^m E$ of $f(x)$ and $\xi$ is equal to the natural pairing between $\cO_{\P(E)}(-m)_{x, [\xi]}$ and $\cO_{\P(E)}(m)_{x, [\xi]}$ under the above identification.
In particular,
$$|f(x) \cdot \xi| \leq | f|(x, [\xi]) |\xi|_{} (x, [\xi])$$
Here we identify $\cO_{\P(E)}(m)_{x, [\xi]}$ as $\C$ under any local trivialization near $(x, [\xi])$.
The supremum of $| f|(x, [\xi])$ for $f \in \mathcal{H}_j(m)$, $\Vert f \Vert \le 1$ is by definition the norm of the continuous linear function $f \mapsto f(x)$ under the chosen local trivialization near $(x, [\xi])$.
(Remark that in the trivialization, by mean value inequality, the value of the holomoprhic function at the center of a ball is bounded from above by the $L^2$ norm of the function on the ball which is bounded from above by the $L^2$ norm of the section on $\P(E)$ with the singular weight.)
Thus $f \mapsto f(x)\cdot \xi$ is a continuous linear function.
The square of its norm is $\sum_l | \sigma_{j,l}(x) \cdot \xi|^2$ since $\sigma_{j,l}(x) \cdot \xi$ is the $l$-th coordinate in the orthonormal basis $\sigma_{j,l}$ of $\mathcal{H}_j(m)$.
By Montel's theorem, $\sum_{l,k}  \sigma_{j,l}(x) \cdot \xi~\overline{\sigma_{j,l}(w) \cdot \eta}$ is a holomorphic function for $(x,w,\xi,\eta) \in U_j \times \overline{U_j} \times E \times \overline{E}$.
Thus its restriction  $\sum_l | \sigma_{j,l}(x) \cdot \xi|^2$ to the diagonal $U_j \times E$ is a real analytic function.

As a consequence, the metric $\Vert \cdot \Vert_{(m)}$ is a smooth metric, except for the fact that it might degenerate at some points.
To show that this metric has analytic singularities and obtain the curvature estimate, we use Lemma~2 for $w(x, \xi):= \log \Vert \xi \Vert^2_{(m)}$ and
$$w_j(x, \xi )=2m \tilde{v}_j(z^j)+\sqrt{m} (r'^2-|z^j|^2)+\log \sum_l | \sigma_{j,l}(x) \cdot \xi|^2\leqno(3)$$
on the total space $S^m E^*$ covered by $p^{-1}(U'_j)$ where $p: S^m E^* \to X$ is the projection.

To proceed further, we need the following Lemma 3 to compare the behaviour of $w_j$ on different open sets. 
As a consequence of Lemma 3, the functions $w_j(x,\xi)$ satisfy $w_j(x,\xi)\le w_k(x,\xi)$ for any $x\in (U''_j\setminus U'_j) \cap U'_k$ for $m$ large enough.
(Remark that $r'^2_j-|z^j|^2 \leq 0$ and $r'^2_k-|z^k|^2>0$ for such $x$.)
The choice of $m$ depends on the value $r'^2_k-|z^k|^2>0$.
But the function on $U''_j\setminus U'_j$ , $\sup_{k \neq j, x \in U'_k } |a_k-x|$ has a uniform strictly positive lower bound since $U''_j\setminus U'_j$ is compact.
Thus there exists $m_0$ such that for $m \geq m_0$ we have
$$w_j(x) \le   \sup_{k\neq j,x  \in U'_k } w_k(x)  $$ on $U''_j \setminus U'_j$.
We have a curvature estimate
$$\frac{i}{2\pi} \d \dbar w_j \ge mp^*v_j-\sqrt{m} \frac{i}{2\pi} \d \dbar |z^j|^2 \geq mp^*(\gamma- 3 \varepsilon \omega)$$
in the sense of currents,
since $\gamma_j \ge \gamma-2 \varepsilon \omega$ for $m \geq m'_0 \ge m_0$ large enough (independent of $x$).
Then Lemma 2 implies that
$$\frac{i}{2\pi} \d \dbar w \geq mp^*(\gamma- 3 \varepsilon \omega) -p^*\bigg(2\sum_j \mathbbm{1}_{U''_j\setminus U'_j}A_j\omega\bigg).$$
The right side hand is bigger than $mp^*(\gamma- 4 \varepsilon \omega)$ for $m \ge m''_0 \ge m'_0$.

We observe that the metric has analytic singularities. 
By the following Lemma 3, there exist constants $C_{j,k},C'_{j,k}$ such that
$$w_j -C'_{j,k} \leq w_k \leq w_j +C_{j,k}.$$
Note that $w_j$ can be $-\infty$ at some point. 
Thus we have
$$\log\bigg(\sum_j \theta_j^2 e^{C'_{j,k}} e^{w_k}\bigg) \leq w=\log\bigg(\sum_j \theta_j^2 e^{w_j}\bigg) \leq \log\bigg(\sum_j \theta_j^2 e^{C_{j,k}} e^{w_k}\bigg).$$
Without loss of generality, we can assume that $\theta_j$ is a partition of unity, and in particular that $\sum_j \theta_j^2$ is strictly positive on any relative compact set.
Thus $w=w_k +O(1)$ which implies $w$ has analytic singularities along with $w_k$.

Now we show that (2) implies (1).
The sequence of metrics in (2) induces a sequence of Hermitian metrics on $\cO(1)$ over $\P(S^m E)$.
Observe that we have the following commutative diagram given by the Veronese embedding
$$
\begin{matrix}
\cO(m)&\kern-6pt\longrightarrow\kern-6pt & \cO(1)\\
\noalign{\vskip4pt}  
\Big\downarrow&  & \Big\downarrow\\
\P(E)&\kern-6pt\mathop{\longrightarrow}\limits^i\kern-6pt& \P(S^m E).
\end{matrix}
$$
Since the metric is smooth over the pre-image of a dense Zariski open set of $X$.
The restriction of singular metrics is well defined and still has analytic singularities.
Define a sequence of metrics on $\cO_{\P(E)}(1)$ induced from the restricted metrics.
This sequence of metrics is the one required in the definition of a strongly psef vector bundle.

The arguments needed to show that (3) implies (2) are similar. 
By the observation made at the beginning of the proposition, the inequality holds on a dense Zariski open set $V$ where the metric is smooth. The
Skoda-El Mir extension theorem implies that $\mathbbm{1}_V i \d \dbar w_m \geq -m \varepsilon_m p^* \omega$.
Since $w_m$ has analytic singularities, the current $ i \d \dbar w_m $ is normal, and by the support theorem $\mathbbm{1}_{S^m E^* \setminus V}i \d \dbar w_m$ is a sum of closed positive currents obtained by integration on analytic sets with positive coefficients.
Thus the same inequality holds for $i \d \dbar w_m= \mathbbm{1}_V i \d \dbar w_m+\mathbbm{1}_{S^m E^* \setminus V} i \d \dbar w_m$.
\end{proof}
\begin{mylem}
There exist constants $C_{j,k}$ independent of $m$ such that the almost psh functions 
$$\tilde{w}_j(x,\xi) := 2m v_j(z_j) + \log \sum_l |\sigma_{j,l}(x)\cdot \xi|^2, (x,\xi)\in p^{-1}(U''_j) \subset S^mE^*$$
satisfy
on $p^{-1}(U''_j \cap U''_k)$ a bound
$$\tilde{w}_j \le \tilde{w}_k+(2n+2) \log m +C_{j,k}.$$
\end{mylem}
\begin{proof}
By construction $E|_{{U}_j} \cong {U}_j \times \C^r$ is trivial over ${U}_j$. 
Define a Hermitian metric $h_\infty$ on $E|_{{U}_j}$ with strict positive curvature by taking 
$$|\xi|^2:=\sum_{\lambda}|\xi_\lambda|^2 e^{-\sum_j |z^j|^2}.$$
The associated curvature form on $(\cO_{\P(E|_{{U}_j})}(1),h_\infty)$ is strictly positive and thus  defines a K\"ahler metric $\omega_j$ on $\pi^{-1}(U_j)$.
In fact, $\Theta_E= \omega_{\rm eucl}\otimes \Id_E$ where $\omega_{\rm eucl}$ is the standard (flat) Hermitian metric on $U_j$.
By a standard formula (cf.\ formula (15.15) in Chap V of \cite{agbook}), the curvature of $(\cO_{\P(E|_{{U}_j})}(1),h_\infty)$ is equal to the direct sum of the Euclidean metric of $U_j$ and of the Fubini-Study metric of $\P^{r-1}$. In particular, the Ricci curvature of $\omega_j$ is non-negative.
Define $\tau(z) :=n \log|z^j-z^j(x)|$ depending only on the base variables and  possessing a logarithmic pole at $x$. This is a psh function on a neighbourhood of $\pi^{-1}({U_j})$.
Define a singular metric on $\cO_{\P(E|_{{U}_j})}(m)$ as follows.
Twist the metric $h_{\varepsilon}^{\otimes (m-1)} \otimes h_\infty$ by $(m-1)\tilde{v}_j(z^j)+\tau(z^j)$.
The resulting curvature form on $\cO_{\P(E|_{{U}_j})}(m)$ is given by
$$(m-1)\Big(\frac{i}{2 \pi} \Theta_{\cO_{\P(E)}(1)}(h_\varepsilon)-\frac{i}{\pi} \d \dbar \tilde{v}_j\Big)+\omega_j +\frac{i}{\pi} \d \dbar \tau \geq \omega_j$$
by (1).
We consider the Hilbert space $F^{0,q}_j(m)$ of $(0,q)$-forms ($q=0,1$) $f$ on $\pi^{-1}(U_j)$ with values in $\cO_{\P(E)}(m)$, equipped with the $L^2$ norm
$\Vert f \Vert^2_{j,q}=\int_{\pi^{-1}(U_j)}|f|_j^2 dV_j$,
where $dV_j=\omega^{n+r-1}_j/(n+r-1)!$ and where the pointwise norm $|f|_j$ is induced by $\omega_j$ and of the metric defined above on $\cO_{\P(E)}(m)$. 

Now, we apply H\"ormander's $L^2$ estimates for the bundle $-K_X+\cO_{\P(E)}(m)$ and an arbitrary $(0,1)$ form $g$ in $F_j^{0,1}(m)$ with $\dbar g=0$, (i.e.\ a $\dbar$-closed $L^2$ $(n,1)$-form valued in $-K_X+\cO_{\P(E)}(m)$). We conclude that there exists a $(0,0)$-form in $F_j^{0,0}(m)$ such that $\dbar f=g$ and $\Vert f \Vert_{j,0} \le \Vert g \Vert_{j,1}$. (Note that $\Ric(\omega_j) \ge 0$.)

It remains to choose a suitable section $g$ to prove the inequality. 
Fix a point $x \in U''_j \cap U''_k$ and $\xi \in S^mE^*_x$. There exists $h \in  \mathcal{H}_k(m)$ with $\Vert h \Vert_k= 1$ such that
$$|h(x)\cdot \xi|^2=\sum_l |\sigma_{k,l}(x)\cdot \xi|^2.$$
If the right rank side is 0, we can take $h$ to be any element in the orthonormal basis. Otherwise, the linear functional $f \mapsto f(x)\cdot \xi$ is a non zero functional whose kernel defines a closed hypersurface in~$\mathcal{H}_k(m)$. 
Thus there exists $h \in  \mathcal{H}_k(m)$ with $\Vert h \Vert_k= 1$ which is orthogonal to the kernel.
It is easy to see that such a point $h$ is a maximum of the function $\mathcal{H}_k(m) \setminus 0 \to \R$: 
$$v \mapsto  \frac{|v \cdot \xi|}{\Vert v \Vert^2},$$
and hence we have the equality.
Let $\chi$ be a cut-off function with support in the (coordinate) ball $B(x,1/m)$, equal to 1 on $B(x,1/2m)$ and with $|\d \chi| \le m$.
For $m\ge m_0$ large enough (independent of $x \in U''_j \cap U''_k$) we have $B(x,1/m) \subset U_j \cap U_k$.
We consider the solution of the equation $\dbar f=h\dbar(\chi \circ \pi)$
on $\pi^{-1}(U_j)$. We then get a holomorphic section
$$h':= h (\chi \circ \pi)-f \in H^0(\pi^{-1}(U_j),\cO_{\P(E)}(m)).$$
The section $h'$ coincide with $h$ over $\pi^{-1}(x)$, since the Lelong number of the local weight at a point in $\pi^{-1}(x)$ is at least that of the local weight of $\tau$ which is $n$.
The fact that the section $f$ is in $L^2$ implies that it has to vanish along $\pi^{-1}(x)$. On the other hand, we have

\begin{align*}
\Vert h \dbar(\chi \circ \pi) \Vert_{j,1}^2 &\leq m^2 \int_{\pi^{-1} (B(x,1/m) \setminus B(x,1/2m))} \frac{|h|^2_{h_{\varepsilon}^{\otimes (m-1)} \otimes h_\infty} e^{2(m-1)\tilde{v}_j(z^j)}}{|z^j-z^j(x)|^{2n}}dV_j \\
& \leq C m^{2n+2} \int_{\pi^{-1} (B(x,1/m) \setminus B(x,1/2m))} |h|^2_{h_{\varepsilon}^{\otimes (m-1)} \otimes h_\infty} e^{2(m-1)\tilde{v}_j(z^j)}dV_j
\\ &\leq C m^{2n+2} \int_{\pi^{-1} (B(x,1/m) )} |h|^2_{h_{\varepsilon}^{\otimes (m-1)} \otimes h_\infty} e^{2(m-1)\tilde{v}_j(z^j)}dV_j
\\& \leq C m^{2n+2} \int_{\pi^{-1} (B(x,1/m) )} |h|^2_{h_{\varepsilon}^{\otimes m} } e^{2(m-1)\tilde{v}_j(z^j)}dV_j
\\&\leq C m^{2n+2}e^{2m(\tilde{v}_j(z^j(x))-\tilde{v}_k(z^k(x)))} \int_{\pi^{-1} (B(x,1/m) )} |h|^2_{h_{\varepsilon}^{\otimes m} } e^{2m\tilde{v}_k(z^k(z^k))}dV_k
\\&\leq Cm^{2n+2}e^{2m(\tilde{v}_j(z^j(x))-\tilde{v}_k(z^k(x)))} \Vert h \Vert_k^2
\end{align*}

All the constants are independent of $x$ and $m$.
For the fourth inequality we use the fact that $h_{\varepsilon} \geq C h_\infty$ for some $C$ on $\P(E|_{{U}_j})$, since $h_{\varepsilon}$ has analytic singularities, $h_\infty$ is smooth and the $U_j$'s are relatively compact.
For the fifth inequality, we use the fact that the oscillation of $\tilde{v}_j$ and $\tilde{v}_k$ on $B(x,1/m)$ is $O(1/m)$.
By H\"ormander's $L^2$ estimates we obtain
$$\Vert f \Vert_{j,0}^2 \leq Cm^{2n+2}e^{2m(\tilde{v}_j(z^j(x))-\tilde{v}_k(z^k(x)))} \Vert h \Vert_k^2.$$
Since $\tau \leq 0$ and $h_{\varepsilon} \geq C h_\infty$, we have for some $C$
$$\Vert f \Vert_{j}^2 \leq C \Vert f \Vert_{j,0}^2.$$ 
The norm $\Vert h (\chi \circ \pi)\Vert_j$ satisfies a similar estimate
$$\Vert h (\chi \circ \pi)\Vert_j \leq Cm^2 e^{2m(\tilde{v}_j(z^j(x))-\tilde{v}_k(z^k(x)))} \Vert h \Vert_k^2$$
where $C$ comes from the change of volume form from $dV_j$ to $dV_k$ and the oscillation of $\tilde{v}_j$ and $\tilde{v}_k$ on $B(x,1/m)$.
Thus we have
$$\Vert h' \Vert_j \leq Cm^{2n+2} e^{2m(\tilde{v}_j(z^j(x))-\tilde{v}_k(z^k(x)))},$$ 
\begin{align*}
\sum_l |\sigma_{j,l}(x)\cdot \xi|^2 & \geq C^{-1}m^{-2n-2} e^{-2m(\tilde{v}_j(z^j(x))-\tilde{v}_k(z^k(x)))} |h'(x) \cdot \xi|^2\\
& \geq C^{-1}m^{-2n-2} e^{-2m(\tilde{v}_j(z^j(x))-\tilde{v}_k(z^k(x)))} \sum_l |\sigma_{k,l}(x)\cdot \xi|^2
\end{align*}
since $h'(x) =h(x)$ and $\sum_l |\sigma_{k,l}(x)\cdot \xi|^2=|h(x) \cdot \xi|^2$.
By taking logarithms, we infer the desired inequality.
\end{proof}
\begin{myrem} {\rm
We have formulated the proposition in terms of $E^*$ instead of $E$ for the following reason. According to \cite{BP08} and section 16 of \cite{HPS}, the dual metric of a singular metric of vector bundle is always pointwise well defined.
However the dual metric is not necessarily continuous if the original metric is continuous.
Let us consider a case where the metric has analytic singularities.
Assume that $\log |\xi|_h$ has analytic singularities as a function on the total space $V$ for some vector bundle $(V,h)$ and is the form of $\log \sum |f_i(x) \cdot \xi|^2+ \psi(x)$ with $f_i$ are holomorphic vector bundle sections and $\psi$ is bounded. This is for instance the case for the approximating metrics used in Proposition~1.
The function $\log |\xi^*|_{h^*}$ on the total space $V^*$  is the difference of two real analytic functions modulo bounded terms, on the dense Zariski open set where the metric is smooth. At points where the metric is smooth, we have $\log |\xi|^2_h= \log (\xi^\dagger H(x) \xi)$ for some Hermitian matrix $H(x)$ where $\dagger$ means the Hermitian transpose.
Thus one has
$\log |\xi^*|_{h^*}= \log (\xi^{*\dagger} (H^{-1 }(x)) \xi^*)$ which can be calculated from  the determinant and the adjugate matrix of $H(x)$.
Each component of the adjoint matrix and of the determinant is the product of a bounded function times a real analytic series in the $z^j$'s (coordinates of $x$) and in $\xi$.
Near the singular locus of the metric $h$, both functions can tend to infinity for fixed $\xi^*$. These facts would result in more difficulties to be dealt with.

Here is a concrete example taken from Raufi \cite{Rau}.
Let $E$ be the trivial rank 2 vector bundle over $\C$ where the metric at $z \in \C$ is represented by the matrix
\begin{equation*}
H := 
\begin{pmatrix}
1+|z|^2 & z  \\
\overline{z} & |z|^2  \\
\end{pmatrix}.
\end{equation*}
On $\C^*$, the dual metric can be represented by the matrix
\begin{equation*}
(H^{-1} )^\dagger= \frac{1}{|z|^4}
\begin{pmatrix}
|z|^2 & -z  \\
-\overline{z} & 1+|z|^2  \\
\end{pmatrix}.
\end{equation*}
Thus $\log|\xi^*|_{h^*}=\log( |z \xi^*_2|^2+|\bar{z}\xi_1^*+\xi^*_2|^2)-\log|z|^4$.
At $\xi^*=(1,0)$, $\log|\xi^*|_{h^*}$ is a difference of two functions both tending to infinity  when $z$ tends to 0.}
\end{myrem}
\begin{myrem}{\rm
We can also interpret the inequality $$i\Theta_{S^mE^*,h_m} \le m\varepsilon_m \omega \otimes \Id $$
in the sense of currents as follows: 
for any non-trivial local section $s$ of $S^m E^*$, $m\varepsilon_m \omega+i \d \dbar \log|s|^2_{h_m}$ is a positive current.
The local section can be seen as a map $i$ from an open subset of $X$ to the total space  $S^m E^*$.
If we pull back the current (2) to $U$ via $i$, we see that $m\varepsilon_m \omega+i \d \dbar \log|s|^2_{h_m}$ is a positive current. Here
$|s|_{h_m}$ is not identically zero since it is non vanishing outside of the zero locus of $s$ and of singular locus of $h_m$.

Further discussions of these points can be found in \cite{Paun16}.
The above proposition also answers partially to a question proposed in Remark 2.11 of \cite{Paun16}.
Given a singular Finsler metric with analytic singularities on a vector bundle, one can produce singular Hermitian metrics on high order symmetric tensor products of the given vector bundle, with arbitrary small loss of positivity.}
\end{myrem}
As a direct consequence of the approximation statement, we have the following corollary.
\begin{mycor} {\it
If $E$ is a strongly psef vector bundle of rank $r$ over a compact K\"ahler manifold $(X, \omega)$, then $\det(E)$ is a psef line bundle.}
\end{mycor}
\begin{proof}
On $X \setminus Z_m$, the curvature inequality
$$i\Theta_{S^mE^*,h_m} \le m\varepsilon_m \omega \otimes \Id$$ 
implies that $i \Theta_{\det S^m E,\det h_m^*} \geq - \rank(S^m E)m \varepsilon_m \omega$.
On the other hand
$$\det S^m E=(\det E)^{\otimes \frac{m \rank(S^m E)}{r}}.$$
Therefore, the induced metric on $\det(E)$ satisfies on $X \setminus Z_m$ the curvature inequality
$$i\Theta_{\det(E)} \geq -r\varepsilon_m \omega.$$
Let us point out that the metric $h_m$ is smooth on $X$ (although it might vanish at some points). The induced metric on $-\det(E)$ is locally bounded.
In other words, the local weight of the dual metric on $\det(E)$ is locally bounded from above.
By the Riemann extension theorem, the curvature inequality holds in the sense of currents throughout $X$, and not only on $X \setminus Z_m$.
By weak compactness, up to taking some subsequence, we get in the limit a closed positive current belonging to the class $c_1(\det E)$. This shows that $\det(E)$ is psef.
\end{proof}
Another direct application of the approximation is the following corollary.
\begin{mycor} {\it
Let $E$ be a vector bundle over a compact K\"ahler manifold $(X, \omega)$.
The following properties are equivalent.

(1) $E$ is strongly psef.

(2) For any $m \in \N^*$, $S^m E$ is strongly psef.

(3) There exists $m \in \N^*$ such that $S^m E$ is strongly psef.}
\end{mycor}
\begin{proof}
(2) implies (3) trivially. (3) implies (1) as in the proof of (2) implying (1) in Proposition 1.
(1) implies (2) is a direct consequence of Proposition 1. All symmetric products $S^{mp} E$ of $ E$ ($p \in N^*$) are quotients of symmetric products of $S^p(S^m E)$.
On the other hand, the induced metric on the quotient bundle of a vector bundle will satisfy similar curvature condition as the original metric as in point (1) the following corollary.  
\end{proof}
As a consequence, one can also define "$\Q-$twisted" strongly psef vector bundles as follows.
\begin{mydef}
Let $(X, \omega)$ be a compact K\"ahler manifold and $E$ a holomorphic vector bundle on $X$ and $D$ be a $\Q-$line bundle. 
Then $E\langle D \rangle$ is said to be $\Q-$twisted strongly pseudo-effective $(\Q-$strongly psef for short$)$ if $S^m E \otimes \cO_X(mD)$ is strongly psef for some (hence any by Corollary 2) $m >0$ such that $\cO_X(mD)$ is a line bundle.
\end{mydef}
As in \cite{DPS94}, one can derive some natural
algebraic properties of strongly psef vector bundles.
\begin{mycor}[{\rm Algebraic properties of strongly psef vector bundles}]\strut {\it
\begin{enumerate}
\item A quotient bundle of a strongly psef vector bundle is strongly psef.
\item A direct summand of strongly psef vector bundles is strongly psef.
\item A direct sum of strongly psef vector bundles is strongly psef.
\item A tensor product $($or Schur functor of positive weight$\,)$
  of strongly psef vector bundles is strongly psef.
\end{enumerate}}
\end{mycor}
\begin{proof} One can obtain lower bounds of the curvature through
calculations very similar to those of \cite{DPS94}. We first show
that the induced singular metric has analytic singularities.

Assume $E$ to be strongly psef. The surjective bundle morphism $E \to Q$ induces a
closed immersion of $\P(Q)$ into $\P(E)$, and the restriction
of $\cO_{\P(E)}(1)$ to $\P(Q)$ is $\cO_{\P(Q)}(1)$.
The singular metrics on $\cO_{\P(E)}(1)$ prescribed in the definition
of a strongly psef vector bundle induce by restriction singular metrics with
analytic singularities on $\cO_{\P(Q)}(1)$. If we observe that all metrics
involved are smooth over inverse images of non-empty Zariski open sets,
we infer that the restricted metrics are not identically infinite.
This concludes the proof of (1).

(1) implies (2) since a direct summand can be seen as a quotient bundle.
Now, let $E,F$ be two strongly psef vector bundles.
The Hermitian metrics on $\cO_{\P(E)}(1)$ and $\cO_{\P(F)}(1)$ correspond to Finsler metrics on $E^*$ and $F^*$ denoted by $h_E,h_F$. Then $h_E+h_F$ defines
a Finsler metric with analytic singularities on $E^* \oplus F^*$.
It corresponds to a Hermitian metric on $\cO_{\P(E \oplus F)}(1)$, and the
properties required in the definition can easily be checked for $h_E+h_F$ if
they are satisfied for $h_E$ and $h_F$. This concludes the proof of (3).

By Corollary 2 and (3), $S^2(E \oplus F)$ is strongly psef as soon as $E,F$ are.
Since $$S^2(E \oplus F) \cong S^2 E \oplus (E \otimes F) \otimes S^2 F,$$
we infer by (2) that $E \otimes F$ is strongly psef. Finally, the fact that a Schur tensor power is a direct summand of a tensor product implies (4).
\end{proof}
\begin{mycor} {\it
Let
$$0\to S \to E \to Q \to 0$$
be an exact sequence of holomorphic vector bundles.
If $E$ and $(\det(Q))^{-1}$ are strongly psef, then $S$ is strongly psef.}
\end{mycor}
\begin{proof}
We have $S= \bigwedge^{s-1}S^* \otimes \det S$ where $s$ is the rank of $S$. By dualizing and taking the $s-1$ exterior product, we get a surjective bundle morphism 
$$\bigwedge^{s-1} E^* \to \bigwedge^{s-1} S^*= S \otimes (\det S)^{-1}.$$
On the other hand, we have $\det E \cong  \det S \otimes \det Q$, thus we have a surjective bundle morphism
$$\bigwedge^{r-s-1} E \otimes (\det Q)^{-1} \to S$$
where $r$ is the rank of $E$ by tensoring $\det E$.
By (4) of Corollary 3, $\bigwedge^{r-s-1} E \otimes (\det Q)^{-1}$ is strongly psef.
By (1) of Corollary 3, $S$ is strongly psef.
\end{proof}
\section{Reflexive sheaves}
In this section, we show that a numerically flat reflexive sheaf on a compact K\"ahler manifold is in fact a vector bundle.
We need the following topological lemmata.
\begin{mylem}
Let $X$ be an arbitrary complex manifold $($non necessarily compact$)$ and $E$ be a vector bundle on $X$. 
Let $X_0$ be a Zariski open set in $X$ with $codim(X \setminus X_0)\geq 3$. 
Then the morphism induced by the restriction morphism
$H^1(X, E)\to H^1(X_0, E)$ is surjective.
\end{mylem}
\begin{proof}
We start by proving that
$$H^1(\C^3 \setminus \{(0,0,0)\}, \cO_{\C^3 \setminus \{(0,0,0)\}})=0.$$
It is done by direct calculation. Cover $\C^3 \setminus \{(0,0,0)\}$ by three Stein open sets isomorphic to $\C^* \times \C^2$, say $U_i=\{z_i \neq 0\}$,
with coordinates $(z_0,z_1,z_2)$. A 1-cochain can be identified with
a triple of convergent power series $(f_{01},f_{02},f_{12})$ with $f_{12}$ (say) of type
$$\sum_{(\alpha, \beta, \gamma) \in \Z^2 \times  \N} 
c_{\alpha \beta \gamma }z_0^\alpha z_1^\beta z_2^\gamma$$
over $\C^{*2} \times \C$ (the intersection of two Stein open sets).
Similarly, $f_{02}$ is a sum over $(\alpha, \beta, \gamma) \in \Z \times \N \times \Z$ and $f_{01}$ is a sum over $(\alpha, \beta, \gamma) \in \N \times \Z^2$.

The condition that $(f_{01},f_{02},f_{12})$ is closed means that $f_{01}-f_{02}+f_{12}=0$ on the intersection of the three Stein open sets $U_0\cap U_1\cap U_2$, biholomorphic to $\C^{*3}$.
We can write $f_{01}$ as a sum of three convergent power series $g_{01}^0$, $g_{01}^1$, $g_{01}$ such that $g_{01}$ has only positive power terms, $g_{01}^0$ has only negative power terms in $z_0$ and $g_{01}^1$ has only negative power terms in $z_1$.
Similarly, we decompose $f_{02},f_{12}$.
Now the closeness condition is equivalent to
$$g_{01}-g_{02}+g_{12}=0,\quad
g_{01}^0=g_{02}^0,\quad g_{12}^2=g_{02}^2,\quad
g_{01}^1+g_{12}^1=0.$$
We define a 0-cochain in such a way that its differential is $(f_{01},f_{02},f_{12})$. On $U_0$, resp.\ $U_1$, $U_2$, we take the convergent power series $g_{01}+g_{01}^0$, resp.\ $g_{12}^1$, $-g_{12}-g_{02}^2$.
This implies that every 1-cocycle is exact, hence
$$H^1(\C^3 \setminus \{(0,0,0)\}, \cO_{\C^3 \setminus \{(0,0,0)\}})=0.$$
Now, on every polydisc $D$ in $\C^3$, a holomorphic function is uniquely determined by its Taylor expansion at origin, and the same calculation shows that 
$$H^1(D \setminus \{(0,0,0)\}, \cO_{D \setminus \{(0,0,0)\}})=0.$$
By a similar calculation, we can show that for any polydisc $D$ of dimension at least 3,
$$H^1(D \setminus \{0\}, \cO_{D \setminus \{0\}})=0.$$
By the K\"unneth formula, for $B' \times (B''\setminus \{0\})$ where $B',B''$ are polydiscs with dimension of $B''$ at least~$3$, we have
$H^1(B' \times (B''\setminus \{0\}), \cO_{B' \times (B''\setminus \{0\})})=0.$

We now return to the general case. By the standard lemma below ensuring the existence of stratifications of analytic sets, we can reduce ourselves to the situation where $X \setminus X_0$ is a closed manifold.

Cover $X$ by the Stein open sets $U_\alpha$ and $B_\beta:=B'_\beta \times B''_\beta$ such that $X_0$ is covered by $U_\alpha$ and  $B'_\beta \times (B''_\beta\setminus \{0\})$ where $B'_\beta,B''_\beta$ are polydiscs with dimension of $B''_\beta$ at least 3. 
Assume that $E$ is trivial on $U_\alpha$ and $B_\beta$.
Cover $B'_\beta \times(B''_\beta\setminus \{0\})$ by $B_\beta^\gamma$
($1 \leq \gamma \leq \dim B''_\beta$) such that
each $B_\beta^\gamma$ is isomorphic to a polydisc minus a hyperplane defined as zero set of one coordinate.
Since $U_\alpha$, $B_\beta^\gamma$ are Stein, the cohomology on $X_0$ can be calculated as the \v{C}ech cohomology with respect to this open covering of~$X_0$, which we denote by $\mathcal{V}$. We also denote by $\mathcal{U}$ the open covering of $X$ consisting of the sets $U_\alpha$, $B_\beta$.
Any element $s$ of $H^1(X_0,E)$ can be represented by a family of sections
\begin{eqnarray*}
&&\kern-20pt(s_{\alpha_1, \alpha_2}, s_{\alpha \beta}^\gamma, s_{\beta}^{\gamma_1, \gamma_2},s_{\beta_1, \beta_2}^{\gamma_1, \gamma_2}) \in\\
&&\kern10pt
   \prod \Gamma(U_{\alpha_1} \cap U_{\alpha_2},E) \times  \prod \Gamma(U_\alpha \cap B_\beta^\gamma,E) \times \prod \Gamma(B_\beta^{\gamma_1} \cap B_\beta^{\gamma_2}, E)  \times \prod \Gamma(B_{\beta_1}^{\gamma_1} \cap B_{\beta_2}^{\gamma_2}, E).
\end{eqnarray*}
Since $H^1(B'_\beta \times B''_\beta,E)=0$ by the previous case,
there exists
$$(s_\beta^\gamma) \in \prod \Gamma(B_{\beta}^{\gamma}, E)$$
such that for any $\beta$ fixed
$$s_{\beta}^{\gamma_1, \gamma_2}=(-1)^{\gamma_1+1}s_{\beta}^{\gamma_1}+(-1)^{\gamma_2+1}s_{\beta}^{\gamma_2}.$$
Define a 0-cochain 
$$(s_\beta^\gamma, 0) \in \prod \Gamma(B_{\beta}^{\gamma}, E) \times \prod \Gamma(U_\alpha,E).$$
Then we have $(s_{\alpha_1, \alpha_2}, s_{\alpha \beta}^\gamma, s_{\beta}^{\gamma_1, \gamma_2},s_{\beta_1, \beta_2}^{\gamma_1, \gamma_2})+\delta (-s_\beta^\gamma, 0)$ as another representative of the same cohomology class on $X_0$.
The components in $ \Gamma(B_\beta^{\gamma_1} \cap B_\beta^{\gamma_2}, E)$ are 0 by construction.
Thus we can assume that the components in $ \Gamma(B_\beta^{\gamma_1} \cap B_\beta^{\gamma_2}, E)$ are 0 from the beginning.

Since the representative is closed, the components in 
$\Gamma(B_{\beta}^{\gamma} \cap U_{\alpha}, E)$ glue to a section $s_{ \alpha, \beta}\in \Gamma(B_{\beta}\setminus (B'_\beta \times \{0\}) \cap U_\alpha, E)$ when $\gamma$ varies. By the Hartogs theorem, this section extends across the submanifold $B'_\beta \times \{0\}$, as its codimension is at least 3. The components in
$\Gamma(B_{\beta_1}^{\gamma_1} \cap B_{\beta_2}^{\gamma_2}, E)$ can be glued into a section of $\Gamma(B_{\beta_1}^{\gamma_1}\cap B_{\beta_2}, E)$ when $ \gamma_2$ varies, and into a section of $\Gamma(B_{\beta_1}\cap B_{\beta_2}^{\gamma_2}, E)$ when $ \gamma_1$ varies.
By the unique continuation theorem for holomorphic functions, in fact they define a holomorphic section $s_{\beta_1, \beta_2}$ of $E$ on $B_{\beta_1}\cap B_{\beta_2}$.

We claim that after performing this glueing, the sections
$$(s_{\alpha_1, \alpha_2},s_{\alpha, \beta},s_{\beta_1, \beta_2}) \in \prod \Gamma(U_{\alpha_1} \cap U_{\alpha_2},E) \times  \prod \Gamma(U_\alpha \cap B_\beta,E) \times  \prod \Gamma(B_{\beta_1}^{} \cap B_{\beta_2}^{}, E)$$
define a $1$-cocycle of $X$ with respect to the open covering
$U_\alpha$, $B_\beta$, and that its class in $H^1(X_0,E)$ is exactly $s$.

The reason is as follows.
The image of $(s_{\alpha_1, \alpha_2},s_{\alpha, \beta},s_{\beta_1, \beta_2})$ from $H^1(\mathcal{U},E)$ to $H^1(\mathcal{U} \cap X_0,E)$ is just the restriction of sections.
The covering $\mathcal{V}$ is a refinement of $\mathcal{U} \cap X_0$ given by
the inclusion of open sets:
$U_\alpha \subset U_\alpha$,
$B_\beta^\gamma \subset B_\beta$. 
The image under this refinement of open sets is precisely $s$.
\end{proof}
\begin{mylem}[{\rm Stratification of analytic sets, see e.g.\
Proposition 5.6 in Chap. II of \cite{agbook}}]\strut\\
Let $Z \subset X$ be an analytic subset of dimension $n$.
Then $Z$ admits a stratification $\emptyset =Z_{n+1} \subset \cdots \subset Z_0=Z$ by closed analytic sets $Z_k$ of dimension $n_k > n_{k+1}$ such that $Z_k \setminus Z_{k+1}$ is a closed complex submanifold of dimension $n_k$ of $X \setminus Z_{k+1}$.
\end{mylem}
Let us point out that the result is false if the codimension is equal to~$2$.
For example, the group
$H^1(\C^2 \setminus \{(0,0)\}, \cO_{\C^2 \setminus \{(0,0)\}})$
is infinite dimensional, while $H^1(\C^2 , \cO_{\C^2 })=0$ by Cartan's
theorem~B.

We recall briefly the construction of Chern classes of a coherent sheaf $\cF$ in the de Rham cohomology.
We refer to \cite{GR58} for more details.
If $X$ is connected complex compact manifold (or more generally a Zariski open set $U$ of in $X$), by \cite{Voi}, $\cF$ does not necessarily admit a resolution by holomorphic vector bundles.
On the other hand, a real analytic coherent sheaf does by real analytic vector bundles.
Let 
$$0 \to E^{2n} \to \cdots E^0 \to \cF \otimes_{\cO_X} \cO_X^{\mathrm{an}} \to 0  $$
be a resolution of $\cF \otimes_{\cO_X} \cO_X^{\mathrm{an}} $ by real analytic vector bundles $E^i$ where $\cO_X^{\mathrm{an}}$ is sheaf of real analytic function on $X$ and $n$ is the complex dimension of $X$.
Define the total Chern class of $\cF$ by
$$c_\bullet(\cF):= \prod_i c_\bullet(E^i)^{(-1)^i}.$$
By restriction on $U$, same formula defines $c_\bullet(\cF|_U)$.
It can be check that, this is independent the choice of resolution.
\begin{mylem}
Let $\cF$ be a coherent torsion sheaf over a compact complex manifold $X$ (not necessarily K\"ahler) of dimension $n$.
Assume that $\cF$ is supported in a SNC divisor $E= \cup_i E_i$ where $E_i$ are the irreducible components.
Let $\alpha$ be a smooth closed form over $X$ such that $\alpha|_{E_i}=0$ for any $i$.
Then for any $i < n$,
$$\int_X c_i(\cF) \wedge \alpha^{n-i}=0.$$ 
More generally we have
for any $i < n$ and any cohomology class $\beta$ of $X$,
$$\int_X \ch(\cF) \wedge \beta  \wedge \alpha^{n-i}=0.$$
\end{mylem}
\begin{proof}
Denote for any divisor $D$ (not necessarily irreducible) $G_D(X)$ the Grothendieck group of coherent sheaves over $X$ supported in $D$.
We have exact sequence
$$\oplus_i G_{E_i}(X) \to G_E(X) \to 0.$$ 
Let $(\cF_i) \in \oplus_i G_{D_i}(X)$ be a preimage of $\cF$.
Then we have by construction of Chern character class (cf. \cite{Gri}),
$$\ch(\cF)=\sum_i i_{E_i*} (\ch(\cF_i)\td(N_{E_i/X})^{-1})$$ 
where $i_{E_i}$ is the closed immersion and $\td(N_{E_i/X})$ is the Todd class of the normal bundle of $E_i$.
For any cohomology class $\beta$ on $X$,
$$\int_X \ch(\cF)\wedge \beta \wedge \alpha^{n-i}=
\sum_i \int_{E_i} \ch(\cF_i)\td(N_{E_i/X})^{-1} 
\wedge i_{E_i}^* \beta \wedge i_{E_i}^* \alpha^{n-i}=0$$
since $i_{E_i}^* \alpha=0$.
\end{proof}
As an application of this lemma, we have the following result.
\begin{mylem}
Let $\cF$ be a reflexive sheaf over a compact complex manifold $X$.
Let $\sigma: \tilde{X} \to X$ be a modification of $X$ such that there exists a SNC divisor $E$ in $\tilde{X}$ such that 
$$\sigma: \tilde{X} \setminus E \to X \setminus \pi(E)$$
is biholomorphism with $E$ a SNC divisor and the codimension of $\pi(E)$ at least 3
and $\sigma^* \cF /\tors$ is locally free.
Then we have that for $i=1,2$
$$c_i(\cF)=\sigma_* (c_i(\sigma^* \cF/\tors)).$$
\end{mylem}
\begin{proof}
First observe that such a modification always exists by the fundamental work of \cite{Ros}, \cite{GR}, \cite{Rie}.

Without loss of generality we can assume that the dimension of $X$ is at least 3.
Otherwise, $\cF$ is locally free and the result is direct.
By Poincar\'e duality, it is the same to prove that for $i=1,2$ and any cohomology class $\alpha$ we have that
$$\int_X c_i(\cF) \wedge \alpha=\int_{\tilde{X}} (c_i(\sigma^* \cF/\tors)) \wedge \sigma^* \alpha.$$
Recall that $\sigma^* \ch(\cF)=\sum_i (-1)^i \ch(L^i \sigma^* \cF)^{}$ where $L^i \sigma^* $ is the $i$-th left derived functor of $\sigma^*$.
Without loss of generality, we can assume that $\cF$ is locally free over $X \setminus \pi(E)$.
In particular,
$L^i \sigma^* \cF$ for any $i >0$ is supported in the exceptional divisor.
On the other hand, the torsion part of $\sigma^*$ is also supported in the exceptional divisor.
By the above lemma, we have that
$$\int_{\tilde{X}} (c_i(\sigma^* \cF/\tors)) \wedge \sigma^* \alpha=\int_{\tilde{X}} \sigma^*(c_i(\cF)) \wedge \sigma^* \alpha$$
which concludes the proof.
\end{proof}
 We now introduce the definition of nef and strongly psef torsion-free sheaves.
\begin{mydef}[Nef/ Strongly psef torsion-free sheaf]\strut\\
Assume that $\cF$ is a torsion free sheaf over a compact complex manifold $X$.
We say that $\cF$ is nef $($resp.\ strongly psef$\,)$ if there exists some modification
$\pi: \tilde{X} \to X$ such that $\pi^* \cF/\tors$ is a nef
$($resp.\ strongly psef$\,)$ vector bundle where $\tors$ means the torsion part.
\end{mydef}
Notice that for any further modification $\pi': \tilde{X}' \to \tilde{X}$, $\pi'^*(\pi^* \cF/\tors)=(\pi \circ \pi')^* \cF /\tors$ (in particular, further pull back is still a nef or strongly psef vector bundle).
In fact, for any morphism $\pi',\pi$, there exist natural surjective morphisms
$$(\pi \circ \pi')^* \cF=\pi'^*  \pi^* \cF \to \pi'^* ( \pi^* \cF /\tors) \to (\pi'^* ( \pi^* \cF /\tors))/\tors$$
which induces a surjection
$(\pi' \circ \pi)^* \cF /\tors \to (\pi'^* ( \pi^* \cF /\tors))/\tors$.
It is generic isomorphism on which $\cF$ is locally free.
Thus the kernel of the induced morphism is a torsion sheaf.
Since $(\pi' \circ \pi)^* \cF /\tors$ is torsion free, the morphism is also injective.

More generally, we show in the next remark that the definition is independent of the choice of the pull back. 
\begin{myrem}{\rm
By the work of \cite{Ros}, \cite{GR}, \cite{Rie}, for any torsion-free sheaf $\cF$ over a compact complex manifold, there exists a modification $\pi: \tilde{X} \to X$ such that $\pi^* \cF/\tors$ is a locally free sheaf (i.e.\ a vector bundle).
In the above definition, we say that $\cF$ is nef or strongly psef if $\pi^* \cF/\tors$ is nef or strongly psef.

Let us recall here Theorem 1.B.1 of \cite{Paun}.
Let $f: Y \to X$ be a surjective holomorphic map between compact complex manifolds. 
Let $\alpha$ be a cohomology class in the Bott-Chern cohomology class $H^{1,1}_{BC}(X, \C)$.
Then $\alpha$ is nef if and only if $f^* \alpha$ is nef.

For the vector bundle case, a modification $\sigma: \tilde{X} \to X$ induces a surjection $\tilde{\sigma}:\P(\sigma^* E) \to \P(E)$ where $E$ is a vector bundle over $X$.
The pull back of $\cO_{\P(E)}(1)$ under $\tilde{\sigma}$ is $\cO_{\P(\sigma^* E)}(1)$.
Thus $\sigma^* E$ is nef if and only if $\cO_{\P(\sigma^* E)}(1)$ is nef which is equivalent to say that $\cO_{\P( E)}(1)$ is nef, i.e. $E$ is nef.

Thus in the above definition, it is same to say that $\cF$ is nef if and only if for {\it every} modification $\sigma: \tilde{X} \to X$ such that $\sigma^* \cF/\tors$ is a vector bundle, $\sigma^* \cF/\tors$ is nef.

Similarly, let $f: Y \to X$ be a surjective holomorphic map between compact complex manifolds. 
Let $\alpha$ be a cohomology class in the Bott-Chern cohomology class $H^{1,1}_{BC}(X, \C)$.
Then $\alpha$ is strongly psef if and only if $f^* \alpha$ is strongly psef. 
The pull back of a strongly psef vector bundle $E$ under a modification $\sigma$ is psef if and only if $E$ itself is psef.
Once a smooth metric has been fixed on $E$, the singular metrics on $\cO_{\P(\sigma^* E)}(1)$ (resp.\ on $\cO_{\P( E)}(1)$) are identified with quasi-psh functions. Let us observe  that the push forward of a psh function with analytic singularities under a proper modification is still a psh function with analytic singularities.
The singular set of the pushed forward weight on $\cO_{\P( E)}(1)$ is the image of the singular set of the weight function on $\cO_{\P(\sigma^* E)}(1)$. 

More precisely, denote by $\tilde{\pi}: \P(\sigma^* E) \to \tilde{X}$ and $\pi: \P(E) \to X$ the projections. We have $\pi \circ \tilde{\sigma}=\sigma \circ \tilde{\pi}$.
For a simple blow-up with a smooth irreducible centre, the opposite of the cohomology class of the exceptional divisor has a smooth representative that is positive along the fibers of the projectivised normal bundle. From this, it is easy to see that exists a smooth form $\omega_E$ on $\tilde{X}$ such that $\sigma^*  \omega_X +\omega_E$ is a K\"ahler form on $\tilde{X}$, and $\{\omega_E\}=-\{[E]\}$ for a suitable combination $E=\sum\delta_jE_j$, $\delta_j\in\R_{>0}$  of the irreducible components $E_j$ of the exceptional divisor. Notice that $\{ \sigma_* \omega_E \}$ is the zero cohomology class.
Denote by $\varphi$ a quasi psh function on $\tilde{X}$ such that 
$$\omega_E=-[E]+i \d \dbar \varphi.$$
Assume that $\sigma^* E$ is strongly psef and let us use a reference metric $\sigma^* h_\infty$ induced by a smooth metric $h_\infty$ on $E$. Then there exist quasi-psh functions $\psi_\varepsilon$ with analytic singularities such that
$$i \Theta(\cO_{\P(\sigma^* E)}(1),\sigma^* h_\infty e^{-\psi_\varepsilon}) \geq - \varepsilon \tilde{\pi}^* (\sigma^*  \omega_X +\omega_E),$$
and $\sigma^* h_\infty e^{-\psi_\varepsilon-\tilde{\pi}^*\varphi}$ are singular metrics with analytic singularities on $\cO_{\P(\sigma^* E)}(1)$.
By taking the push-forward of the quasi-psh functions $\psi_\varepsilon+\varepsilon \tilde{\pi}^*\varphi$ under the modification $\tilde{\sigma}$, we get singular metrics $h_\varepsilon:=h_\infty e^{-\tilde{\sigma}_*(\psi_\varepsilon+\varepsilon\tilde{\pi}^*\varphi)}$ on $\cO_{\P( E)}(1)$ possessing analytic singularities and satisfying the condition
$$i \Theta(\cO_{\P( E)}(1),{ h}_\varepsilon) \geq - \varepsilon {\pi}^* \omega_X .$$
In the above definition, it is thus the same to say that $\cF$ is strongly psef if and only if for {\it every} modification $\sigma: \tilde{X} \to X$ such that $\sigma^* \cF/\tors$ is a vector bundle, $\sigma^* \cF/\tors$ is strongly psef.}
\end{myrem}
In fact, following the arguments in \cite{Paun} and \cite{DPS94}, we can prove a more general result.
\begin{mythm}
Let $f: Y \to X$ be a surjective holomorphic map between compact K\"ahler manifolds.
Let $E$ be a vector bundle over $X$.
Then $f^* E$ is strongly psef if and only if $E$ is strongly psef.
\end{mythm}
\begin{proof}
It is easy to see that $E$ is strongly psef implies that $f^*E$ is strongly psef. 
To prove the inverse direction, we use the Hironaka flattening theorem which shows the existence of a commutative diagram
$$
\begin{matrix}
Z&\kern-6pt\mathop{\longrightarrow}\limits^{\pi_2}
\kern-6pt & \tilde{X}\\
\noalign{\vskip3pt}  
\llap{\hbox{$\scriptstyle\pi_1$}}\Big\downarrow&  &
\Big\downarrow\rlap{\hbox{$\scriptstyle\sigma$}}\\
Y&\kern-6pt\mathop{\longrightarrow}\limits^f\kern-6pt& X
\end{matrix}
$$
where $Z$ is a compact K\"ahler complex space, $\pi_2$ a flat morphism (i.e.\ with equidimensional fibres) and $\sigma$ a composition of blow-ups of smooth centres.
In the previous remark, we prove that the pull back of a vector bundle under a blow-up of smooth center is strongly psef if and only if it is itself strongly psef.
The result will follow if we prove that the pull back of a vector bundle under a flat morphism is strongly psef if and only if it is itself strongly psef.
Intuitively, we would want to take the quasi-psh weight at any point to be the supremum of the quasi-psh weight on the pre-image of that point.
But this operation does not necessarily give the desired lower bound of curvature. In order to overcome this difficulty, we use a modified version of the
argument given in \cite{DPS94} Proposition 1.8, as follows.
\end{proof}
\begin{myprop}{\it
Let $f: Y \to X$ be a surjective holomorphic map with equidimensional fibres where $X$ is a compact K\"ahler manifold and $Y$ is a compact K\"ahler complex space.
Let $E$ be a vector bundle over $X$.
Then $f^* E$ is strongly psef if and only if $E$ is strongly psef.}
\end{myprop}
\begin{proof}
The proof is essentially the same as for Th\'eor\`eme 1.B.1 in~\cite{Paun} and Proposition 1.8 in~\cite{DPS94}.
We just outline the arguments with the necessary modifications.

We denote by the same symbol $f$ the induced map $\P(f^* E) \to \P(E)$.
Let $\alpha$ be the curvature form in the cohomology class $c_1(\cO_{\P( E)}(1))$ induced by some smooth metric on $E$.
Let $\psi_\varepsilon$ be quasi psh functions with analytic singularities on $\P(f^* E)$ such that
$$f^* \alpha+\frac{i}{2 \pi} \d \dbar \psi_\varepsilon \geq -\varepsilon \omega',
\quad\varepsilon>0,$$
for some K\"ahler form $\omega'$ on $\P(f^* E)$.
The existence follows from the definition of a strongly psef vector bundle (the definition of a strongly psef vector bundle is still valid for a compact K\"ahler complex space).

Denote by $p$ the dimension of fibres.
For every $y \in \P(f^* E)$ there exist local holomorphic functions $w_1, \cdots, w_p$ in a neighbourhood $U$ of $y$ such that $z \mapsto (f(z), w_1(z), \cdots, w_p(z))$ is a proper finite morphism from $U$ to a neighbourhood of $\{f(y)\} \times \{0\}$ in $\P(E) \times \C^p$.
Thus there exist local coordinates centered at $f(y)$ on $\P(E)$ such that
$$|F(z)-F(y)|^2+\sum_{1 \leq j \leq p}|w_j(z)|^2>0$$ 
on $\dbar U$, where $F= (F_1,\cdots,F_n)$ denote the local coordinate components of $f$.

Since $\P(f^* E)$ is compact, we can cover $\P(f^* E)$ by finitely many such sets $U_k$ centered at $y_k \in \P(f^* E)$, and find corresponding holomorphic functions $(w_1^{(k)}, \cdots, w_p^{(k)})$ on $\overline{U}_k$, as well as components $F^{(k)}$. Each $U_k$ can be supposed to be embedded as a closed analytic set of some open set in $\C^{N_k}$ with coordinates $(w_1^{(k)}, \cdots, w_p^{(k)},\cdots, w_{N_k}^{(k)})$
(i.e., we complete $(w_1^{(k)}, \cdots, w_p^{(k)})$ into a local coordinate system of $\C^{N_k}$).
By construction,
$$2 \delta_k:=\inf_{\d U_k}|F^{(k)}(z)-F^{(k)}(y_k)|^2+\sum_{1 \leq j \leq p}|w_j^{(k)}(z)|^2>0.$$
We can even suppose that the open sets
$$V_k:= \{z \in U_k; |F^{(k)}(z)-F^{(k)}(y_k)|^2+\sum_{1 \leq j \leq p}|w_j^{(k)}(z)|^2 <\delta_k\}$$
cover $\P(f^* E)$. Define for $z \in \overline{U_k}$
$$\lambda_\varepsilon^{(k)}(z):= \varepsilon^3 \sum_{1 \leq i \leq N_k} |w_i^{(k)}|^2-\varepsilon^2 (|F^{(k)}(z)-F^{(k)}(y_k)|^2+\sum_{1 \leq j \leq p}|w_j^{(k)}(z)|^2-\delta_k),$$
and for $x \in \P(E)$,
$$\varphi_\varepsilon:= \sup_{y \in f^{-1}(x) \cap \overline{U}_k} (\psi_{\varepsilon^4}(y)+\lambda_\varepsilon^{(k)}(y))$$
where the supremum is also taken with respect to $k$. The curvature condition is checked in the same way as in \cite{Paun} and \cite{DPS94}.

Let us observe that by using a regularization, one can assume that the quasi psh weight $\psi_\varepsilon$ is continuous (i.e.\ locally the weight is of the form $c \log \sum |g_j|^2+f$ where $f$ is continuous, and not just bounded).

By choosing $\varepsilon$ small enough, we get $\varphi_\varepsilon$ continuous with values in $[ -\infty, \infty[$.
In fact, for $\varepsilon$ small enough, $\lambda^{(k)}_\varepsilon$ is strictly negative on the boundary of $U_k$ and positive on $V_k$.
Thus the function $\Psi_\varepsilon(y):=\psi_{\varepsilon^4}(y)+\sup_{ y \in \overline{U}_k} \lambda_\varepsilon^{(k)}(y) $
is continuous on $Y$.
Since $\varphi_\varepsilon(x)=\sup_{y \in f^{-1}(x)} \Psi_\varepsilon(y)$, $\varphi_\varepsilon$ is continuous on $X$.

We now turn ourselves to the proof that $\varphi_\varepsilon$ has analytic singularities. Observe that $\varphi_\varepsilon$ has the same singularities as the function $\sup_{y \in f^{-1}(x)} \psi_{\varepsilon^4}(y)$ on $X$,
since the functions $\lambda_\varepsilon^{(k)}$ are bounded.
We claim the following more general fact: let $f: Y \to X$ be a proper morphism between complex spaces,
and $\varphi_\varepsilon$ be a quasi psh function with analytic singularities on $Y$, then the function
$$f_* \varphi(x):=\sup_{y \in f^{-1}(x)} \varphi(y)$$
has analytic singularities on $X$.
Here ``$\varphi$ is a quasi-psh function over a complex space'' means that $\varphi$ can be locally extended as a quasi-psh function to any open set of $\C^{N}$ in which $Y$ can be embedded as a closed analytic set;
that $\varphi$ has analytic singularities means for every $y \in Y$, there exists an open set on which $\varphi=c \log (\sum |g_i|^2)+f$, with holomorphic functions $g_i$ and a bounded function $f$. 
 
By Hironaka, there exists a modification $\sigma:\tilde{Y} \to Y$ such that $\tilde{Y}$ is smooth. By considering $f \circ \sigma$ and $\varphi \circ \sigma$, we are reduced to the case where $Y$ is smooth.
For every $x \in X$, we can cover $f^{-1}(x)$ by finite open sets $U_k$ such that the restriction of $\varphi$ to each open set is of the form $c\log\sum |g_i^{(k)}|^2+O(1)$, where $g_i$ are holomorphic functions on this open set and $O(1)$ is a bounded term.
There exists an open neighbourhood $V$ of $x$ such that $f^{-1}(V) \subset \cup_k U_k$.
For every $x \in V$,
$$f_* \varphi(x)=\sup_k \sup_{y \in f^{-1}(x) \cap U_k} \varphi(y).$$ 
Since a finite supremum of quasi-psh functions with analytic singularities still has analytic singularities, it is enough to show that $\sup_{y \in f^{-1}(x) \cap U_k} \varphi(y)$ has analytic singularities for every $k$.
Since we take a finite supremum, the bounded terms will remain bounded after taking the supremum, therefore we are only concerned with the logarithmic term in what follows.

Let $J_k$ be the maximal germ of ideal sheaf at $x$ such that $f^* J_k|_{U_k} \subset (g_i^{(k)})$ with respect to the inclusion relation.
(Here one may have to shrink the open set $U_k$, i.e.\ the inclusion is to be understood in the sense of germs at any point of $f^{-1}(x)$.)
Then the ideal $(g_i^{(k)})$ is generated by finitely many holomorphic functions that are either of the form $f^* h_\alpha^{(k)}$ for some holomorphic function germ at $x$, or of the form $f_\beta^{(k)}$ for some holomorphic function on $U_k$.
We claim that the zero set $V(f_\beta^{(k)})$ is not of the form $f^{-1}(f(V(f_\beta^{(k)})))$.
Otherwise, by Hilbert's Nullstensatz, $f_\beta^{(k)}$ is contained in the germ of pull back of the prime ideal sheaf vanishing on $f(V(f_\beta^{(k)}))$, contradicting the maximality of $J_k$. Therefore
$$\log(\sum_\alpha | g_\alpha^{(k)}|^2)(x)=\sup_{y \in f^{-1}(x) \cap U_k}\log(\sum_\alpha |f^* h_\alpha^{(k)}|^2+\sum_\beta |f_\beta^{(k)}|^2),$$
which also has analytic singularities.
\end{proof}
\begin{myrem}
{\rm
Observe that when the manifold $X$ is projective, there exists a subtlety in the definition of a strongly psef torsion free sheaf.
Recall that a torsion free sheaf $\cF$ over a projective manifold $X$ with an ample line bundle $A$ is called weakly positive in the sense of Nakayama (cf. \cite{Nak}) if for any $a \in \N$, there exists $b \in \N^*$ such that
 $(\mathrm{S}^{ab}  \cF)^{\lor \lor} \otimes  A^b$
is globally generated at some point (hence generically globally generated).  
Our definition of strongly psef torsion free sheaf implies that it is the weak positive torsion free sheaf in the sense of Nakayama, but not inversely in general.

First we show that if $\cF$ is a strongly psef torsion free sheaf, then it is weakly positive in the sense of Nakayama.
Let $\sigma: \tilde{X} \to X$ be a composition of blow-ups of the smooth centers such that the pull back of torsion free coherent sheaf $\sigma^* \cF/\tors$ is a strongly psef vector bundle.
Let $A$ be an ample line bundle over $X$.
For $b$ large enough, 
$\sigma^* A^b -E$ is an ample line bundle over $\tilde{X}$ where $E$ is the exceptional divisor.
$\mathrm{S}^{ab} (\sigma^* \cF/\tors) \otimes \sigma^* A^b \otimes \cO(-E)$ is generically globally generated over $\tilde{X}$ by possible larger $b$ and by changing $\cO(-E)$ by its multiple. 
It is from the assumption that $\sigma^* \cF/\tors$ is a strongly psef vector bundle over $\tilde{X}$.
By tensoring the canonical section of the line bundle $\cO(E)$, $\mathrm{S}^{ab} (\sigma^* \cF/\tors) \otimes \sigma^* A^b$ is generically globally generated over $\tilde{X}$.
Thus the same holds for $(\mathrm{Sym}^{ab} \cF)^{\lor \lor} \otimes A^b$ over $X$ by the natural isomorphism $(\mathrm{Sym}^{ab} \cF)^{\lor \lor} \otimes A^b \to [\sigma_*(\mathrm{Sym}^{ab} (\sigma^* \cF/ \tors) \otimes \sigma^* A^b)]^{\lor \lor}$.

To indicate the subtlety, we use the same notations as above. 
For the inverse direction, we hope to show that 
$(\mathrm{S}^{ab} (\sigma^* \cF/\tors))^{\lor \lor} \otimes \sigma^* A^b$ is generically globally generated over $\tilde{X}$ for large $b$ from the fact that $(\mathrm{S}^{ab}  \cF)^{\lor \lor} \otimes  A^b$ is generically globally generated over $X$ for large $b$.
Let $S$ be the analytic set of codimension at least 2 in $X$ such that $\sigma: \tilde{X} \setminus E \to X\setminus S$ is biholomorphic.
But the global sections
$$H^0(X,(\mathrm{S}^{ab}  \cF)^{\lor \lor} \otimes  A^b) \cong H^0(X \setminus S,(\mathrm{S}^{ab}  \cF)^{\lor \lor} \otimes  A^b) \cong H^0(\tilde{X} \setminus E,\mathrm{S}^{ab} (\sigma^* \cF/ \tors) \otimes  \sigma^* A^b)$$
do not necessarily extend over $\tilde{X}$ even seen as a section of $\mathrm{S}^{ab} (\sigma^* \cF/\tors) \otimes  \sigma^* A^b \otimes \tilde{A}$ where $\tilde{A}$ is an arbitrary ample line bundle over $\tilde{X}$.
The reason is that
the sections may have essentially singularity along $E$.

A typical example is the following. Let $S$ be an analytic set of codimension at least two over a projective manifold $X$ and let $\cI_S$ be the ideal sheaf associated to $S$.
The bidual of $\cI_S$ is $\cO_X$ as well as all the symmetric power.
Let $U$ be any open set of $X$. We have that for any $m$ and any vector bundle $E$
$$H^0(U, (S^m \cI_S)^{\lor \lor} \otimes E)\cong H^0(U \setminus S, (S^m \cI_S)^{\lor \lor} \otimes E)=H^0(U\setminus S, E)\cong H^0(U,E)$$
where the last equality follows from the Hartogs theorem.
As a consequence, for any ample divisor $A$ and any $a \in \N^*$,
$(S^{ab} \cI_S)^{\lor \lor} \otimes A^{ b}$
is globally generated for $b$ large enough.
In particular, $\cI_S$ is weakly positive in the sense of Nakayama.

However, observe that $\cI_S$ has some ``negativity" along $S$, even if it is weakly positive in the sense of Nakayama.
It can be seen as follows.
Let $\sigma: \tilde{X} \to X$ be a composition of blow-ups with smooth centers such that $\sigma^* \cI_S/ \tors=\cO_{\tilde{X}}(-E)$ where $E$ is an effective divisor supported in the exceptional divisor.
By the definition of strongly psef torsion free sheaf, if $\cI_S$ is strongly psef, $\sigma^* \cI_S /\tors$ should be a psef vector bundle since it is a line bundle.
But it is not the case which means that $\cI_S$ is not strongly psef.
In other words, our definition of strong psef torsion free sheaf is reasonable which forbids the above kind of negativity which will appear in some birational models.
}
\end{myrem}
Like the bundle case, the strongly psef torsion-free sheaf is stable under the usual algebraic operations, with the consideration of taking torsion free part.
\begin{myex}{\rm (The pull back of a torsion free sheaf is not necessarily torsion free)

According to the knowledge of the author, this example can be found in \cite{GR}.
Let $X$ be the blow up of the origin of $\C^2$ with $\pi: X \to \C^2$.
Let $(x,y)$ be the coordinate of $\C^2$.
The maximal ideal at the origin can be resolved by the Koszul complex
$$\cO_{\C^2} \xrightarrow{(-y,x)} \cO_{\C^2}^{\oplus 2} \to \mathfrak{m}_0 \to 0$$
where the second arrow sends $(f,g)$ to $xf+yg$.
The pull back is right exact which induces the exact sequence
$$\cO_{X} \xrightarrow{(-v,uv)} \cO_{X}^{\oplus 2} \to \pi^* \mathfrak{m}_0= \mathfrak{m}_0 \otimes_{\cO_{\C^2}} \cO_X \to 0$$
where in local coordinates $\pi(u,v)=(uv,v)$ and the the second arrow sends $(f,g)$ to $f \otimes x +g \otimes y$.
We denote the second arrow as $\varepsilon$.
We claim that  $\varepsilon(-1,u)$ is not a zero element in $\pi^* \mathfrak{m}_0$.
Otherwise, $(-1,u)$ is in the kernel of $\varepsilon$ which by exactitude is of the form $(-vf,uvf)$ for some $f$. Contradiction.

On the other hand $(-1,u)$ is torsion in $\pi^* \mathfrak{m}_0$ since $v \varepsilon (-1,u)=-v \otimes x + vu \otimes y=(-v \pi^* x+vu \pi^* y) \otimes 1=0$.
Consider the composition $\cO_{X}^{\oplus 2}/ \mathrm{Ker} (\varepsilon) \to \pi^* \mathfrak{m}_0 \to \cO_X $.
The image is $ \cI_E$ where $\cI_E$ is the ideal sheaf associated to the exceptional divisor $E$.
The first morphism is in fact isomorphism.
In local coordinates the composition sends $(f,g)$ to $uvf+vg$.
The kernel is $\cO_X(-1,u)$ which is torsion.
We have isomorphism between $\cO_{X}^{\oplus 2}/ \mathrm{Ker} (\varepsilon)$ modulo this kernel and $\cI_E$.
Thus the image under $\varepsilon$ (i.e. $\cO_X \varepsilon(-1,u)$) gives all the torsion elements.
In other words, $ \pi^* \mathfrak{m}_0 / \tors \cong \cO_X(-E)$.

In fact, the morphism $u: \cO_{\C^2} \to  \cO_{\C^2}^{\oplus 2}$ induces a meromorphic map from $\C^2 $ to $\mathrm{Gr}(1,2)$ which sends $z$ to the image of $u(z)$.
The meromorphic map induces a holomorphic map from the blow up of the origin to $\mathrm{Gr}(1,2)$ which resolves the indeterminacy set of the meromorphic morphism.
The total space of $\cO_{\P^2}(-1)$ is also the blow-up of $\C^2$ at the origin with the natural projection $\tau : X \to \P^2$.
The pull back of the tautological line bundle over $\mathrm{Gr}(1,2)$ 
admits exact sequence
$$\cO_X^{\oplus 2} = \tau^* \cO_{\P^2}^{\oplus 2} \to \tau^* \cO_{\P^2}(1) \to 0. $$
The image of the kernel of $\varepsilon$ in $\tau^* \cO_{\P^2}(1)$ is 0.
Thus we have factorisation
$\cO_X^{\oplus 2}/ \mathrm{Ker} (\varepsilon) \cong \pi^* \mathfrak{m}_0  \to \tau^* \cO_{\P^2}(1) \to 0$.
The kernel of the factorisation is supported in the exceptional divisor which is thus torsion.
In conclusion, we have isomorphism $ \pi^* \mathfrak{m}_0 / \tors \cong \tau^* \cO_{\P^2}(1)$.
This shows in this special case how to find a modification such that the pull back of a torsion free sheaf is locally free modulo torsion.
This construction was generalised in  \cite{Ros}, \cite{GR}, \cite{Rie}.
(We recall it briefly in Lemma 7.)}
\end{myex}
\begin{myex}{\rm (The symmetric and wedge power of torsion free sheaves are not necessarily torsion free)

Consider the maximal ideal sheaf $\mathfrak{m}_0$ in $X=\C^2$.
The wedge power $\bigwedge^2 \mathfrak{m}_0$ is supported at the origin,
however $z_1 \wedge z_2$ is a non zero element of germ of $\wedge^2 \mathfrak{m}_0$ at the origin.

For the symmetric powers, let us first recall the following important theorem in \cite{Mic64} (cf. also Theorem 3 of \cite{LaB}).
Let $A$ be a domain, $M$ be a finitely generated $A$-module.
Then $\oplus_{i \geq 0} S^i M$ is a domain if and only if $S^i M$ is torsion free, for all $i \geq 0$.
To give a concrete example, consider a surjection from a holomorphic vector bundle $E$ to a torsion free sheaf $\cF$ over a compact manifold $X$.
Then $\P(\cF)$ is a closed analytic set of $\P(E)$.
If $\P(\cF)$ is not irreducible, by the above theorem there exists $i >0$ such that $S^i \cF$ is not torsion free.}
\end{myex}
We summarise the algebraic properties of strongly psef torsion free sheaf in the following propositions.
\begin{myprop}{\it
Let $\cF$ be a torsion free sheaf over a compact K\"ahler manifold $(X, \omega)$.
The following properties are equivalent.

(1) $\cF$ is strongly psef.

(2) For any $m \in \N^*$, $S^m \cF$ modulo its torsion part is strongly psef.

(3) There exists $m \in \N^*$ such that $S^m \cF$ modulo its torsion part is strongly psef.}
\end{myprop}
\begin{proof}
(1) implies (2) as follows. Let $\sigma$ be a modification of $X$ such that $\sigma^* \cF/\tors$ and $\sigma^*( S^m \cF/\mathrm{Tors})/\tors$ are vector bundles where $\mathrm{Tors}$ means the torsion part.
We have a surjection
$$\sigma^* S^m \cF \cong  S^m \sigma^* \cF \to \sigma^* ( S^m \cF/\mathrm{Tors}).$$
It induces a surjection
$$  S^m (\sigma^* \cF/\tors) \to \sigma^* ( S^m \cF/\mathrm{Tors})/\tors.$$
The reason is as follows.
Recall that there exists an exact sequence
$$\tors \otimes S^{m-1} \sigma^* \cF \to S^m(\sigma^* \cF) \to S^m(\sigma^* \cF /\tors)\to 0$$
induced by
$$0 \to \tors \to \sigma^* \cF \to \sigma^* /\tors \to 0.$$
The image of $ \tors \otimes S^{m-1} \sigma^* \cF$ consists of torsion elements, and induces the morphism $  S^m (\sigma^* \cF/\tors) \to \sigma^* ( S^m \cF/\mathrm{Tors})/\tors.$
Thus Corollary 2 and (1) of Corollary 3 implies that $\sigma^* ( S^m \cF/\mathrm{Tors})/\tors$ is a strongly psef vector bundle.

(3) implies (1) as follows. With the same notation, the above surjection is in fact an isomorphism since both sides have the same rank.
Thus Corollary 2 implies that $\sigma^* \cF/\tors$ is a strongly psef vector bundle.
\end{proof}
\begin{mydef}
Let $(X, \omega)$ be a compact K\"ahler manifold, $\cF$ be a torsion free sheaf on $X$ and $D$ be a $\Q-$Cartier divisor. 
Then $\cF \langle D \rangle$ is said to be $\Q-$twisted strongly pseudo-effective $($ $\Q-$twisted strongly psef for short$)$ if $S^m \cF/\mathrm{Tors} \otimes \cO_X(mD)$ is strongly psef for some $($hence any by Proposition 3$)$ $m >0$ such that $\cO_X(mD)$ is a line bundle.
\end{mydef}
\begin{myprop}{\it
\strut
\begin{enumerate}
\item A torsion free quotient sheaf of a strongly psef torsion free sheaf is strongly psef.
\item A direct summand of strongly psef torsion free sheaf is strongly psef.
\item A direct sum of strongly psef torsion free sheaves is strongly psef.
\item A tensor product $($or Schur functor of positive weight$\,)$ modulo its torsion part
  of strongly psef torsion free sheaves is strongly psef.
\end{enumerate}}
\end{myprop}
\begin{proof}
Let $\cF \to \cQ$ be a surjective morphism of torsion free sheaves with $\cF$ strongly psef over $X$.
Let $\sigma: \tilde{X} \to X$ be a modification such that $\sigma^*F/\tors$,$\sigma^* \cQ/\tors$ are vector bundles.
By assumption $\sigma^*F/\tors$ is a  strongly psef vector bundle.
$\sigma^*$ is right exact which induces surjection
$\sigma^*F/\tors \to \sigma^* \cQ/\tors$
passing to quotient.
Thus $\sigma^* \cQ/\tors$ is a quotient bundle of $\sigma^* \cF/\tors$
Using Proposition 1 we can conclude that $\sigma^* \cQ/\tors$ is strongly psef.
The other conclusions are similar and can be obtained in a formal manner.
\end{proof}
A natural operation for torsion free sheaves consists of taking the bidual.
The relationships between a torsion free sheaf and its bidual will be stated in the next propositions.
The following example indicates some of the occurring phenomena.
\begin{myex}{\rm 
Let $D$ be a smooth effective divisor over a compact K\"ahler manifold $X$ with canonical section $s_D$.
We have generic surjective sheaf morphism
$$\alpha: \cO_X^{\oplus 2} \to \cO_X(D) \oplus \cO_X(2D)$$
induced by global section $(s_D,s_D^2)$.
Then $\det(\alpha)\cong \cO_X(3D)$ has a global section $s_D^3$.
The division by this global section induces a bimeromorphic map between the total spaces of $(\cO_X(D) \oplus \cO_X(2D))^* $ and $\cO_X^{\oplus 2}\otimes \det(\alpha)^*$.
Since $\cO_X^{\oplus 2}$ is strongly psef, there exists a global (quasi-)psh function on the total space of its dual.
Pairing with $s_D^3$ induces a global (quasi-)psh function on the total space of $\cO_X^{\oplus 2}\otimes \det(\alpha)^*$ which induces a (quasi-)psh function on the total space of $(\cO_X(D) \oplus \cO_X(2D))^* $ outside a smooth divisor.
We claim that this (quasi-)psh function extends across the divisor by boundedness from above.
In particular, $(\cO_X(D) \oplus \cO_X(2D))$ is strongly psef.

For example, locally consider the psh function on the total space of $\cO_X^{\oplus 2}\otimes \det(\alpha)^*$
$$\varphi(z, \xi_1,\xi_2):=\log(|z|^6(|\xi_1|^2 +|\xi_2|^2))$$
where $\alpha(z, \xi_1, \xi_2)=(z, z \xi_1, z^2 \xi_2)$.
The induced psh function outside the divisor $\{z=0\}$ is given by
$$\varphi(z, \xi_1,\xi_2):=\log(|z|^6(|\xi_1/z|^2 +|\xi_2/z^2|^2)),$$
which is bounded from above near the divisor and can thus be extended across the divisor.}
\end{myex}
\begin{myprop}{\it
Let $\cE, \cF$ be two torsion free sheaves over a compact K\"ahler manifold $(X, \omega)$.
Let $\alpha: \cE \to \cF$ be a morphism of sheaves which is an isomorphism over a Zariski open set $X \setminus A$.
Assume that $\cE$ is strongly psef.
Then $\cF$ is strongly psef.}
\end{myprop}
\begin{proof}
Let $\sigma$ be a modification of $X$ such that $\sigma^* \cE/\tors$ and $\sigma^* \cF/ \tors$ are locally free and $\sigma^* \cE/ \tors$ is strongly psef.
We can assume that $\sigma^* \alpha$ is an isomorphism outside a divisor $E$.
Then $\det(\sigma^* \alpha)$ is an effective divisor supported in $E$.
Division by this global section induces bimeromorphic map between the total spaces of $(\sigma^* \cF/ \tors)^* $ and $(\sigma^* \cE /\tors)^* \otimes \det(\sigma^* \alpha)^*$.
By Proposition 1, the fact that $\sigma^* \cE / \tors $ is strongly psef
implies the existence of quasi-psh functions with analytic singularities on the total space of the symmetric powers of $(\sigma^* \cE/ \tors)^*$.
Pairing with the canonical section of $\det(\sigma^* \alpha)$ induces  global (quasi-)psh functions on the total space of $S^m(\sigma^* \cE/ \tors \otimes \det(\alpha))^*$.
We denote these quasi-psh functions by $w_m$.
The functions $w_m$ induce  quasi-psh function on the total space of $(S^m \sigma^* \cF/ \tors)^* $ outside the divisor $E$.
We claim that these quasi-psh functions extend across all the irreducible components of the divisor $E$ by boundedness from above.
In particular, by Proposition 1, $\sigma^* \cF/ \tors$ is a strongly psef vector bundle.

The claim is proven by a local coordinate calculation.
In local coordinate $\sigma^* \alpha(z, \xi)=(z, A(z)\xi)$
where 
$A(z)$ a matrix of holomorphic functions.
Locally
$$w_m(z, \xi)=\log(\sum_j |B_j(z)\xi|^2)+O(1)+\log(|\det(A(z))|^2)$$
where $B_j(z)$ are matrices of holomorphic functions.
The induced quasi-psh functions outside the divisor $E$ over $(S^m \sigma^* \cF / \tors)^*$ are of the form
$$\tilde{w}_m(z, \xi)=\log(\sum_j |B_j(z)A^{-1}(z)\xi|^2)+O(1)+\log(|\det(A(z))|^2).$$
Since the inverse is given by the co-adjoint of the matrix divided by its determinant,
$\tilde{w}_m$ is locally bounded from above near the divisor.
\end{proof}
The inverse direction is in general false. 
To 
get  a 
counter-example, 
we 
consider 
an 
inclusion $\cI_A \to \cO_X$ where $A$ is an analytic set of codimension at least 2.
Then $\cI_A$ is not strongly psef, while $\cO_X$ is.
However, the inclusion is an isomorphism over $X \setminus A$.
\begin{myprop}{\it
Let
$$0\to \cS \to \cF \to \cQ \to 0$$
be an exact sequence of torsion free sheaves.
If $\cF, (\det(\cQ))^{-1}$ are strongly psef and $\cS$ is reflexive, then $\cS$ is strongly psef.}
\end{myprop}
\begin{proof}
We have $S= \bigwedge^{s-1}S^* \otimes \det S$ where $s$ is the rank of $S$ outside an analytic set $A$ of codimension at least 2. 
Assume that all three sheaves are locally free outside $A$.
We have a surjective bundle morphism over $X \setminus A$
$$\bigwedge^{r-s-1} \cF/\mathrm{Tors} \otimes (\det \cQ)^{-1} \to S$$
where $r$ is the rank of $\cF$.
Since $\cS$ is reflexive (hence normal), the morphism extends as a morphism of sheaves over $X$.
By (4) of Proposition 4, $\bigwedge^{r-s-1} \cF/\mathrm{Tors} \otimes (\det \cQ)^{-1}$ is strongly psef.
By (1) of Proposition 4,
the image of this sheaf morphism is strongly psef.
Since the image ans $\cS$ are isomorphism over $X\setminus A$, by Proposition 5, 
$S$ is strongly psef.
\end{proof}
\begin{myprop}{\it
Let $\cF$ be a strongly psef torsion free sheaf of rank $r$.
Then $\det(\cF)$ is a psef line bundle.}
\end{myprop}
\begin{proof}
By (4) of Proposition 4, $\wedge^r \cF/\mathrm{Tors}$ is strongly psef.
Since $\wedge^r \cF/\mathrm{Tors}$ and $\det(\cF)$ is generic isomorphism, by Proposition 5, $\det(\cF)$ is a psef line bundle. 
\end{proof}
\begin{myprop}{\it
Let $\cF$ be a strongly psef torsion free sheaf with $c_1(\cF)=0$.
Then $\cF^*$ is a strongly psef reflexive sheaf.}
\end{myprop}
\begin{proof}
The fact that $\cF^*$ is reflexive is purely algebraic.
Outside an analytic set of codimension at least 2, $\cF$ is locally free.
Over this open set, we have an isomorphism
$$\wedge^{r-1} \cF/ \tors \otimes (\det(\cF))^{-1} \to \cF^*.$$
Since $\cF^*$ is reflexive, this morphism extends across the analytic set.
By (4) of Proposition 4, the left hand term is strongly psef.
Thus the image is strongly psef.
Moreover, the fact that we have a generic isomorphism implies that $\cF^*$ is strongly psef.
\end{proof}
\begin{mylem}
Let $\cF$ be a strongly psef torsion free sheaf with $c_1(\cF)=0$ over $X$.
Let $\sigma: \tilde{X} \to X$ be a modification such that both $\sigma^* \cF/ \tors$ and $\sigma^* \cF^*/ \tors$ are locally free.
Then $c_1(\sigma^* \cF/ \tors)=0$.
\end{mylem}
\begin{proof}
There exists natural morphism
$$\sigma^* \cF^*/ \tors \to (\sigma^* \cF/ \tors)^*$$
which is generic isomorphism.
Note that $(\sigma^* \cF/ \tors)^* \cong (\sigma^* \cF)^*$ by Corollary 4.9 Chap. V \cite{Kob}.
The above morphism is induced by $\sigma^* \cF^* \to (\sigma^* \cF/ \tors)^* \cong (\sigma^* \cF)^*$
under which the torsion part is in the kernel since $(\sigma^* \cF)^*$ is torsion free.
By Proposition 5, $(\sigma^* \cF/ \tors)^*$ is strongly psef.
In other words, both $\sigma^* \cF/ \tors$ and its dual are strongly psef vector bundle which infers that its first Chern class is 0.
\end{proof}
We can now prove the main result of this section assuming the main theorem (whose proof is independent of the main result of this section). For the convenience of readers, we recall here the construction of reduction of torsion free sheaf to the vector bundle case modulo torsion.
For a complete proof, we recommend the paper of \cite{Ros}.
\begin{mylem}
Let $\cF$ be a torsion free sheaf of generic rank $r$ over $X$ a complex manifold.
There exists some modification $\sigma: \tilde{X} \to X$ such that $\sigma^* \cF/ \tors$ is locally free.
Then for every $i =1,2$, the Chern class $c_i(\cF)$ is well defined in the Bott-Chern cohomology group $H^{i,i}_{BC}(X, \C)$.

If $X$ is compact K\"ahler and $\cF$ is a reflexive sheaf, these two Chern classes can be represented by normal currents $($in fact differences of two closed positive currents$\,)$.
 \end{mylem}

\begin{proof}

Cover $X$ by Stein open sets $U_\alpha$.
On each $U_\alpha$, there exists an exact sequence
$$\cO_{U_\alpha}^{\oplus M_\alpha} \to \cO_{U_\alpha}^{\oplus N_\alpha} \to \cF|_{U_\alpha} \to 0$$
which induces a meromorphic map
$$f_\alpha: U_\alpha\dasharrow \Gr(r, N_\alpha) .$$
The maps $\cO_{U_\alpha}^{\oplus M_\alpha} \to \cO_{U_\alpha}^{\oplus N_\alpha}$ are locally given as holomorphic matrices $A_\alpha(z)$ which are of constant rank over Zariski open sets, and $f_\alpha$ sends $z$ to the image of $A_\alpha(z)$. Let $\hat{U}_\alpha$
be the graph of this map $\hat{f}_\alpha:\hat{U}_\alpha\to
\Gr(r, N_\alpha)$ be the corresponding morphism (given by the second projection
of the graph). The $\hat{U}_\alpha$ glue into a complex space $\hat{X}$
sitting over $X$, and by Hironaka, we can find a
modification $\sigma: \tilde{X} \to
\hat{X}\to X$ such that $\tilde{X}$ is smooth and
$\sigma^* \cF/\tors$ is a vector bundle (the pull-back to $\hat{X}$ comes
locally from the tautological quotient bundle $Q_\alpha$ of $\Gr(r, N_\alpha)$ generically, hence is already a vector bundle generically).
It can be shown that the surjection $\sigma^* \cF \to Q_\alpha$ which is in fact generic isomorphism. 
This infers in particular that the kernel is torsion and isomorphism $\sigma^* \cF/ \tors \to Q_\alpha$.
We equip $Q_\alpha$ with a smooth metric
(e.g.\ the standard one coming from a Hermitian structure on~$\C^{N_\alpha}$)
and use a partition of unity to endow $\sigma^* \cF/ \tors$ with a smooth metric $h$.
Then the Chern forms $c_i(\sigma^* \cF/ \tors,h)$ associated with the curvature tensor
represent the Chern classes $c_i(\sigma^* \cF/\tors)$ in Bott-Chern cohomology
on $\tilde{X}$.
We define the Chern classes $c_i( \cF/ \tors)$ in
Bott-Chern cohomology on $X$ to be
the direct images $\sigma_*c_i(\sigma^* \cF,h)$ for $i=1,2$ as in Lemma 7.
(Notice that in Lemma 7, we work with the de Rham cohomology. 
By the work of \cite{Gri} and the result in \cite{Wu19}, the same formula holds in the complex Bott-Chern cohomology.)
It is well known that these
classes are independent of the choice of the metric~$h$.

Assume now that $X$ is a compact K\"ahler manifold.
Thus $\tilde{X}$ is also a compact K\"ahler manifold.
Let $\omega$ be a smooth K\"ahler form on $\tilde{X}$.
Then for $C$ large enough, $c_i(\sigma^* \cF/ \tors,h)$ can be written as difference of two positive forms $c_i(\sigma^* \cF / \tors,h)+C \omega^i$ and $C \omega^i$.
The second statement holds by taking direct images of these positive forms.
\end{proof}

\begin{myprop}{\it
Let $\cF$ be a nef reflexive sheaf over a compact K\"ahler manifold $(X, \omega)$ with $c_1(\cF)=0$. 
 Then $\cF$ is a nef vector bundle.}
\end{myprop}
\begin{proof}
The following proof derived from Theorem 1.18 of \cite{DPS94} (cf. also \cite{Deng16}) is recommended to the author by Demailly.
Let us recall the statement of this theorem.

{\it Let $E$ be a numerically flat vector bundle over a compact K\"ahler manifold $(X,\omega)$.
Then there exists a filtration of $E$
$$0=E_0 \subset E_1 \subset \cdots \subset E_p=E$$
by vector subbundles such that the quotients $ E_k/E_{k-1}$ are hermitian flat, i.e.
given by unitary representations $\pi_1(X) \to U(r_k)$.}

Since $\cF$ is a nef reflexive sheaf with $c_1(\cF)=0$,
there exists a modification such that $\sigma: \tilde{X} \to X$ such that $\sigma^* \cF/ \tors$ is a nef vector bundle with vanishing first Chern class by Lemma 8.
By the above theorem, there exists a filtration of $\sigma^* \cF / \tors$ 
$$0=\tilde{E}_0 \subset \tilde{E}_1 \subset \cdots \subset \tilde{E}_p=\sigma^* \cF/ \tors$$
by vector bundles over $\tilde{X}$ such that $\tilde{ E}_k/\tilde{E}_{k-1}$ are hermitian flat.

We claim that $\tilde{ E}_k/\tilde{E}_{k-1}=\sigma^* (E_k/E_{k-1})$ for some vector bundle $E_k/E_{k-1}$ over $X$ for each $k$.
(For the moment, $E_k/E_{k-1}$ is just a notion, not the quotient of two vector bundles over $X$.
But it is the case which is proven in the next paragraph.)
The reason is as follows.
$\sigma_*: \pi_1(\tilde{X}) \to \pi_1(X)$ is an isomorphism since we can assume that $\sigma$ is composition of a sequence of blows-up of smooth centres and as a CW complex a blow-up of smooth center changes skeleton of (real) codimension at least 2 which preserves the fundamental group. 
Thus we have unitary representations $\pi_1(X) \to U(r_k)$ which proves the claim.

Let $A$ be the analytic set such that $\cF$ is locally free over $X\setminus A$.
Since $\cF$ is reflexive, $A$ is of codimension at least 3 in $X$.
Without loss of generality, we can assume that $\sigma$ induces an isomorphism between $\sigma^{-1}(X \setminus A)$ and $X \setminus A$.
Thus we have extension of vector bundles over $X \setminus A$ 
$$0 \to E_{k-1}|_{X \setminus A} \to E_{k}|_{X \setminus A} \to E_{k-1}/E_{k-1}|_{X \setminus A} \to 0$$ 
where $E_k$ are a priori vector bundles defined over $X \setminus A$.
By Lemma 4, the extensions extend across $A$.
Thus there exist vector bundles $E_k$ over $X$ which are the extensions of $E_{k-1}$ and $E_k /E_{k-1}$.

By construction, we have isomorphism $\cF|_{X \setminus A}\cong E_p|_{X \setminus A}$.
Since $\cF$ is reflexive, we have isomorphism $\cF\cong E_p$ over $X$.
In particular, $\cF$ is a vector bundle. 
\end{proof}
\section{Segre forms}
In this section, 
we are interested in the following problem. 
Assume that $E$ is a holomorphic vector bundle of rank $r$ over a compact K\"ahler manifold $(X, \omega)$.
Can one find a $(k,k)$-closed positive current in the Segre class $s_k(E):=\pi_*(c_1(\cO_{\P(E)}(1))^{k+r-1})$? We have to point out that a similar construction is made in \cite{LRRR}, based on Demailly's improvement (\cite{Dem92}) of the Bedford-Taylor theory (\cite{BT}) of Monge-Amp\`ere  operators.
The authors define the corresponding current as a limit of smooth forms induced from local smooth regularizations of the metric given in \cite{Rau}.
Compared to theirs, our construction has the advantage that we define the relevant current as a limit of currents defined by Monge-Amp\`ere operators without necessarily employing a regularizing sequence. In that way, we are still in a position to estimate the Lelong number of the limiting Segre current in terms by the Lelong number of the approximating sequence of weights. On the other hand,
in the case of \cite{LRRR}, the approximation is given by smooth forms, hence the Lelong number of the approximation forms is identically zero, and one does not a priori obtain any information on the Lelong number of the limiting current.
The Lelong number estimate will be necessary in the next section.

In particular, starting from a singular metric with analytic singularities on $\cO_{\P(E)}(1)$, the construction yields a singular metric on $\det(E)$ which is unique up to a constant and, as a consequence, the curvature of the induced metric of $\det(E)$ is uniquely determined by the curvature of the metric on $\cO_{\P(E)}(1)$.

To start with, we state some results of pluripotential theory.
Some of this material is not essentially needed in the construction, but it provides intuition for a few arguments. The following statement is an improvement by Demailly of the Bedford-Taylor theory (\cite{BT}) of Monge-Amp\`ere  operators.

\begin{mylem}[Proposition 10.2 \cite{Dem93}]\strut\\
Let $\psi$ be a plurisubharmonic function on a $($non necessarily compact$\,)$ complex manifold $X$ such that $\psi$ is locally bounded on $X \setminus A$, where $A$ is an analytic subset of $X$ of codimension $\geq p+ 1$ at each point.  
Let $\theta$ be a closed positive current of bidimension $(p,p)$.

Then $\theta \wedge i \d \dbar \psi$ can be defined in such a way that $\theta \wedge i \d \dbar \psi=\lim_{\nu \to \infty} \theta \wedge i \d \dbar \psi_{\nu}$ in the weak topology of currents, for any decreasing sequence $(\psi_\nu)_{\nu \geq 1}$ of plurisubharmonic functions converging to $\psi$. Moreover, at every point $x \in X$ we have
$$\nu\Big(\theta \wedge \frac{i}{\pi} \d \dbar \psi,x \Big)
\geq \nu(\theta,x) \nu(\psi,x).$$
\end{mylem}
\begin{myprop}{\it
Let $T$ be a $(k,k)$-closed positive current in the cohomology class $\alpha$, over a  compact K\"ahler manifold $(X, \omega)$.
Let $U$ be a coordinate open set of $X$ such that on $U$, 
$$C^{-1} \omega \leq \frac{i}{2\pi} \d \dbar |z|^2 \leq C \omega.$$
Then for any $r_0>0$ and for any $x \in U$ with $d(x, \d U)\geq r_0$ with respect to the Euclidean metric in the coordinate chart, we have for $r \leq r_0$
$$\frac{1}{r^{2n-2k}} \int_{B(x,r)} T \wedge \omega^{n-k} \leq \frac{C^{2n-2k}}{r_0^{2n-2k}} (\alpha \cdot \{\omega\}^{n-r}) .$$
Here $ (\alpha \cdot \{\omega\}^{n-r})$ is the intersection product of cohomology classes.}
\end{myprop}
\begin{proof}
It is enough to prove that 
$$\frac{1}{r^{2n-2k}} \int_{B(x,r)} T \wedge
\Big(\frac{i}{2\pi} \d \dbar |z|^2\Big)^{n-k}
\leq \frac{C^{n-k}}{r_0^{2n-2k}} (\alpha \cdot \{\omega\}^{n-r}) .$$
By a basic observation of Lelong in \cite{Lel}, the left hand term is a increasing function with respect to $r$.
Thus we have
$$\frac{1}{r^{2n-2k}} \int_{B(x,r)} T \wedge
\Big(\frac{i}{2\pi} \d \dbar |z|^2\Big)^{n-k}
\leq \frac{1}{r_0^{2n-2k}} \int_{B(x,r_0)} T \wedge
\Big(\frac{i}{2\pi} \d \dbar |z|^2\Big)^{n-k}.$$
However, the right hand term is at most
$$\frac{1}{r_0^{2n-2k}} \int_{B(x,r_0)} T \wedge (C \omega)^{n-k} \leq  \frac{C^{2n-2k}}{r_0^{2n-2k}} (\alpha \cdot \{\omega\}^{n-r}) $$
since $T$ is a positive current.
\end{proof}
We will need the following standard local parametrization theorem for analytic sets.
\begin{mylem}[local parametrization theorem, cf.\
e.g.\ Theorem 4.19, Chap. II \cite{agbook}]\strut\\
Let $\cI$ be an ideal in $\cO_n$, let $A=V(\cI)$ and $A_j$ be the irreducible components of $A$ whose dimension is equal to the dimension of $A$. For every $j$ and $d=d_j=\dim A_j$, there exists a generic choice of coordinates
$$(z',z'')=(z_1, \cdots,z_d; z_{d+1}, \cdots,z_n) \in \Delta' \times \Delta''$$
such that the restriction of the canonical projection to the first component $\pi_j: A_j \cap (\Delta' \times \Delta'') \to \Delta'$ is a finite and proper ramified cover, which moreover yields an \'etale cover $A_{j} \cap \pi^{-1}(\Delta' \setminus S) \to \Delta' \setminus S$, where $S$ is an analytic subset in $\Delta'$.
\end{mylem}
\begin{mylem}
Let $A$ be a compact analytic subset of a complex manifold~$M$. Assume that $\dim_{\C} A=d$ and $\dim_{\C} M=n$.
Let $(W_\nu)$ be relatively compact coordinate charts which form a finite open covering of $A$.
Without loss of generality, assume that $W_\nu$ is taken to be relatively compact in some larger coordinate chart, and is the coordinate chart provided by the local parametrization theorem.
Then there exists $C >0$ such that for $r >0$ small enough, the open neighbourhood $\bigcup_{\nu}\{x \in W_\nu, d(A, x)<r \}$ of $A$ can be covered by at most $\frac{C}{r^{2d}}$ balls of radius $r$.
Here the distance is calculated by the coordinate distance in each coordinate chart.
\end{mylem}
\begin{proof}
It is enough to prove this for each $W_\nu$.
We verify that the volume of the open set $\{x \in W_\nu, d(A, x)<r \}$ has an upper bound $C r^{2n-2d}$ for $r$ small enough.
We take in each local tubular neighbourhood a maximal family of points with mutual coordinate distance $\ge r$.
For $r$ small enough, every point is at distance $\leq r$ to at least one of the centres, otherwise the family of points would not be maximal.
In particular, balls of radius $2r$ centered at these points  cover the tubular neighbourhood.
On the other hand, balls of radius $r/2$ centered at these points are disjoint.
Therefore, the number of such balls $N_\nu$ satisfies the relation
$$c_n N_\nu\Big(\frac{r}{2}\Big)^{2n} \leq \Vol(\{x \in W_\nu, d(A, x)<r \}) \leq Cr^{2n-2d}.$$
Here $c_n$ is the volume of the unit ball in $\C^n$.
The lemma follows from the inequality.

The proof of the volume estimate for the tubular neighbourhood is obtained by induction on the dimension of the analytic set $A$.
When $d=0$, i.e.\ when $A$ consists of a finite set, the estimate is trivial.
Assume that we have already proven the result for all analytic sets of dimension $d \leq \dim_{\C}(A)-1$. Then, we use the local parametrization theorem and the fact that $A \cap \pi^{-1}(S)$ is a proper analytic set of $A \cap W_\nu$.
By the induction hypothesis, we have
$$\Vol(\{x \in W_\nu, d(A\cap \pi^{-1}(S), x)<r \}) \leq Cr^{2n-2d+2},$$
and a similar estimate holds for the open set of points with distance${}<r$ to the irreducible components of $A$ of dimension $\leq d-1$.
On the other hand, $A \cap \pi^{-1}(\Delta' \setminus S)$ is contained in the union of $A_{j}+\sum_{i=d+1}^n \mathbb{D}(0,r)e_i$  where $e_i$ is the standard basis of $\C^n$ and $\mathbb{D}(0,r)$ is the disc in $\C$ centered at 0 of radius r.
Here $A_j$ are the irreducible components of dimension d of $A$ intersecting $\pi^{-1}(\Delta' \setminus S)$.
Each open set  $A_{j}+\sum_{i=d+1}^n \mathbb{D}(0,r)e_i$ has volume equal to $c(n,d)\Vol(A_j) r^{2n-2d}$ where $c(n,d)$ is the volume of the unit disc in $\C^{n-d}$.
This is because that $\pi$ induces a biholomorphism between  $A_{j}+\sum_{i=d+1}^n \mathbb{D}(0,r)e_i$ and  $\Delta' \times 0+\sum_{i=d+1}^n \mathbb{D}(0,r)e_i$
which preserves the Lebesgue volume form.
On the other hand the tubular neighbourhood of $A$ $\{d(x,A)<r\}$ is included in the union of the union of $A_{j}+\sum_{i=d+1}^n \mathbb{D}(0,r)e_i$, the open set of points whose distance to the dimension $\leq d-1$  irreducible components of $A$ $<r$ and $\{x \in W_\nu, d(A\cap \pi^{-1}(S), x)<r \}$ from which the estimate follows.
\end{proof}
\begin{myprop}{\it
Let $T$ be a $(k,k)$-closed positive current in the cohomology class $\alpha$, over a compact K\"ahler manifold $(X, \omega)$. Let $A$ be an analytic subset of $X$ of dimension $d$.
There exists a sequence of open neighbourhoods $U(r)$ of $A$ $($independent of $T\,)$ such that $\bigcap_{r >0} U(r)=A$ and the volume of $U(r)$ is at most $C r^{2n-2d}$, with a constant $C$ independent of $T$.
Moreover there exists $C'$ independent of $T$ such that
$$\int_{U(r)} T \wedge \omega^{n-k} \leq C'r^{2n-2k-2d}.$$ 
Here $C'$ depends on $\alpha$, $(X, \omega)$ and $A$.}
\end{myprop}
\begin{proof}
This is a direct consequence of Proposition 10 and Lemma 12.
\end{proof}
Remark that in particular, if $A$ is codimension at least $k+1$, the contribution of mass of $T$ on $U(r)$ vanishes asymptotically as $r \to 0$, and the above Proposition holds uniformly for all positive currents $T$ in the cohomology class~$\alpha$.
The codimension condition is optimal since that the mass of the current $[A]$ associated with a $k$-dimensional analytic set $A$ does not vanish in the limit.

Now we return to the construction of positive currents in the Segre classes.
Observe that a codimension condition is needed to ensure the existence of such closed positive currents; this is shown by the following easy example.
\begin{myex}{\rm 
Let $X$ be the blow up of $\P^2$ at some point and let $D$ be the exceptional divisor.
Consider the vector bundle $E:= \cO(D)^{\oplus r}$ of rank $r \geq 2$ over $X$. Corollary 3 shows that $E$ is a strongly psef vector bundle as a direct sum of strongly psef line bundles.


An equivalent definition of total Segre class (i.e. $\sum_k s_k(E)$) is the inverse of the total Chern class.
Remark that for any vector bundles $E, F$, the total Chern class satisfies the axiom 
$c(E \oplus F)=c(E)c(F)$.
Thus the same relation holds for the total Segre class since the cohomological ring is commutative.
In particular,
$s(E)=s(\cO(D))^r$
with $s_2(E)= {{r}\choose{2}}(c_1 (\cO(D))^2)=- {{r}\choose{2}}$.
Thus there exists no closed positive current in the class $s_2(E)$.
}
\end{myex}

For the convenience of the reader, we recall the definition of a Finsler metric on a vector bundle, as introduced in \cite{Kob} (cf.\ also \cite{Dem99}).
\begin{mydef}
A (positive definite) Finsler metric on a holomorphic vector bundle $E$ is a positive complex homogeneous function
$$\xi \to \| \xi \|_x$$
defined  on  each  fibre $E_x$,  that  is,  such  that $\|\lambda\xi \|_x=|\lambda|\| \xi \|_x$ for each $\lambda \in \C$ and $\xi \in E_x$, and $\|\xi\|_x>0$ for $\xi \neq 0$.
\end{mydef}
We say that the metric is smooth if it is smooth outside of the zero section on the total space of $E$.
Observe that a Finsler metric on a line bundle $L$ is the same as a Hermitian metric on $L$.
A Finsler metric on $E^*$ can also be viewed as a Hermitian metric $h^*$ on the line bundle 
$\cO_{\P(E)}(-1)$ (as the total space of $\cO_{\P(E)}(-1)$ coincides with the blow-up of $E^*$ along the zero section).
In particular, $\cO_{\P(E)}(1)$ carries a smooth Hermitian metric of positive Chern curvature form if and only if
$E$ carries a smooth Finsler metric 
whose logarithmic indicatrix defined by 
$$\chi(x, \xi):=\log \| \xi \|_x$$
is plurisubharmonic on the total space.
Let us observe that the logarithmic indicatrix has a pole along the zero section and can be extended as a global psh function on the total space, even though it is a priori psh only outside of the zero section.

Assume that we have a smooth Hermitian metric on $(E,h)$ rather than just a Finsler metric on $E$, and let us consider the corresponding  Hermitian metric on $\cO_{\P(E)}(1)$. We have the following calculation, which can be seen as a direct consequence of intersection theory, and is still valid on the level of forms without passing to cohomology classes: for every $k \in \N$
$$\pi_* \left(\frac{i}{2 \pi} \Theta (\cO_{\P(E)}(1),h)\right)^{r+k}= s_k(E,h).$$
Note that the Segre classes can be written in terms of Chern classes and the Chern classes can be represented by the Chern forms derived from the curvature tensor. For our application, we only detail the calculation for the case $k=1$ that we need.
For the general case, we refer for example to the papers \cite{Div16}, \cite{Gul12} and \cite{Mou04}.
The author thanks Simone Diverio for the references.
\begin{mylem}
Let $E$ be a holomorphic vector bundle of rank $r$ on a $($non necessarily compact$\,)$ complex manifold $X$. Let $\pi$ be the canonical projection $\P(E) \to X$.
Assume that $E$ is endowed with a smooth Hermitian metric $h$, and consider the induced metrics on $\cO_{\P(E)}(1)$ and $\det (E)$ $($which we still denote by $h)$.
Then
$$\pi_* \left(\frac{i}{2 \pi} \Theta (\cO_{\P(E)}(1),h)\right)^r= \frac{i}{2 \pi} \Theta (\det(E),\det(h))$$
where $\Theta$ means the curvature tensor.
\end{mylem}
\begin{proof}
To start with, we recall formula (15.15) of Chap.~V in \cite{agbook}, expressing the curvature of $\cO(1)$ for the projectivisation of a vector bundle.
Let $(e_{\lambda})$ be a normal coordinate frame of $E$ at $x_0 \in X$ and let
$$i \Theta(E)_{x_0}=\sum c_{jk \lambda \nu} i dz_j \wedge d\overline{z}_k \otimes e_\lambda^* \otimes e_\nu$$
be the curvature tensor of $E$.
At any point $a \in \P(E)$ represented by a vector $\sum_{\lambda} a_\lambda e_\lambda^* \in E^*_{x_0}$ of norm 1,
the curvature of $\cO_{\P(E)}(1)$ is
$$\Theta(\cO_{\P(E)}(1))_a=\sum c_{jk\nu \lambda} a_\lambda \overline{a}_\nu dz_j \wedge d\overline{z}_k+\sum_{1 \leq \lambda \leq  r-1} d\xi_\lambda \wedge d\overline{\xi}_\lambda,$$
where $(\xi_\lambda)$ are the coordinates near $a$ on $\P(E)$, induced by unitary coordinates of the hyperplane $a^{\perp} \subset E^*_{x_0}$.
In other words, if $\P(E|_U)$ is locally isomorphic to $U \times \P^{r-1}$ with coordinates $(z, [\xi])$, we have a canonical projection $pr_2: \P(E) \to \P^{r-1}$ and the curvature at $(z, [\xi])$ is given by
$$\frac{i}{2 \pi}\Theta(\cO_{\P(E)}(1))(z, [\xi])=-\frac{\langle \frac{i}{2 \pi} \Theta_{E^*} \xi, \xi \rangle_h}{\langle \xi, \xi \rangle_h}+pr_2^* \omega_{FS}$$
where $\omega_{FS}$ is the Fubini-Study metric on $\P^{r-1}$.
Therefore we have
$$\pi_* \left(\frac{i}{2 \pi} \Theta (\cO_{\P(E)}(1),h)\right)^r=-r \int_{\P^{r-1}} \frac{\langle \frac{i}{2 \pi} \Theta_{E^*} \xi, \xi \rangle_h}{\langle \xi, \xi \rangle_h} \wedge
 \omega_{FS}^{r-1}.$$
Observe that $\P^{r-1} \cong S^{2r-1}/S^1$ by the Hopf fibration. The Fubini-Study metric is the metric induced on the quotient $\P^{r-1}$ by the restriction of the standard Euclidean metric to the unit hypersphere. We denote by $d\sigma$ the volume form of the standard Euclidean metric restricted to that sphere. Then we have
$$\pi_* \left(\frac{i}{2 \pi} \Theta (\cO_{\P(E)}(1),h)\right)^r=-r \int_{S^{r-1}} \langle \frac{i}{2 \pi} \Theta_{E^*} \xi, \xi \rangle_h \wedge
 d\sigma.$$
Note that for a Hermitian form $Q(\xi, \xi)= \sum \lambda_i |\xi_i|^2$ we have
$$\int_{S^{r-1}} Q(\xi,\xi)d\sigma(\xi)=\frac{1}{r}\tr(Q)=\frac{1}{r} \sum \lambda_i,$$
since $\int_{S^{r-1}} |\xi_i|^2 d\sigma(\xi)=\frac{1}{r}$ by symmetry.
Thus we get
$$\pi_* \left(\frac{i}{2 \pi} \Theta (\cO_{\P(E)}(1),h)\right)^r=-\tr_{\xi} \langle \frac{i}{2 \pi} \Theta_{E^*} \xi, \xi \rangle_h=\frac{i}{2 \pi} \Theta (\det(E),h).$$
\end{proof}
As a direct consequence of the above formula, if $h$ is a smooth semi-positive metric on $\cO_{\P(E)}(1)$, the induced metric on $\det(E)$ is also semi-positive.
This is the positive form what we want.
More generally, the forms 
$\pi_* \left(\frac{i}{2 \pi} \Theta (\cO_{\P(E)}(1),h)\right)^{r+k}= s_k(E,h)$
are smooth positive currents in the $k-$th Segre class.
Hence if $h$ is a smooth semi-positive metric on $\cO_{\P(E)}(1)$, we can find positive forms in the Segre classes, which we will call Segre forms (or Segre currents) in the sequel.

In the case where the metric is singular, the construction is more complicated.
The difficulty is that Monge-Amp\`ere operators are not always well-defined
for arbitrary closed positive currents.

In general, for a strongly psef vector bundle, in order to get a singular metric with analytic singularities, we have to allow a bounded negative part. Accordingly, we have to work in a more general setting.
Let $E$ be a vector bundle of rank $r$ on a compact K\"ahler manifold $(X, \omega)$, and let $T$ be a closed positive $(1,1)$-current on $\P(E)$, in the cohomology class of a fixed closed smooth form $\alpha$.
Notice that that the restriction of the cohomology class $\{\alpha\}$ is constant on any fibre of $\pi:\P(E)\to X$.
A~typical case is $\{\alpha\}=c_1(\cO_{\P(E)}(1))+C \pi^* \omega$ for some $C \geq 0$.
Write 
$$T=\alpha+i \d\dbar \varphi.$$
Assume that $\varphi$ is smooth over $\P(E) \setminus A$ where $A$ is an analytic set in $\P(E)$ such that $A= \pi^{-1}(\pi(A))$ and $\pi(A)$ is of codimension at least $k$ in $X$. We wish to define a current $\pi_* T^{r-1+k}$.
A priori, this Monge-Amp\`ere operator is not well defined by just invoking the codimension condition, since the exponent $r-1+k$ is larger than the codimension~$k$.
This problem can be overcome by defining the desired current as a weak limit of a sequence of less singular currents, in such a way that the limit is still unique. 

Let $\psi$ be a quasi-psh function on $\P(E)$ that is smooth outside an analytic set $A'$ such that $A'$ is of dimension at most $n-k-1$.
In other words, the codimension of $A'$ in $\P(E)$ is at least $k+r$.
This implies that the codimension of $\pi(A')$ in $X$ is at least $k+1$.
Then the Monge-Amp\`ere operator $(\alpha+i \d \dbar \log(e^\varphi+\delta e^\psi))^{r-1+k}$ is well defined for every $\delta >0$, as a consequence of Demailly's techniques \cite{Dem92}.
Thus, by a weak compactness argument, the sequence of currents 
$$\pi_* (\alpha+i \d \dbar \log(e^\varphi+\delta_\nu e^\psi))^{r-1+k}$$
which all belong to the cohomology class $\pi_* \alpha^{r-1+k}$,
has a weak limit as $\delta_\nu \to 0$ for some subsequence.
Observe that if we take $\psi=0$, for any $\delta >0$, the function $\log(e^{\varphi}+\delta)$ is a bounded quasi-psh function.
In that case the wedge product
$$\pi_* (\alpha+i \d \dbar \log(e^\varphi+\delta))^{r-1+k}$$
is already well defined as a current  by the work of \cite{BT}. However, we want the flexibility of choosing a non constant potential $\psi$ in order to get quasi-psh functions with isolated singularities that can be used to get Lelong number estimates. Note that since all currents involved are closed, the limit current is still closed.

Now, we show that the limit is uniquely defined. The intuition is as follows.
As we have observed at the end of Proposition 11, the family of currents indexed by $\delta$ has a contribution of mass 0 along the singular part of $\pi(A')$, and we can therefore guess that the limit should be independent of the choice of $\psi$. (Nevertheless, without passing to the limit, each current may still have a positive Lelong number at some point of $\pi(A')$.)
\begin{mylem}
The limit current is independent of the choice of the smooth representative $\alpha$, as well as of the choice of $\psi$.
\end{mylem}
\begin{proof}
Fix a sequence $\delta_\nu$ tending to 0 such that the weak limit corresponding to $\alpha$ and $\psi=0$ exists.
Up to taking a subsequence which preserves the weak limit, we can assume in the following that the same sequence $\delta_\nu$ gives a weak limit  for different choice of $\alpha$ and $\psi$.
We will prove that the weak limits are the same, although a priori they might be different. 

Let $\tilde{\alpha}$, $\alpha$ be two representatives in the same cohomology class. Then there exists a smooth function $f$ on $\P(E)$ such that
$$\tilde{\alpha}=\alpha+i \d \dbar f.$$
Let $\tilde{\varphi}$ be the quasi-psh function such that $T=\tilde{\alpha}+i \d \dbar \tilde{\varphi}$.
Without loss of generality, we can assume that $\tilde{\varphi}=\varphi-f$.
Thus we have
$$\pi_* (\tilde{\alpha}+i \d \dbar \log(e^{\tilde{\varphi}}+\delta_\nu e^\psi))^{r-1+k}=\pi_* (\alpha+i \d \dbar \log(e^\varphi+\delta_\nu e^{\psi+f}))^{r-1+k}.$$
Thus to prove that the limit is independent of the choice of $\alpha$, it is enough to prove that the limit is independent of $\psi$, and this is what the proof will be devoted to from now on. 
As before, let $A' \subset \P(E)$ be the proper analytic subset such that $\psi$ is smooth on $\P(E) \setminus A'$.

On the regular part $X \setminus (\pi(A)\cup \pi(A'))$, the limit current is equal to
$$\pi_* (\alpha+i \d \dbar \varphi)^{r-1+k}$$
by the continuity of Monge-Amp\`ere operators with respect to bounded decreasing sequences and the fact that the currents are smooth on the pre-image of $X \setminus (\pi(A)\cup \pi(A'))$. Thus the limit currents corresponding to different choices of $\psi$ coincide on the regular part. Now, consider a K\"ahler form $\tilde{\omega}$ on $\P(E)$ 
satisfying the conditions
$$\alpha \geq - \tilde{\omega}/2,\quad i \d \dbar \psi \geq -\tilde{\omega}/2.$$
We can assume that the restriction of $\tilde{\omega}$ over all the fibres $\P^{r-1}$ is a fixed cohomology class.
For example, we can take 
$$\tilde{\omega}=C\pi^{*} \omega+c_1(\cO_{\P(E)}(1),h_\infty)$$
for some $C\gg 0$ and for a smooth metric $h_\infty$ on $\cO_{\P(E)}(1)$ induced by a Hermitian metric on $E$. For any $ \delta>0$ we have
$$\begin{matrix}&&\kern-10pt\alpha+i \d \dbar \log( e^{\varphi}+\delta e^\psi)\hfill\\
&&\displaystyle
\geq \alpha +\frac{e^{\varphi}}{e^{\varphi}+\delta e^\psi} (i \d \dbar \varphi)+\frac{\delta e^{\psi}}{e^{\varphi}+\delta e^\psi} (i \d \dbar \psi)+\frac{\delta e^{\varphi+\psi}}{(e^{\varphi}+\delta e^\psi)^2} i \d (\psi-\varphi) \wedge \dbar(\psi- \varphi)\geq -\tilde{\omega}\kern10pt\end{matrix} \leqno(*)$$
in the sense of currents, and the lower bound is independent of $\delta$.

By adding and subtracting $\tilde\omega$ and using the Newton binomial formula,
we see that the current $(\alpha+i \d \dbar \log( e^{\varphi}+\delta e^\psi))^{r+k-1}$ can be written  as a difference of two closed positive currents equal to summations of terms
$$(\alpha+i \d \dbar \log( e^{\varphi}+\delta e^\psi)+\tilde{\omega})^{i}\wedge \tilde{\omega}^j$$
with $i+j=r+k-1$.
Since the direct image functor transforms closed positive currents into closed positive currents, $\pi_*(\alpha+i \d \dbar \log( e^{\varphi}+\delta e^\psi))^{r+k-1}$ can also be written as a difference. 
If we compute the limit as $\delta$ tends to 0 (up to taking some convergent subsequence), the limit current will be a difference of two closed positive currents, in particular, $\lim_{\nu \to \infty}\pi_*(\alpha+i \d \dbar \log( e^{\varphi}+\delta_\nu e^\psi))^{r+k-1}$ is a normal current.

Denote by $T_1$, $T_2$ the limit currents obtained with different choices of $\psi$, namely $\psi_1$ and $\psi_2$. Assume that $A'$ is the union of the singular loci of $\psi_1$ and $\psi_2$. By assumption, $\pi(A')$ is of codimension at least $k+1$ in $X$. Then $T_1-T_2$ is a normal $(k,k)$-current supported in $\pi(A) \cup \pi(A')$.
If the codimension of $\pi(A)$ in $X$ is at least $k+1$, standard support theorems imply that $T_1=T_2$. If the codimension of $\pi(A) $ in $X$ is $k$,
the support theorem yields
$$T_1-T_2= \sum_\nu c_\nu [Z_\nu]$$
where $Z_\nu$ are the codimension $k$ irreducible components of $\pi(A)$ and $c_\nu \in \R$, and there exists no components of $\pi(A')$ as its codimension is higher. We now check that the limit current is independent of the choice of $\psi$ by a Lelong number calculation, i.e.\ by showing that $c_\nu=0$.

For any $x \in Z_{\nu_0,reg} \setminus (\bigcup_{\nu \neq \nu_0} Z_\nu \cup \pi(A'))$, there exists a coordinate chart $V$ such that $x=0$, $V \Subset X \setminus \pi(A') $, and $ Z_{\nu_0}=\{z_1=\cdots=z_k=0\}$ locally.
Take a cut-off function $\theta$ supported in $V$ and define
$$T_{1, \delta}=\alpha+i \d \dbar \log( e^{\varphi}+\delta e^{\psi_1}),$$
$$T_{2, \delta}=\alpha+i \d \dbar \log( e^{\varphi}+\delta e^{\psi_2}).$$
It is enough to prove that
$$\lim_{\delta \to 0}\int_X \left( \pi_*
 T_{1, \delta}^{k+r-1}-\pi_* T_{2, \delta}^{k+r-1} \right)
\wedge \theta\omega^{n-k}=0$$
which will imply that
$$\int_X (T_1-T_2) \wedge \theta\omega^{n-k}=0. $$
By a direct calculation, we have that
$$T_{1,\delta}^{k+r-1}-T_{2,\delta}^{k+r-1}=\Bigg(\sum_{j=0}^{k+r-1} T_{1,\delta}^j \wedge T_{2, \delta}^{r+k-1-j} \Bigg)\wedge (T_{1, \delta}-T_{2, \delta})$$
$$=\Bigg(\sum_{j=0}^{k+r-1} T_{1,\delta}^j \wedge T_{2, \delta}^{r+k-1-j}\Bigg)\wedge i \d \dbar \log\bigg(\frac{e^{\varphi}+\delta e^{\psi_1}}{e^{\varphi}+\delta e^{\psi_2}}\bigg).$$
An integration by parts gives
$$\int_X \big( \pi_*
 T_{1, \delta}^{k+r-1}-\pi_* T_{2, \delta}^{k+r-1} \big)
\wedge \theta\omega^{n-k}=\int_{\P(E)}\kern-4pt i \d \dbar \theta \wedge \omega^{n-k} \wedge \Bigg(\sum_{j=0}^{r+k-1} T_{1,\delta}^j \wedge T_{2, \delta}^{r+k-1-j}\Bigg) \log\bigg(\frac{e^{\varphi}+\delta e^{\psi_1}}{e^{\varphi}+\delta e^{\psi_2}}\bigg).$$
Define
$$F_{\delta}:= \log\bigg(\frac{e^{\varphi}+\delta e^{\psi_1}}{e^{\varphi}+\delta e^{\psi_2}}\bigg),$$
which is a uniformly bounded function on $V$ since $\bar{V}$ is outside of the image of the singular locus of $\psi_1$, $\psi_2$ under $\pi$. Note also that the bound is independent of $\delta$.
Moreover, $F_{\delta}$ tends to 0 almost everywhere as $\delta \to 0$.
The convergence is locally uniform outside of the pole set $A$ of $\varphi$.

Define $Z_{ \eta}:=\{z \in V, d(z,\pi( A)) \leq \eta\}$
with respect to the K\"ahler metric $\omega$.
The volume of $Z_{ \eta}$ with respect to $\omega$ tends to 0 as $\eta \to 0$ by the assumption that  $V \cap \pi(A)$ is a smooth submanifold in $V$.
Now we separate the estimate in different terms
$$\int_X \left( \pi_*
 T_{1, \delta}^{k+r-1}-\pi_* T_{2, \delta}^{k+r-1} \right)
\wedge \theta\omega^{n-k}=\int_{\pi^{-1}(Z_{\eta})} i \d \dbar \theta \wedge \omega^{n-k} \wedge\Bigg(\sum_{j=0}^{r+k-1} T_{1,\delta}^j \wedge T_{2, \delta}^{r+k-1-j}\Bigg) F_{\delta}$$
$$+\int_{\pi^{-1}(V \setminus Z_{\eta})} i \d \dbar \theta \wedge \omega^{n-k} \wedge \Bigg(\sum_{j=0}^{r+k-1} T_{1,\delta}^j \wedge T_{2, \delta}^{r+k-1-j}\Bigg) F_{\delta},$$
and we use the Fubini theorem to perform a double integration with respect to the base direction $V \setminus Z_{\eta}$ (resp. $ Z_{\eta}$) and the fibration direction $\P^{r-1}$, for $V$ sufficiently small.
The first term in the integration is bounded by $$C \omega^{n-k+1} \wedge\Bigg(\sum_{j=0}^{r+k-1} (T_{1,\delta}+\tilde{\omega})^j \wedge (T_{2, \delta}+\tilde{ \omega})^{r+k-1-j}\Bigg)$$
with $C$ independent of $\delta$ since $F_{\delta}$ is uniformly bounded on $\bar{V}$ and $i \d \dbar \theta$ is bounded by $C \omega$ for $C$ large enough.

The currents $T_{1, \delta}$ and $T_{2, \delta}$ are not smooth on $Z_\eta$, thus some attention has to be paid to apply the Fubini theorem.
Let $U(\eta)$ (resp. $U'(\eta)$) be the open neighbourhoods of $A$ (resp. $A'$) in $\P(E)$ given by Proposition 11.
Note that $T_{1, \delta}$ and $T_{2, \delta}$ are smooth near the boundary of $U(\eta)\cup U'(\eta)$.
Without loss of generality, we can assume that $\pi^{-1}(Z_\eta)$ is contained in $U(\eta)\setminus U'(\eta)$.
Take smooth currents $\tilde{T}_{i, \delta}$ on $U(\eta) \cup U'(\eta)$ cohomologous to $T_{i, \delta}$, which coincide with $T_{i, \delta}$ ($i=1,2$) near the boundary of $U(\eta) \cup U'(\eta)$.
By Stokes' theorem, 
$$\int_{U(\eta) \cup U'(\eta)} \omega^{n-k+1} \wedge\Bigg(\sum_{j=0}^{r+k-1}(T_{1,\delta}+\tilde{\omega})^j \wedge (T_{2, \delta}+\tilde{ \omega})^{r+k-1-j}\Bigg)$$
$$=\int_{U(\eta) \cup U'(\eta)} \omega^{n-k+1} \wedge\Bigg(\sum_{j=0}^{r+k-1}(\tilde{T}_{1,\delta}+\tilde{\omega})^j \wedge (\tilde{T}_{2, \delta}+\tilde{ \omega})^{r+k-1-j}\Bigg).$$
Therefore we can apply the Fubini theorem in the right hand side since all terms are smooth.
The integral on $\pi^{-1}(Z_\eta)$ is bounded from above by the integral of the same term on $U(\eta) \cup U'(\eta)$ by the inclusion relation
$\pi^{-1}(Z_\eta)\subset U(\eta) \cup U'(\eta)$.

We first perform the integration along the fibres $\P^{r-1}$.
The integration of $ \sum_{j=0}^{r+k-1} (\tilde{T}_{1,\delta}+\tilde{\omega})^j \wedge (\tilde{T}_{2, \delta}+\tilde{ \omega})^{r+k-1-j}$ along the fibre direction is a cohomological constant since we assume that the restriction of cohomology class of $\alpha$ along each fibres is a fixed cohomology class on $\P^{r-1}$.
Thus the integral on $U(\eta)\cup U'(\eta)$ is bounded from above by $C \int_{U(\eta)\cup U'(\eta)} \omega^n$, for some $C$ independent of $\delta$.
Observe that the constant is the same as the supremum of $|F_\delta|$ on $\overline{V}$ (independent of $\delta$), since for $\eta$ small enough $\overline{V} \cap  U'(\eta) = \emptyset$.


The second term appearing in the integral is bounded by
$$\sup_{\pi^{-1}(X \setminus Z_{\eta})}|F_{\delta}| \sup_X |i \d \dbar \theta|_\omega  \omega^{n-k+1} \wedge\Bigg(\sum_{j=0}^{r+k-1} (T_{1,\delta}+\tilde{\omega})^j \wedge (T_{2, \delta}+\tilde{ \omega})^{r+k-1-j}\Bigg).$$
On $V \setminus Z_\eta$, the currents $T_{1, \delta}$ and $T_{2, \delta}$ are smooth, thus the Fubini theorem applies.
We first integrate along $\P^{r-1}$.
The integration of $ \sum_{j=0}^{r+k-1} ((T_{1,\delta}+\tilde{\omega})^j \wedge (T_{2, \delta}+\tilde{ \omega})^{r+k-1-j})$ along the fibre direction is a cohomological constant as above.
Thus the second term obtained after integrating is bounded from above by $C \sup_{\pi^{-1}(X \setminus Z_{\eta})}|F_{\delta}|$, for some $C$ independent of $\delta$.

For every $\varepsilon' >0$, there exist $\eta$ such that $C \int_{U(\eta) \cap U'(\eta)} \omega^n < \frac{\varepsilon'}{2}$. 
There also exists $\delta_0$
such that $C \sup_{X \setminus Z_{\eta}}|F_{\delta}|<\frac{\varepsilon'}{2}$ for every $\delta \leq \delta_0$.
Thus the two parts of estimate (integration on $U(\eta) \cup U'(\eta)$ and on $\pi^{-1}(V \setminus Z_\eta)$) are both bounded from above by $\frac{\varepsilon'}{2}$ for $\delta \leq \delta_0$.
This concludes the proof that the limit current is independent the choice of $\psi$.

\end{proof} 
In what follows we show that the weak limit is also independent of the subsequence $\delta_\nu$ if the weight function $\varphi$ has analytic singularities.
It seems that the independence of the weak limit does not hold in general if we only require that $\varphi$ is smooth outside an analytic set of sufficient high codimension.
However some special cases can be easily checked.
\begin{myex}{\em 
Assume that there exists some $C_2 \geq C_1 >0 $ such that
$$ C_1 \delta'_\nu \leq \delta_\nu \leq C_2 \delta'_\nu$$  up to taking some subsequence but with the same limit currents.
Then the function
$$\log\bigg(\frac{e^{\varphi}+\delta_\nu e^\psi}{e^{\varphi}+\delta'_\nu e^{\psi}}
\bigg)$$
is uniform bounded on $\mathbb{P}(E)$ (independently of $\nu$).
It is locally uniformly convergent to $0$ on $\pi^{-1}(X \setminus Z_{\eta})$.
The same arguments as above can be used to achieve the proof.

Another easy case is when the projection of the singular part of $\varphi$ is of codimension at least $k+1$.
In this case, different choices of subsequence $\delta_\nu$ will have the same closed positive limit outside an analytic set of codimension at least $k+1$.
By standard support theorems, they have to coincide over $X$.
}
\end{myex}
The case of potentials with analytic singularities comes from the following observation of Demailly.
\begin{myprop}{\it
Let $\varphi$ be a quasi-psh function with analytic singularities over on a $($connected$)$ complex
$n$-dimensional mani\-fold~$X$, and $u\in C^\infty(X)$.
Then for any exponent $p$ $(1 \leq p \leq n)$, the asymptotic limit of Monge-Amp\`ere operator $\lim_{\delta\to 0}(i \d \dbar\log(e^\varphi+\delta e^u))^p$  is always well defined as a current $($but not necessarily positive, even when $i\d\dbar\varphi\geq 0$, and the limit may depend on $u)$.}
\end{myprop}
\begin{proof} By writing $\log(e^\varphi+\delta e^u)=
\log(e^{\varphi-u}+\delta)+u$ and using a binomial expansion, it is sufficient
to consider the case $u=0$, after replacing $\varphi$ with $\varphi-u$.  
Let us now consider the divisorial case, i.e., assume that $X=\C^n$ and that $\varphi$ is of the form $\varphi=\log|f|^2+\psi$ for some holomorphic function $f=\prod_{i=1}^m z_i^{m_i} \in \cO(X)$ and $\psi \in C^{\infty}(X)$.
We can define $h=e^{\psi}$ a smooth Hermitian metric on
$L:=\cO_X$. We denote by $\nabla_h$ the associated Chern connection.

Then, for every $\delta>0$, we have 
$i\d \dbar\log(e^{\varphi}+\delta)=i\d\dbar\log(|f|^2_h+\delta)$
which converge to $i \d\dbar \varphi$ as $\delta \to 0+$.
We will define the Monge-Amp\`ere operator $(i \d \dbar \varphi)^p$ as the limit of $\big( i\d\dbar\log(|f|^2_h+\delta) \big)^p$ as $\delta \to 0+$.
For every $\delta>0$, we have
$$
i\ddbar\log(|f|^2_h+\delta)
=i\partial{\langle f,\nabla_h f\rangle\over |f|^2_h+\delta}
={i\langle \nabla_hf,\nabla_h f\rangle\over |f|^2_h+\delta}-
i{\langle \nabla_h f,f\rangle\over |f|^2_h+\delta}\wedge
{\langle f,\nabla_h f\rangle\over |f|^2_h+\delta}
+i{\langle f,\nabla^{0,1}_h\nabla_h^{1,0} f\rangle\over |f|^2_h+\delta}$$
$$={\delta\over (|f|^2_h+\delta)^2}\,i\langle \nabla_hf,\nabla_h f\rangle
- {|f|^2_h\over |f|^2_h+\delta}\,i\,\Theta_{L,h}.
$$
Now, $i\langle \nabla_hf,\nabla_h f\rangle$ is a $(1,1)$-form of rank $1$. In particular, its wedge powers of exponents${}>1$ are equal to~$0$. If we raise to power $p$, the Newton binomial formula implies
$$
\Big({i\over 2\pi}\ddbar\log(|f|^2_h+\delta)\Big)^p
={p\delta\over (|f|^2_h+\delta)^2}
\Big({|f|^2_h\over |f|^2_h+\delta}\Big)^{p-1}
\,{i\over 2\pi}\langle \nabla_hf,\nabla_h f\rangle\wedge
\Big(-{i\over 2\pi}\Theta_{L,h}\Big)^{p-1}
$$$$+
\Big({|f|^2_h\over |f|^2_h+\delta}\Big)^p
\Big(-{i\over 2\pi}\Theta_{L,h}\Big)^p.
$$
The last term
converges almost everywhere to $(-{i\over 2\pi}\Theta_{L,h})^p$, thus
it converges weakly to the same limit by the bounded convergence theorem as $\delta \to 0+$. We claim that
$$
{p\delta\,|f|_h^{2p-2}\over (|f|^2_h+\delta)^{p+1}}
\,{i\over 2\pi}\langle \nabla_hf,\nabla_h f\rangle\to [Z_f]\leqno(*)
$$
weakly, where $[Z_f]$ is the current of integration on the zero divisor of $f$.
Terms that depend on $h$ in $\nabla_hf$ are equal to $f\partial\varphi$, and
they can be seen to yield zero limits, using the Cauchy-Schwarz formula and
the fact that
$$
{p\delta\,|f|_h^{2p-2}\over (|f|^2_h+\delta)^{p+1}}\cdot |f|^2_h\le p
$$
converges to zero almost everywhere. 
In fact the limit (if it exists) is a positive current as a limit of
positive currents. It will also be closed, since
$$
\dbar\bigg({p\delta\,|f|_h^{2p-2}\over (|f|^2_h+\delta)^{p+1}}
\,{i\over 2\pi}\langle \nabla_hf,\nabla_h f\rangle\bigg)
={p\delta\,|f|_h^{2p-2}\over (|f|^2_h+\delta)^{p+1}}
\,{1\over 2\pi}\langle f,\nabla_h f\rangle\wedge\Theta_{L,h}
$$
and we can again apply a Cauchy-Schwarz argument to see that the right hand side
converges to $0$. A priori the limit current (if it exists) should be
supported on $|Z_f|$. However, at any regular point of $Z_f$ we can find local holomorphic coordinates in which $f(z)=z_1^m$, where $m$ is the multiplicity
of the irreducible component. An easy calculation yields
$$
\int_{z_1\in\C}
{p\delta\,|z_1^m|^{2p-2}\over (|z_1^m|^2+\delta)^{p+1}}
\,{i dz_1^m\wedge d\overline{z_1^m}\over 2\pi}=m.\leqno(**)
$$
Equality $(**)$ can be checked e.g.\ by putting $w=z_1^m$,
using polar coordinates $w=r e^{i\theta}$ and making a change of variables
$t={r^2\over r^2+\delta}$.
More generally, if $f(z)=\prod z_i^{m_i}$, we have to consider the integration 
$$
\int_{\{|z_i|\leq 1\}}
{p\delta\,|\prod_{i=1}^m z_i^{m_i}|^{2p-2}\over (|\prod_{i=1}^m z_i^{m_i}|^2+\delta)^{p+1}}
\,{i d(\prod_{i=1}^m z_i^{m_i})\wedge d(\overline{\prod_{i=1}^m z_i^{m_i}})\over (2\pi)^n} \wedge \omega_{\mathrm{eucl}}^{n-1}
$$
where $\omega_{\mathrm{eucl}}$ is the standard $(1,1)$-form associated with the euclidean metric on $\C^n$. It is bounded by sums of integrals of the type
$$
\int_{\{0<|z_i|\leq 1,2\leq i \leq n\}}
{p\delta\,||\prod_{i=2}^m z_i^{m_i}|z_1^{m_1}|^{2p-2}\over (||\prod_{i=2}^m z_i^{m_i}|z_1^{m_1}|^2+\delta)^{p+1}}
\,{i |\prod_{i=2}^m z_i^{m_i}|d(z_1^{m_1})\wedge |\prod_{i=2}^m z_i^{m_i}|d\overline{z_1^{m_1}}\over (2\pi)^n} \wedge \omega_{\mathrm{eucl}}^{n-1}.
$$
The integral is finite by the Fubini theorem and a calculation similar to $(**)$, putting e.g.\ $w=|\prod_{i=2}^m z_i^{m_i}|z_1^{m_1}$.
In particular, up to taking a subsequence, the limit in formula $(*)$ exists as $\delta \to 0+$. By the support theorem any limit current is associated to a divisor supported in $|Z_f|$. To show that the weak limit is unique, it is
sufficient to check formula $(*)$ at a regular point of $|Z_f|$ and to show
that the coefficient is unique. This actually follows from equality $(**)$.

As a consequence of the above calculations, we find
$$
\Big({i\over 2\pi}\ddbar\log(|f|^2_h+\delta)\Big)^p\to
(-1)^{p-1}[Z_f]\wedge\Big({i\over 2\pi}\Theta_{L,h}\Big)^{p-1}+
(-1)^p\Big({i\over 2\pi}\Theta_{L,h}\Big)^p.
$$

For the general case, we apply Hironaka's theorem. There exists a certain modification $\sigma: \tilde{X} \to X$ of $X$ such that $\sigma^* \varphi$ is locally of the form considered in the previous case, where $f$ has a simple normal crossing divisor.
Thus the limit
$$\lim_{\delta \to 0+} \big( i\d\dbar\log(e^{ \varphi}+\delta) \big)^p=
\sigma_*\Big(\lim_{\delta \to 0+} \big( i\d\dbar\log(e^{\sigma^*\varphi}+\delta) \big)^p\Big)
$$
exists by the weak continuity of the direct image operator $\sigma_*$. By the filtering property of modifications, one can also see that  the above limit is independent of the choice of the modification~$\sigma$.
\end{proof}
It follows directly from the proposition that the limit current is independent of the subsequence $\delta_\nu$ if the weight function $\varphi$ has analytic singularities.
It has been communicated 
to
us
by
Richard
Lärkäng
that
a
similar
calculation
has
been
done
in
\cite{ABW19}
and
\cite{Bl19}.
The advantage of the construction made in Lemma 14 is that under the assumption that the weight function is smooth outside of an analytic set of sufficient high codimension, one can show that the limit current is positive. This is shown in Theorem 2 below.
\begin{myex}{\rm
We describe below a special case of the previous construction.
Let $E$ be a strongly psef vector bundle over a compact K\"ahler manifold $(X, \omega)$.
Let $h_\infty$ be an arbitrary metric on $E$.
Since $\cO_{\P(E)}(1)$ is relatively ample with respect to
the projection $\pi: \P(E) \to X$,
there exists $C >0$ big enough such that
$$i \Theta(\cO_{\P(E)}(1),h_\infty)+C \pi^* \omega >0.$$
We take the above form as a smooth representative in the
class $c_1(\cO_{\P(E)}(1))+C\pi^*\{ \omega\}$.
By definition of a strongly psef vector bundle, there exists 
a singular metric $h_\varepsilon$ with analytic singularities on
$\cO_{\P(E)}(1)$ such that 
$$i \Theta(\cO_{\P(E)}(1),h_\varepsilon) \geq - \varepsilon \pi^* \omega.$$
By the above construction, $\pi_* \left(\frac{i}{2 \pi} \Theta(\cO_{\P(E)}(1),h_\varepsilon)+C \pi^* \omega \right)^r$
is well defined for $\varepsilon$ small enough by taking that $\psi=0$.
In the construction, all currents are positive currents.
In particular,  $\pi_* \left(\frac{i}{2 \pi} \Theta(\cO_{\P(E)}(1),h_\varepsilon)+C \pi^* \omega \right)^r$ is a closed positive current on $X$ for $\varepsilon$ small enough.
On the other hand,
$$\phantom{\Bigg|}\kern-3.5cm
\pi_* \left(\frac{i}{2 \pi} \Theta(\cO_{\P(E)}(1),h_\varepsilon)+C \pi^* \omega \right)^r=\pi_* \left(\frac{i}{2 \pi} \Theta(\cO_{\P(E)}(1),h_\varepsilon) \right)^r$$
$$\kern6.5cm{}+r \pi_* \left(C \pi^* \omega \wedge (\frac{i}{2 \pi} \Theta(\cO_{\P(E)}(1),h_\varepsilon))^{r-1} \right) +\cdots.$$
In the $\cdots$ summation, there are terms of the form
$$\pi_* \left( \pi^* \omega^i \wedge (\frac{i}{2 \pi} \Theta(\cO_{\P(E)}(1),h_\varepsilon))^{r-i} \right)$$
for $i \geq 2$.
By the projection formula, we have
$$\pi_* \left( \pi^* \omega^i \wedge (\frac{i}{2 \pi} \Theta(\cO_{\P(E)}(1),h_\varepsilon))^{r-i} \right)=\pi_*\left(\frac{i}{2 \pi} \Theta(\cO_{\P(E)}(1),h_\varepsilon)\right)^{r-i} \wedge \omega^i.$$
By a degree consideration, for $i \geq 2$, the right hand side is $0$ and for $i=1$ it is equal to $\omega$. In conclusion, 
$$\pi_* \left(\frac{i}{2 \pi} \Theta(\cO_{\P(E)}(1),h_\varepsilon)+C \pi^* \omega \right)^r=\pi_* \left(\frac{i}{2 \pi} \Theta(\cO_{\P(E)}(1),h_\varepsilon) \right)^r+C r \omega \geq 0$$
in the sense of currents. In particular,
$\pi_* \left(\frac{i}{2 \pi} \Theta(\cO_{\P(E)}(1),h_\varepsilon) \right)^r$
is a quasi-positive current (i.e.\ a current bounded below by a smooth form), belonging to the cohomology class $c_1(\det(E))$ by Lemma~15.}
\end{myex}
More generally, we have the following Segre current construction.
\begin{mythm} (Main technical lemma)
Let $E$ be a vector bundle of rank $r$ over a compact K\"ahler manifold $(X, \omega)$, and let $T$ be a closed positive $(1,1)$-current on $\P(E)$, belonging to the same cohomology class as a smooth form $\alpha$. Write 
$$T=\alpha+i \d\dbar \varphi.$$
Assume that $\varphi$ is smooth over $\P(E) \setminus A$, where $\pi: \P(E) \to X$ is the projection and $A$ is an analytic set in $\P(E)$ such that $A= \pi^{-1}(\pi(A))$ and $\pi(A)$ is of codimension at least $k$ in $X$.
Then there exists a $(k,k)$-positive current in the class $\pi_* \{\alpha \}^{r+k-1}$.
\end{mythm}
\begin{proof}
The desired current $\pi_* (T^{r+k-1})$ has been constructed, and its uniqueness has been shown in the previous lemma. It remains to show that $\pi_* (T^{r+k-1})$ is positive. It is enough to prove this near an arbitrary point $x \in X$, since positivity is a local property. There exists a smooth function $\psi$ on $\P(E)$ such that
$$\alpha+i \d \dbar \psi \geq 0$$
on an open neighbourhood $U$ of $x$. Thus over $U$, for every $\delta>0$, we have
$$\alpha+ i \d \dbar \log(e^\varphi+\delta e^{\psi}) \geq 0$$
using $(*)$ in the previous lemma. Therefore, over $U$ again, we see that
$$\pi_* T^{r+k-1}=\lim_{\delta \to 0} \pi_* (\alpha+ i \d \dbar \log(e^\varphi+\delta e^{\psi}))^{r+k-1}$$
is positive as a limit of positive currents.
\end{proof}
In fact, the above construction would work for a broader situation as stated in the following theorem.
\begin{mythm}
Let $\pi: X \to Y$ be a submersion between compact K\"ahler manifolds  of relative dimension $r-1$. 
Let $T$ be a closed positive $(1,1)-$current in the cohomology class $\{ \alpha \} \in H^{1,1}(X, \R)$
such that $T$ has analytic singularities and is smooth on $X \setminus \pi^{-1}(Z)$ with $Z$ a closed analytic set of codimension at least $k$.
Assume that for any $y \in Y$, there exist an open neighborhood $U$ of $y$ and a quasi-psh function $\psi$ on $X$ such that
$\alpha + i \d \dbar \psi \geq 0$
in the sense of currents on
$\pi^{-1}(U)$ and $\psi$ is smooth outside a closed analytic set of codimension at least $k+r$.
Then there exists a closed positive current in the cohomology class $\pi_* \alpha^{r+k-1}$.
\end{mythm}
\begin{proof}
The only place of relative nature is the proof of Lemma 14.
The assumption on $\psi$ ensures that
$$\alpha+ i \d \dbar \log(e^\varphi+\delta e^{\psi}) \geq 0$$
is positive on $\pi^{-1}(U)$ and the Monge-Amp\`ere operator is well-defined on $X$.
In the proof, we explicitly choose a K\"ahler form $\tilde{\omega}$ on $X$ such that the restriction of $\tilde{\omega}$ to the fibres of $\pi$ is of constant cohomology class -- this property being
automatically true for any smooth proper morphism.
\end{proof}

In the special case of Segre currents, we get
\begin{mycor} {\it
Let $E$ be a strongly psef vector bundle of rank $r$ over a compact K\"ahler manifold $(X, \omega)$. 
Let $(\cO_{\P(E)}(1),h_\varepsilon)$ be a singular metric with analytic singularities such that 
$$i \Theta(\cO_{\P(E)}(1),h_\varepsilon) \geq - \varepsilon \pi^* \omega$$
and the codimension of $\pi(\Sing (h_\varepsilon))$ is at least $k$ in $X$.
Then there exists a $(k,k)$-positive current in the cohomology class $\pi_* (c_1(\cO_{\P(E)}(1))+\varepsilon \pi^* \{\omega\})^{r+k-1}$.
In particular, $\det(E)$ is a psef line bundle.}
\end{mycor}
\begin{proof}
The first part is a direct consequence of Theorem 2.
The second part is consequence of the fact that when $k=1$ one has
$$\pi_* (c_1(\cO_{\P(E)}(1))+\varepsilon \pi^* \{\omega\})^{r}=c_1(\det (E))+ r \varepsilon \omega.$$
\end{proof}
\begin{myrem}{\rm
Let $h$ be a smooth metric on $\cO_{\P(E)}(1)$ (not necessarily coming from a Hermitian metric on $E$).
We can define an induced singular metric on $\det(E)$ in the following non canonical way.
Fix an arbitrary smooth Hermitian metric $h_\infty$ on $\P(E)$.
Then there exists $\psi \in C^{\infty}(\P(E))$ such that $h=h_\infty e^{-\psi}$.
Therefore we have
$$\frac{i}{2 \pi} \Theta(\cO_{\P(E)}(1), h)-\frac{i}{2 \pi} \Theta(\cO_{\P(E)}(1), h_\infty)= \frac{i}{2 \pi} \d \dbar \psi.$$
Define a metric on $\det(E)$ by $\det(h_\infty)e^{-\varphi}$ with 
$$\varphi:= \pi_* \left(\psi \sum_{j=0}^{r-1}\Big(\frac{i}{2 \pi} \Theta(\cO_{\P(E)}(1), h)\Big)^j \wedge\Big(\frac{i}{2 \pi} \Theta(\cO_{\P(E)}(1), h_\infty)\Big)^{r-1-j} \right).$$ 
We have that
$$\frac{i}{2 \pi} \d \dbar \varphi= \pi_* \left(\Big(\frac{i}{2 \pi} \Theta(\cO_{\P(E)}(1), h)\Big)^r-\Big(\frac{i}{2 \pi} \Theta(\cO_{\P(E)}(1), h_\infty)\Big)^r \right).$$
In other words,
$$ \frac{i}{2 \pi} \Theta(\det(E), \det(h_\infty)e^{-\varphi})=\pi_*\Big(\frac{i}{2 \pi} \Theta(\cO_{\P(E)}(1), h)\Big)^r.$$
If $h$ comes from a Hermitian metric of $E$, we get precisely the same curvature formula as in Lemma~15.}
\end{myrem}
\begin{myrem}{\rm
The definition in the previous remark is non canonical in the sense that it depends on the choice of the reference metric $h_\infty$.
This can be seen as follows. In analogy with
the Monge-Amp\`ere functional, we consider the functional
$$M_{h_\infty}: C^\infty(\P(E)) \to C^{\infty}(X)$$
$$\psi \mapsto \pi_* \left(\psi \sum_{j=0}^{r-1}\Big(\frac{i}{2 \pi} \Theta(\cO_{\P(E)}(1), h_\infty)+\frac{i}{2\pi} \d \dbar \psi\Big)^j \wedge\Big(\frac{i}{2 \pi} \Theta(\cO_{\P(E)}(1), h_\infty)\Big)^{r-1-j} \right).$$
Let $\psi_t$ be a smooth path in $ C^\infty(\P(E))$.
We compute the Fr\'echet differential
$$\frac{dM_{h_\infty}(\psi_t)}{dt}= \pi_* \left(\dot{\psi_t} \sum_{j=0}^{r-1}\Big(\frac{i}{2 \pi} \Theta(\cO_{\P(E)}(1), h_\infty)+\frac{i}{2\pi} \d \dbar \psi_t\Big)^j \wedge\Big(\frac{i}{2 \pi} \Theta(\cO_{\P(E)}(1), h_\infty)\Big)^{r-1-j} \right)+$$
$$\pi_* \left( \psi_t \sum_{j=0}^{r-1} j\frac{i}{2\pi}\d \dbar \dot{\psi_t} \wedge\Big(\frac{i}{2 \pi} \Theta(\cO_{\P(E)}(1), h_\infty)+\frac{i}{2\pi} \d \dbar \psi_t\Big)^{j-1} \wedge\Big(\frac{i}{2 \pi} \Theta(\cO_{\P(E)}(1), h_\infty)\Big)^{r-1-j} \right)$$
which, by an integration by parts, is equal to
$$\pi_* \left( \dot{\psi_t}\Big(\frac{i}{2 \pi} \Theta(\cO_{\P(E)}(1), h_\infty)\Big)^{r-1} \right).$$
Now let $h_\infty$, $\tilde{h}_\infty$ be two smooth metrics on $E$ and denote the induced metrics on $\cO_{\P(E)}(1)$ by the same notation.
Let $\psi_t$ be a smooth path connecting $h_\infty$ and $\tilde{h}_\infty$.
For example we can take $\psi_t$ such that $h_\infty e^{-\psi_t}=h_\infty^t \tilde{h}_\infty^{1-t}$.
As a consequence of the calculation of Fr\'echet differential,
our functional satisfies for any $\varphi \in C^{\infty}(\P(E))$
the  cocycle relation
$$M_{h_\infty}(\varphi+\psi_1)=M_{\tilde{h}_\infty}(\varphi)+M_{h_\infty}(\psi_1).$$
Let us note that $M_{h_\infty}(\varphi+\psi_1)$ (resp. $M_{\tilde{h}_\infty}(\varphi)$) is the weight function of the induced metric on $\det(E)$ with respect to the reference metric $h_\infty$ (resp. $\tilde{h}_\infty$), associated with the weight function $\varphi+\psi_1$ (resp. $\varphi$) on $\P(E)$.
In particular, they correspond to metrics on $\det(E)$ that are induced by the same metric on $\cO_{\P(E)}(1)$, but with different reference metrics $\tilde{h}_\infty$ and $h_\infty$.
Since $i \d \dbar M_{h_\infty}(\varphi)$ is independent of the choice of the reference metric $h_\infty$, we have
$i \d \dbar M_{h_\infty}(\psi_1) \equiv 0$, and this means that $M_{h_\infty}(\psi_1)$ is a constant. Therefore the metric defined in the previous remark is uniquely defined up to a constant.}
\end{myrem}
\section{Strongly pseudoeffective and numerically trivial bundles}
In this section, we use the Lelong number estimate to show that a strongly psef vector bundle with trivial first Chern class is in fact numerically flat.
In particular, this implies that a strongly psef reflexive sheaf with trivial first Chern class is in fact a numerically flat vector bundle.
As an application of the previous section, we get the following result.
\begin{mythm} (Main theorem)
Let $E$ be a strongly psef vector bundle of rank $r$ on a compact K\"ahler manifold $(X, \omega)$, such that $c_1(E)=0$. 
 Then $E$ is a nef $($and thus numerically flat$\,)$ vector bundle.
\end{mythm}
\begin{proof}
We show through Lelong number estimates and regularization, that the vector bundle $E$ is in fact nef.
Let $h_\varepsilon$ be a singular metric with analytic singularities on
$\cO_{\P(E)}(1)$, such that 
$$i \Theta(\cO_{\P(E)}(1),h_\varepsilon) \geq - \varepsilon \pi^* \omega.$$
Let us write $h_\varepsilon=h_\infty e^{-\varphi_\varepsilon}$ with respect
to some smooth reference metric $h_\infty$ on $\cO_{\P(E)}(1)$.
Define
$$T_\varepsilon:= \pi_*\Big(\frac{i}{2\pi} \Theta(\cO_{\P(E)(1)}, h_\infty)+\frac{i}{2\pi} \d \dbar {\varphi}_\varepsilon\Big)^r$$
by means of Theorem 2. We have
$T_\varepsilon \geq - r \varepsilon \omega$.
More precisely, we are going to prove the Lelong number estimate
$$\nu(T_\varepsilon,z) \geq \min(  \bigg(\sup_{w,\pi(w)=z} \nu(\varphi_\varepsilon,w)\bigg)^r,1).$$
The proof of this estimate is similar to the proof of Theorem 10.2 of \cite{Dem93}. For the convenience of the reader, we briefly outline the proof here.
Fix $w_0 \in \pi^{-1}(x)$ and $\gamma= \nu(\varphi_\varepsilon,w_0)$.
The inequality is trivial when $\gamma=0$.
Otherwise, for any $\varepsilon' <\min( \gamma,1)$, let us choose the following weight function.
Choose $\eta>0$ so small that $\{|z|\leq 2 \eta \}$ is contained in a coordinate chart with $\pi(w_0)=0$.
Identify locally $\pi^{-1}\{|z|\leq 2 \eta \}$ with $\{|z|\leq 2 \eta \} \times \P^{r-1}$ with projection $p_2:\{|z|\leq 2 \eta \} \times \P^{r-1} \to \P^{r-1}$.
Let $\log|w-w_0|$ be an $\omega_{\mathrm{FS}}-$psh function on $\P^{r-1}$ with isolated singularity and Lelong number 1 at $p_2(w_0)$ where  
$\omega_{\mathrm{FS}}$ is the Fubini-Study metric. 
Define
$$\psi:= (\min(\gamma,1)-\varepsilon')\pi^* \theta p_2^* \log|w-w_0|+ \rho$$
where $\theta$ is a cut off function supported in $\{|z|\leq 2 \eta \}$ which is identically equal to 1 on $\{|z|\leq \eta \}$ and  $\rho$ a smooth function on $\P(E)$ such that 
$$\frac{i}{2\pi} \Theta(\cO_{\P(E)(1)}, h_\infty)+\frac{i}{2\pi} \d \dbar \psi$$
is positive on $\pi^{-1}\{|z|\leq \eta \}$.
By Lemma 14, we have
$$T_\varepsilon= \lim_{\delta \to 0}\pi_*\left(\frac{i}{2\pi} \Theta(\cO_{\P(E)(1)}, h_\infty)+\frac{i}{2\pi} \d \dbar  \log (e^{{\varphi}_\varepsilon}+\delta e^{\psi})\right)^r$$
in the sense of currents where
$$\frac{i}{2\pi} \Theta(\cO_{\P(E)(1)}, h_\infty)+\frac{i}{2\pi} \d \dbar  \log (e^{{\varphi}_\varepsilon}+\delta e^{\psi})$$
is positive on $\{|z|\leq \eta \}$.
We have
$$\int_{|z| \leq \eta} T_\varepsilon \wedge \Big(\frac{i}{2\pi} \d \dbar \log|z|^2\Big)^{n-1} \geq $$$$\limsup_{\delta \to 0} \int_{|z| \leq \eta} \pi_*\left(\frac{i}{2\pi} \Theta(\cO_{\P(E)(1)}, h_\infty)+\frac{i}{2\pi} \d \dbar  \log (e^{{\varphi}_\varepsilon}+\delta e^{\psi})\right)^r \wedge \Big(\frac{i}{2\pi} \d \dbar \log|z|^2\Big)^{n-1}$$
by the semi continuity of Monge-Amp\`ere operators with respect to decreasing sequences.
By construction, the Lolong number of $\log (e^{{\varphi}_\varepsilon}+\delta e^{\psi})$ at $w_0$ is $\min(\gamma,1)-\varepsilon'$.
Thus we have on a small ball $B(w_0,\eta_\delta) \subset \pi^{-1}(B(0,\eta))$,
$$ \int_{|z| \leq \eta} \pi_*\left(\frac{i}{2\pi} \Theta(\cO_{\P(E)(1)}, h_\infty)+\frac{i}{2\pi} \d \dbar  \log (e^{{\varphi}_\varepsilon}+\delta e^{\psi})\right)^r \wedge (\frac{i}{2\pi} \d \dbar \log|z|^2)^{n-1}$$ 
$$\geq \int_{|w-w_0| \leq \eta_\delta}  (\frac{i}{2\pi} \Theta(\cO_{\P(E)(1)}, h_\infty)+\frac{i}{2\pi} \d \dbar  \log (e^{{\varphi}_\varepsilon}+\delta e^{\psi}))^r \wedge (\frac{i}{2\pi} \d \dbar \log|z|^2)^{n-1}\geq (\min(\gamma,1)-\varepsilon')^r.$$
Taking $\eta \to 0$ and $\varepsilon' \to 0$ gives the Lelong number estimate.

We have proven in Corollary 5 that $T_\varepsilon \geq -r \varepsilon \omega$, and $T_\varepsilon+r \varepsilon\omega$ is in the class $c_1(\det(E))+r \varepsilon \{ \omega \}$. By weak compactness, there exists a convergent subsequence $T_{\varepsilon_\nu}$ with limit $T$ in the class $c_1(\det(E))$.
Since $T \geq 0$ and $c_1(\det(E))=0$, the only possibility is that $T=0$.

Now, we recall the following version of the regularization theorem given in \cite{Dem82}: {\it let $T=\theta+i \d \dbar \varphi$ be a closed $(1,1)$-current, where $\theta$ is a smooth form. Suppose that a smooth $(1,1)$-form $\gamma$ is given such that $T \geq \gamma$.
Then there exists a decreasing sequence of smooth functions $\varphi_k$ converging to $\varphi$ such that, if we set $T_k:=\theta+i \d \dbar \varphi_k$, we have

\begin{itemize}
\item[$(1)$] $T_k \to T$ weakly,
\item[$(2)$] $T_k \geq \gamma-C \lambda_k \omega$, where $C>0$ is a constant depending on $(X, \omega)$ only, and $\lambda_k$ is a decreasing sequence of continuous functions such that $\lambda_k(x) \to \nu(T, x)$ for all $x \in X$.
\end{itemize}}

\noindent
By Corollary 6 below, we get
$$\lim_{\varepsilon \to 0} \sup_X \nu(T_\varepsilon, x)=0,$$
thus
$$\lim_{\varepsilon \to 0} \sup_{\P(E)} \nu(\varphi_\varepsilon, w)=0$$
thanks to the above Lelong number estimate. By the regularization theorem just recalled, there exists $\tilde{\varphi}_\varepsilon \in C^{\infty}(\P(E))$ such that
$$\frac{i}{2\pi} \Theta(\cO_{\P(E)(1)}, h_\infty)+\frac{i}{2\pi} \d \dbar \tilde{\varphi}_\varepsilon \geq -2 \varepsilon \tilde{\omega}$$
where $\tilde{\omega}$ is some K\"ahler form on $\P(E)$.
In other words, the line bundle $\cO_{\P(E)(1)}$ is nef.
\end{proof}
\begin{mylem}
Let $X$ be a compact complex manifold. 
Let $T_\delta$ $(\delta >0)$ be a sequence of closed positive $(k,k)$-currents.
Assume that $T_\delta \to 0$ weakly as $\delta \to 0$.
Then 
$$\lim_{\delta \to 0}\sup_X \nu(T_\delta,x)=0.$$
\end{mylem}
\begin{proof}
Since $X$ is compact, we can cover $X$ by finite coordinate open charts $V_i (\subset U_i \subset \tilde{U_i})$ such that $V_i$ is relatively compact in $U_i$ and $U_i$ is relatively compact in $\tilde{U_i}$. 
Thus we reduce the proof to the case of coordinate chart $V_i$.

Let $\rho_i$ be cut off functions supported in $\tilde{U_i}$ such that $\rho_i \equiv 1$ on $U_i$ and $0 \leq \rho_i \leq 1$.
Since $T_\delta \to 0$ weakly,
there exists a uniform $C>0$ such that
$$\int_{U_i} T_\delta \wedge \left( \frac{i}{2\pi} \d \dbar |z|^2 \right)^{n-k}
\leq \int_{\tilde{U_i}} T_\delta \wedge \rho_i \left( \frac{i}{2\pi} \d \dbar |z|^2 \right)^{n-k} \leq C.$$

Define for $x\in \overline{V_i}$ and for small $r$
$$\nu(T_\delta,x,r):=r^{-2(n-k)} \int_{|z-x|<r} T_\delta \wedge \left( \frac{i}{2\pi} \d \dbar |z|^2 \right)^{n-k}.$$
Then $\nu(T_\delta,x,r)$ is an increasing function with respect to $r$ and we have that $$\nu(T_\delta,x)=\lim_{r\to 0}\nu(T_\delta,x,r).$$
For small $r >0$ such that $2r < d(V_i, \partial U_i)$, there exists a cut-off function $\theta_x$ supported in $B(x, 2r)$ such that $\theta_x \equiv 1$ on $B(x, r)$ and $0 \leq \theta_x \leq 1$.
Then we have
$$\nu(T_\delta,x,r) \leq r^{-2(n-k)} \int_{U_i} T_\delta \wedge \theta_x \left( \frac{i}{2\pi} \d \dbar |z|^2 \right)^{n-k}.$$
Since $\theta_x$ can be obtained by translation of the same function, $(\theta_x)_{x \in \overline{ V_i}}$ for small $r$ is a compact family with respect to $C^\infty$ topology. 
Thus for fixed small $r$, 
for every $x,y \in \overline{V_i}$,
$$r^{-2(n-k)} \int_{U_i} T_\delta \wedge (\theta_x-\theta_y) \left( \frac{i}{2\pi} \d \dbar |z|^2 \right)^{n-1}\leq C r^{-2(n-k)} \parallel \theta_x-\theta_y \parallel_{L^{\infty}(U_i)}.$$
Thus $r^{-2(n-k)} \int_{U_i} T_\delta \wedge \theta_x \left( \frac{i}{2\pi} \d \dbar |z|^2 \right)^{n-k} $ tends to 0 as $\delta \to 0$ uniformly with respect to $x \in \overline{V_i}$.

In particular, $\nu(T_\delta,x,r)$ tends to 0 as $\delta \to 0$ uniformly with respect to $x \in \overline{V_i}$, hence the same property holds for $\nu(T_\delta,x)$.
\end{proof}
\begin{myrem}{\rm
For a family of $(1,1)$-closed positive currents, the proof is much simpler, using the observation of Proposition 10.

Let $\gamma$ be a Gauduchon metric over $X$ (i.e.\ a smooth metric such that $i \d \dbar (\gamma^{n-1})=0$).
With the same notation as in the proof, we have for $r_0$ small enough
$$\nu(T_\delta,x, r) \geq \nu(T_\delta,x,r_0) \leq \frac{C}{r_0^{2n-2}} \int_X T_\delta \wedge \gamma^{n-1}.$$
Since the right-hand side term (which is cohomological) tends to 0 along with $\delta$, the Lelong number tends to zero locally uniformly.
Since $X$ is compact, the convergence is uniform.
}
\end{myrem}
\begin{mycor} {\it
Let $(X, \omega)$ be a compact K\"ahler manifold. 
Let $T_\delta$ $(\delta >0)$ be a sequence of closed $(1,1)$-currents such that
$$T_\delta \geq -\delta \omega$$
in the sense of currents.
Assume that $T_\delta \to 0$ weakly as $\delta \to 0$.
Then 
$$\lim_{\delta \to 0}\sup_X \nu(T_\delta,x)=0.$$}
\end{mycor}
\begin{proof}
This is a direct consequence of the previous lemma if we consider $T_\delta +\delta \omega$ instead of $T_\delta$.
\end{proof}
Now we can easily conclude our result.
\begin{mycor} {\it
Let $\cF$ be a strongly psef reflexive sheaf over a compact K\"ahler manifold $(X, \omega)$ with $c_1(\cF)=0$. 
 Then $\cF$ is a nef $($and numerically flat$\,)$ vector bundle.}
\end{mycor}
\begin{proof}
By our assumption, there exists a modification such that the pull back of $\cF$ modulo torsion is a strongly psef vector bundle with vanishing first Chern class by Lemma 8.
By Theorem 4, this vector bundle is in fact nef.
Thus by Proposition 9, we conclude the corollary.
\end{proof}
As a geometric application, we obtain the following generalisation of Theorem 7.7 in \cite{BDPP}.
\begin{mycor} {\it
For a compact K\"ahler manifold if $c_1(X) = 0$ and $T_X$ is strongly psef, then a finite \'etale cover of $X$
is a torus.
In particular, an irreducible symplectic, or Calabi-Yau manifold does not have strongly psef tangenet bundle or cotangent bundle.}
\end{mycor}
\begin{proof}
By the Beauville-Bogomolov theorem, up to a finite \'etale cover $\pi: \tilde{X} \to X$, $\tilde{X}$ is a product of $\prod T_i \times \prod S_j \times \prod Y_k$ where $T_i$ are complex tori, $S_j$ are Calabi-Yau manifolds and $Y_k$ are irreducible symplectic manifolds.
Since the tangent bundle of $\tilde{X}$ is numerical flat under the assumption and by Theorem 3, the tangent bundle of all the components in the direct sum is numerical flat.
In particular, all the components have vanishing second Chern class by Corollary 1.19 of \cite{DPS94}. 
(A stronger result in the projective and singular setting can be found in Theorem 1.8 of \cite{HP}.)
By representation theory, the tangent bundle of the Calabi-Yau or irreducible symplectic components is stable.
Thus we have the equality case in the Bogomolov inequality which implies that the tangent bundle of the Calabi-Yau or irreducible symplectic components is projectively flat.
Since the first Chern class of the Calabi-Yau or irreducible symplectic components vanishes,
the tangent bundle is in fact unitary flat.
In particular, the restricted holonomy groups of the Calabi-Yau or irreducible symplectic components is trivial.
In other words, there are only the complex tori components.
\end{proof}
A stronger result in the projective singular setting can be found in Theorem 1.6 of \cite{HP}.
Instead of proving non strong pseudoeffectivity, they prove non weak pseudoeffectivity.

Inspired by the work of \cite{LOY20}, we can slightly generalise the Corollary 7.
\begin{mylem} {\rm(analogue of Lemma 11.13 in  \cite{Voi})}

Let $X$ be a complex manifold (not necessary compact) and $Y$ be a closed submanifold of codimension at least $r+1$. 
Then the restriction map
$$H^l(X, \R) \to H^l(X \setminus Y,\R)$$
is an isomorphism for $l \leq 2r$.
\end{mylem}
\begin{proof}
We have the long exact sequence of relative cohomology
$$\cdots H^{l}(X, X \setminus Y, \R) \to H^l(X, \R) \to H^l(X \setminus Y, \R) \to H^{l+1}(X, X\setminus Y, \R)\cdots .$$
On the other hand, we have by the excision lemma
that for $U$ a tubular neighborhood of $Y$
$$H^l(X, X \setminus Y, \R) \cong H^l(U, U \setminus Y, \R).$$
By Thom isomorphism theorem, we have
$$H^{l-2r}( Y, \R) \cong H^l(U, U \setminus Y, \R).$$
We note that $X$ as a complex manifold is orientable, hence so is $U$.
Therefore the Thom class with $\Z-$coefficients exists by Theorem 4.D.10. in \cite{Hat}.
The natural inclusion $\Z \to \R$ sends the Thom class with $\Z-$coefficients to the Thom class with $\R-$coefficients.
Thus we have the Thom isomorphism by the Corollary 4.D.9 in \cite{Hat}.
It follows that for $j < \mathrm{codim} \; Y$,
$H^j(X,X \setminus Y, \R)=0$.
This finishes the proof of the lemma using the exact sequence.
\end{proof}
\begin{mylem}
Let $X$ be a complex manifold (not necessary compact) and $Y$ be a closed analytic subset of codimension at least $r+1$. 
Then the restriction map
$$H^l(X, \R) \to H^l(X \setminus Y,\R)$$
is an isomorphism for $l \leq 2r$.
\end{mylem} 
\begin{proof}
It is a direct consequence of Lemmas 5 and 16.
\end{proof}
\begin{mylem}(analogue of Lemma 4.5 \cite{LOY20})
Let $(X,\omega)$ be a compact K\"ahler manifold of dimension $n > 2$, and let $\cF$ be a reflexive coherent
sheaf of rank $r > 2$ on $X$. Then, for any positive integer $m > 2$, we have
$$c_2(S^{[m]}\cF) = A c^2_1(\cF) + B c_2(\cF),$$
where $A$ and $B$ are non-zero rational numbers depending only on $m$ and $r$, and satisfy the relation
$$A+\frac{r-1}{r} B -\frac{(R-1)R m^2}{2r^2}=0$$
where $R ={ {r+m-1}\choose
{r} }$ is the rank of $S^{[m]}\cF$.
\end{mylem}
\begin{proof}
The proof is almost identical to Lemma 4.5 in \cite{LOY20}.
The only difference is the abandonment of the use of the auxiliary ample line bundle.
For this reason, we only sketch the proof.
We have trivially the form of the equality over the open set where the sheaf is locally free.
By Lemma 17, the same equality should hold on $X$.
By splitting principle, it is enough to prove the formula for $\cF=\oplus^r L$ where $L$ is a hermitian (complex) line bundle (not necessarily holomorphic).

In this case, $\cF$ is a polystable and projectively flat vector bundle, thus we have the equality case in the Bogomolov-L\"ubke inequality,  
$$(c_2(\cF) -\frac{r-1}{r} c_1(\cF)^2) \cdot \omega^{n-2} =0.$$ 
Develop $c_2(S^m \cF \otimes L^{*\otimes m})=0$ in terms of $c_1(L)$.
Combining with the above equality, we have 
$$(A+\frac{r-1}{r} B -\frac{(R-1)R m^2}{2r^2}) c_1(L)^2 \cdot \omega^{n-2}=0.$$
It suffices to show that there exists a hermitian (complex) line bundle such that $c_1(L)^q \cdot \omega^{n-q} \neq 0$ for any $q$.
Recall that Théorème 4.3 of \cite{Lae02} proved using Kronecker lemma that for any closed real $(1,1)-$form $\alpha$ on a compact complex manifold, for infinite $k$, $k \alpha$ can be approximated in $C^{\infty}$ norm by the curvature of some hermitian (complex) line bundle $L_k$ with respect to some hermitian connection.
In particular, for such $k$ large enough, $c_1(L_k)^q \cdot \omega^{n-q} \neq 0$ for any $q$. 

By choosing $\cF$ as some combination of $L_k$, $L_k^*$ and $\cO_X$, it can be shown that $A,B$ are non-zero.
\end{proof}
For the convenience of the reader, we give here the proof of the compact K\"ahler version of Proposition 4.6 Chap. IV of \cite{Nak}.
\begin{myprop}{\it
Let $(X,\omega)$ be a compact K\"ahler manifold and $\cF$ be an $\omega$-semi-stable reflexive sheaf with 
$$(c_2(\cF) -\frac{r-1}{r} c_1(\cF)^2) \cdot \omega^{n-2} =0.$$ 
Then $\cF$ is locally free.}
\end{myprop}
\begin{proof}
We shall prove by induction on the rank of $\cF$.
If $\cF$ is polystable, it is direct consequence of corollary 3 of \cite{BS}.
We may assume $\cF$ is not
polystable. Then there is an exact sequence
$$0 \to \cS \to \cF \to \cQ \to 0,$$
where $\cS$ and $\cQ$ are non-zero torsion-free sheaves satisfying the relation of slope
$\mu(\cS)=\mu(\cF)=\mu(\cQ)$.
The sheaves $\cS$ and $\cQ^{**}$ are semi-stable.

Recall the formula (II.9) Chap. II \cite{Nak}
$$\hat{\Delta}_2 (\cF)=\hat{\Delta}_2 (\cS)+\hat{\Delta}_2 (\cQ)-\frac{\mathrm{rank \; } \cS \cdot \mathrm{rank \; } \cQ}{\mathrm{rank \; } \cF} (\mu(\cS)-\mu(\cQ))^2$$
where
$\hat{\Delta}_2 (\cF) : = (c_2 (\cF) - \frac{r-1}{2r} c_1 (\cF)^2 )\cdot \omega^{n-2}.$
The Bogomolov inequality gives $\hat{\Delta}_2 (\cS) \geq 0$
and $\hat{\Delta}_2 (\cQ^{**}) \geq 0$.
On the other hand, $c_2(\cQ^{**}/\cQ)$ is represented by an effective cycle supported in the support of the torsion sheaf $\cQ^{**}/\cQ$.
Thus we have that
$$\hat{\Delta}_2 (\cQ^{**})=\hat{\Delta}_2 (\cQ^{})=\hat{\Delta}_2 (\cS^{})=0.$$
By induction, $\cS$ and $\cQ^{**}$ are locally free which by Lemma 4 defines an extension of vector bundles over $X$.
Since $\cF$ coincides with a vector bundle outside an analytic set of codimension at least 3, $\cF$ is locally free.
\end{proof}
As consequence of the lemma and the proposition, we have the following generalisation of Corollary 7.
\begin{mycor}(analogue of Theorem 1.6 \cite{LOY20}) {\it

Let $(X,\omega)$ be a compact K\"ahler manifold of dimension $n$, and let $\cF$ be a reflexive coherent
sheaf on $X$. 
Assume there exists a line bundle $L$ and $m >0$ such that $S^{[m]}\cF \otimes L$ is strongly psef with $c_1(S^{[m]}\cF \otimes L)=0$.
Then $\cF$ is a vector bundle such that $\cF \langle -\frac{1}{m} L \rangle$ is a $\Q-$twisted nef vector bundle.

In particular, let $E$ be a vector bundle of rank $r$ such that $E\langle -\frac{1}{r}\det(E)\rangle$ is $\Q-$twisted strongly psef vector bundle,
then $E\langle -\frac{1}{r}\det(E)\rangle$ is $\Q-$twisted nef vector bundle.}
\end{mycor} 
\begin{proof}
By Corollary 7, $S^{[m]}\cF \otimes L$ is a numerically flat vector bundle.
In particular, $c_2(S^{[m]}\cF \otimes L)=0$ and $S^{[m]}\cF \otimes L$ is semistable.
In fact, $S^{[m]}\cF \otimes L$ admits a filtration of vector bundles
$$0=\cE_0 \subset \cE_1 \subset \cdots \subset \cE_p =S^{[m]}\cF \otimes L$$
such that for each $i$, $\cE_i /\cE_{i-1}$ is flat and polystable.
For any subsheaf $\cS$ of $S^{[m]}\cF \otimes L$, let $i_0:=\max \{i, \cE_i \subset \cS\}$.
Then if $\cS=\cE_{i_0}$, $\mu(\cS)=0$.
Otherwise $\cF/\cE_{i_0}$ is a non zero subsheaf of $\cE_{i_0+1}/\cE_{i_0}$, thus $\mu(\cS)=\mu(\cS/\cE_{i_0})+\mu(\cE_{i_0})\leq \mu(\cE_{i_0+1}/\cE_{i_0})=0$.
  
By the above lemma, direct calculations yield
$$(c_2(\cF)-\frac{r-1}{2r}c_1(\cF)^2) \cdot \omega^{n-2}=0.$$
We claim that $\cF$ is also semistable.
In fact, for any torsion free quotient sheaf $\cQ$ of $\cF$, we have generic surjective morphism
$$\alpha: S^{[m]}\cF \otimes L \to S^{[m]}\cQ \otimes L.$$
The image of $\alpha$ coincide with $S^{[m]}\cQ \otimes L$ outside an analytic set of codimension at least 2,
thus these two sheaves have the same slope.
The inequality $\mu(S^{[m]}\cF \otimes L) \leq \mu(S^{[m]}\cQ \otimes L)$ implies that 
$\mu(\cF) \leq \mu(\cQ)$.
In fact, $S^{[m]}\cF$ and $\cF$ are locally free outside a closed analytic set $A$ of codimension at least 2.
Since $H^2(X, \C) \cong H^2(X \setminus A, \cF)$,
$$c_1(S^{[m]}\cF)=\frac{1}{r} { {m+r-1}\choose{m}} c_1(\cF)$$
from the corresponding formula by restriction on $X\setminus A$ on which the coherent sheaves are locally free.
Here $r$ is the rank of $\cF$.
We have of course similar formula for $\cQ$.

For the general case, it is a direct consequence of the above proposition.
Thus we can prove the following equivalent conclusion.
$\cF$ is locally free and there is a filtration of vector subbundles
$$0 = \cF_0 \subset  \cF_1 \subset \cdots \subset \cF_p = \cF$$
such that $\cF_i/\cF_{i+1}$ are projectively flat vector bundles and $\mu(\cF_i/\cF_{i+1})=\mu(\cF)$ for any $i$.
\end{proof}
  
\end{document}